\documentclass[10pt,a4paper,reqno]{amsart}
\usepackage{a4wide}
\usepackage[utf8x]{inputenc}
\usepackage{ucs}
\usepackage{amsmath,amssymb,paralist,amsthm,array}
\usepackage{hyperref}

\newtheorem{definition}{Definition}[section]
\newtheorem{lemma}[definition]{Lemma}
\newtheorem{theorem}{Theorem}

\newtheorem{proposition}[definition]{Proposition}

\theoremstyle{definition}

\newtheorem{remark}[definition]{Remark}

\newcommand{\ext}{{\mathrm{ext}}}

\newcommand{\bp}{\mathbf{p}}

\title[An entropy formula for a non-self-affine measure]{An entropy formula for a non-self-affine measure with application to Weierstrass-type functions}

\author{Atsuya Otani}
\email{otani@math.fau.de}
\address{Department Mathematik, Universität Erlangen-Nürnberg, Cauerstr. 11, 91058 Erlangen, Germany}

\thanks{This work is funded by DFG grant Ke 514/8-1. This is also supported by the DFG Scientific Network 'Skew Product Dynamics and
Multifractal Analysis'. This is part of my PhD project, supervised by Prof. Dr. G. Keller.}

\date{\today}

\begin{document}

\begin{abstract}
Let $ \tau : [0,1] \rightarrow [0,1] $ be a piecewise expanding map with full branches.
Given $ \lambda : [0,1] \rightarrow (0,1) $ and $ g : [0,1] \rightarrow \mathbb{R} $ satisfying $ \tau ' \lambda > 1 $, we study the Weierstrass-type function
\[
	\sum _{n=0} ^\infty \lambda ^n (x) \, g ( \tau ^n (x) ) ,
\]
where $ \lambda ^n (x) := \lambda(x) \lambda (\tau (x)) \cdots \lambda ( \tau ^{n-1} (x)) $.
Under certain conditions, Bedford proved in \cite{Bedford89} that the box counting dimension of its graph is given as the unique zero of the topological pressure function
\[
	s \mapsto P ( (1-s) \log \tau ' + \log \lambda ) .
\]
We give a sufficient condition under which the Hausdorff dimension also coincides with this value.
We adopt a dynamical system theoretic approach which is used among others in \cite{Ledrappier92} and \cite{Baranski14} to investigate special cases including the classical Weierstrass functions.
For this purpose we prove a new Ledrappier-Young entropy formula, which is a conditional version of Pesin's formula, for non-invertible dynamical systems.
Our formula holds for all lifted Gibbs measures on the graph of the above function, which are generally not self-affine.
\end{abstract}

\maketitle

\section{Introduction}
\subsection{Motivation and preceding results}

Let $ ( I _i ) _{i=0} ^{\ell - 1} $ be a partition of $ [0,1]$ into intervals with $ I _i ^\circ $ and $ \overline{I _i} $ being the interiors and closures, respectively.
Then we consider the map $ \tau : [0,1] \rightarrow [0,1] $ such that the restrictions $ \tau _{| I _i ^\circ} : I _i ^\circ \rightarrow ( 0 , 1 ) $ are $ C ^{1 + } $-diffeomorphisms with  $ \inf (\tau _{| I _i ^\circ})' > 1 $ for $ i \in \{0, \ldots , \ell -1 \} $,
where $ C ^{1 + } $ means that a Hölder continuous derivative exists, without specifying the Hölder exponent.
In addition, let $ \lambda : [0,1] \rightarrow (0,1) $ and $ g : [0,1] \rightarrow \mathbb{R} $ be maps which are $ C ^{1+} $ on each $ I _i ^\circ$ satisfying $ \lambda \tau' > 1 $.

We study the Weierstrass-type function
\[
	W _{\tau , \lambda} ( x ) := \sum _{n=0} ^\infty \lambda ^n (x) g ( \tau ^n (x) )
\]
from a dimension theoretic point of view, where  $ \lambda ^n ( x ) := \lambda (x) \lambda ( \tau (x) ) \cdots \lambda (\tau ^{n-1} (x)) $.

We recall a pair of selected results from the literature.
To be precise, we assume during the citation that $ \tau , \lambda $ and $ g $ can be extended to $ C ^{1+} $-functions on $ \mathbb{R}/ \mathbb{Z} $. 
Let $ s ( \tau , \lambda ) \in \mathbb{R} $ be the unique zero of the Bowen equation
\begin{equation}
	P ( (1-s) \log \tau ' + \log \lambda ) = 0 	\label{eq:Pressure}
\end{equation}
and $ \nu _{\tau , \lambda} \in \mathcal{P} ([0,1]) $ the equilibrium measure, where $ P $ denotes the topological pressure.
Moreover, let $ h _{ \nu _{\tau, \lambda} } $ denote the associated KS-entropy.

First, the box counting dimension of the graph of $ W _{\tau ,\lambda}  $ is proved by T. Bedforld in \cite{Bedford89} to be
\begin{equation}
	s ( \tau, \lambda ) = 1 + \frac{h _{ \nu _{\tau, \lambda} } + \int \log \lambda \, d \nu _{\tau , \lambda} }{\int \log \tau ' \, d \nu _{\tau, \lambda}} ,	\label{eq:s_formula}
\end{equation}
whenever the function $ W _{\tau, \lambda} $ is not differentiable.
Note that he generalised a result of J. Kaplanan, J. Mallet-Pareta and J. Yorkea in \cite{Kaplan84} for a significantly larger class.

Second, A. Moss and C. P. Walkden constructed in \cite{Moss12} a randomized version of the function for a large class of $ g $ including trigonometric functions, in a similar fashion to B. Hunt's work \cite{Hunt98}, for which the Hausdorff and the box counting dimension of the graph coincide with the same $ s ( \tau , \lambda ) $ almost surely.

These results lead to the conjecture that both dimensions may always be identical.
Nevertheless, this general conjecture is denied by a counterexample which M. Urbanski and F. Przytycki constructed in \cite{Urbanski89} for a so-colled limit Rademacher function related to a Pisot number, making use of some combinatoric properties of that algebraic number.

Our Theorems \ref{thm:mu_dim}, \ref{thm:main_theorem}, \ref{thm:example2} give a partial positive answer to this problem by providing sufficient conditions.
We employ a dynamical system theoretic approach making use of the new dimension and entropy formula stated in Theorem \ref{thm:entropy_formula}.

\begin{remark}
Similar approaches were taken among others in \cite{Kaplan84}, \cite{Urbanski89} and \cite{Ledrappier92}.
More recently, K. Barański, B. Bárány and J. Romanowska studied in \cite{Baranski14} the classical Weierstrass function, which is given by choosing $ \tau (x) = \ell x \bmod 1 $, $ g (x) = \cos ( 2 \pi x ) $ and $ \lambda $ to be constant on each $ I _i $.
Combining the results of \cite{Ledrappier92} and \cite{Tsujii01}, they discovered an explicit parameter region, for which the Hausdorff and the box counting dimension of the graph coincide.\footnote{Meanwhile a result which covers all parameters was published, see \cite{Shen15}. He modified the sufficient condition provided in \cite{Tsujii01}.}
Our Theorem \ref{thm:example2} extends their results.
\end{remark}

\begin{remark}
We also mention that an alternative shorter proof for the work \cite{Baranski14} is given by G. Keller in \cite{Keller14}, which can also be modified for our Theorems \ref{thm:example}, \ref{thm:example2} by using large deviation results.
This is, however, not the content of this note.
\end{remark}

\subsection{Main results}
We present our main results, some of whose proofs can be found in Sections \ref{sec:LY}, \ref{sec:L} and \ref{sec:Tsujii}.
Given $ \xi \in [0,1] $, the function $ q _\xi : [ 0 , 1 ] \rightarrow \mathbb{R} $ is defined by $ q _\xi ( x ) := \pi _\xi ^{ss} ( x , W(x) ) $, where $ \pi ^{ss} _\xi : [ 0,1] \times \mathbb{R} \rightarrow \mathbb{R} $ is the projection on the the hyperplane $ \{0\} \times \mathbb{R} $ along the strong stable fibres with respect to $ \xi $ which will be introduced in Subsection \ref{subsec:strong_stable} after a suitable dynamical system is constructed.
The precise definitions are listed at the beginning of Section \ref{sec:LY}.
Then we can disintegrate\footnote{See Proposition \ref{prop:decomposition}.} the lift $ \mu \in \mathcal{P} ( [0,1] \times \mathbb{R} ) $ of a $ \tau $-invariant $ \nu \in \mathcal{P} ( [0,1] ) $ on the graph of $ W _{\tau, \lambda} $ as
\[
	\mu := (\mathrm{Id} , W _{\tau , \lambda}  ) ^\ast \nu  = \int \mu _{(\xi , y )} \, d \left( \nu \circ q _\xi ^{-1} \right) (y)
\]
for each $ \xi \in [0,1] $, where $ \mu _{(\xi , y )} \in \mathcal{P} ( [0,1] \times \mathbb{R} ) $ is to be interpreted as the conditional measure on the strong stable fibre through $ ( \xi , 0 , y ) $.
Moreover, $ \nu $ can be extended to $ \nu ^\ext \in \mathcal{P} ([0,1] ^2) $ naturally as \eqref{eq:ext_def}.

\begin{theorem}[Dimension and entropy formula]	\label{thm:entropy_formula}
Suppose that $ \nu \in \mathcal{P} ( [0,1] ) $ is a Gibbs measure.
Then $ \mu $, $ \mu _{(\xi, q _\xi (x))} $ and $ \nu \circ q _\xi ^{-1} $ are exact dimensional and the dimensions are constant for $ \nu  ^\ext  $-a.a. $ ( \xi, x ) \in [0,1] ^2 $, satisfying $ \dim _H \left( \mu \right) 	 = 	 \dim _H \left(  \mu _{ (\xi, q _\xi (x))} \right) +  \dim _H \left(  \nu  \circ q _\xi ^{-1} \right) $ and
\[
	h _\nu 	 =    \dim _H \left(  \mu _{ (\xi, q _\xi (x))} \right) \cdot	 \int \log \tau ' \, d \nu - \dim _H \left(  \nu \circ q _\xi ^{-1} \right) \cdot \int \log \lambda \, d \nu .
\]
\end{theorem}

\begin{remark}
The dimension and entropy formula were originally introduced by F. Ledrappier and L.-S. Young in \cite{YL285} for diffeomorphisms on a compact Riemannian manifold.
Then Ledrappier established in \cite{Ledrappier92} their counterparts for non-invertible models through an invertible extension.
He sketched the proof for the case $ \tau (x) = 2 x \bmod 1 $ and Lebesgue measure, which we extend in Section \ref{sec:LY}.
\end{remark}
\begin{remark}
Bárány proved in \cite{Barani14} the corresponding formulas for self-affine measures.
Although the measures we investigate are not self-affine, some of his results are fairly similar to ours.
Instead of the self-affinity, we make use of not only the Gibbs property of the marginal measures but also the structure of the function $ W _{\tau, \lambda} $.
\end{remark}

With additional information, the Hausdorff dimension of $ \mu $ can be derived from the above formulas.
Here the random variable $ \Theta ( \cdot , x ) : [0,1] \rightarrow \mathbb{R} $, which will be defined in \eqref{eq:Theta_Def}, describes the 'random' strong stable direction at the point $ ( x , W _{\tau, \lambda}(x) ) $ in sense of \eqref{eq:ss_eigenvalue} together with \eqref{eq:Theta_Def}.
Note that the marginal measure $ \nu ^- := \nu ^\ext ( \, \cdot \, \times [0,1]  ) $ satisfies the relation \eqref{eq:disintegral_vert}.

\begin{theorem}	\label{thm:mu_dim}
Suppose that $ \nu \in \mathcal{P} ( [0,1] ) $ is a Gibbs measure and let $ \mu $ be its lift on the graph of $ W _{\tau, \lambda} $.
If the distribution of $ \Theta ( \cdot , x ) $ under $ \nu ^- $ has Hausdorff dimension $ 1 $ for $ \nu $-a.a. $x \in [0,1] $, then we have
\[
	\dim _H ( \mu )	=  \min \left\{  1 + \frac{ h _\nu +  \int \log \lambda \, d \nu }{\int \log \tau ' \, d \nu} , \frac{h _\nu}{- \int \log \lambda \, d \nu} \right\} ,
\]
where the first value is taken if and only if $ \dim _H ( \mu ) \geqslant 1 $.

Moreover, we have $ \dim _H ( \mu ) \geqslant 1 $ if and only if $ h _\nu \geqslant - \int \log \lambda \, d \nu $.
\end{theorem}

\begin{proof}
By Theorem \ref{thm:entropy_formula} together with Lemma \ref{lem:L} we have
\[
\dim _H ( \mu ) =
	\begin{cases}
	  1 + \frac{ h _\nu +  \int \log \lambda \, d \nu }{\int \log \tau ' \, d \nu}	& \mbox{if }  \dim _H ( \mu ) \geqslant 1 \\
	 \frac{h _\nu}{- \int \log \lambda \, d \nu}	 & \mbox{if } \dim _H ( \mu ) < 1 
	\end{cases} .
\]
Observe that in case $ \dim _H ( \mu ) \geqslant 1 $ we have $ h _\nu + \int \log \lambda \, d \nu \geqslant 0 $, while in the other case we have $ h _\nu + \int \log \lambda \, d \nu < 0 $.
Thus we can finish the proof, rewriting
\[
	\min \left\{  1 + \frac{ h _\nu +  \int \log \lambda \, d \nu }{\int \log \tau ' \, d \nu} , \frac{h _\nu}{- \int \log \lambda \, d \nu} \right\} = \frac{h _\nu}{\int \log \tau ' \, d \nu} + \min \left\{1 , \frac{h _\nu}{- \int \log \lambda \, d \nu} \right\} \cdot \left( 1 + \frac{\int \log \lambda \, d \nu}{\int \log \tau ' \, d \nu}\right) .
\]
\end{proof}

\begin{remark}	\label{rem:degenerate}
Without the assumption on $ \Theta $, generally, the statement of the above theorem is by no means true since the value on the right hand side does not depend on $ g $, although the one on the other side heavily does.
However, we mention that our assumption is much stronger than the non-degenerate case in sense of \cite{Bedford89}, i.e. $ W $ is not differentiable.\footnote{In fact, $ W $ is either $ C ^1 $ or nowhere continuous.}
For instance, in case $ g $ is non-zero constant and $ \lambda $ is piecewise\footnote{Our 'piecewise' is always related to $ ( I _i  ^\circ ) _{i= 0} ^{\ell -1} $.} constant (but not trivial), $ W _{\tau, \lambda}$ is not differentiable despite $ \Theta \equiv 0 $.
In particular, the box counting dimension of the graph is still $ s ( \tau ,  \lambda ) $ according to \cite{Bedford89}.
\end{remark}

\begin{remark}
In the above proof we applied Lemma \ref{lem:L}.
This is a key lemma, whose prototype appears in Ledrappier's note \cite{Ledrappier92} and is also referred to in \cite{Baranski14}.
Because of its importance we will deal with it in Section \ref{sec:L} separately.
There we will give a complete proof for a slightly more general situation, which is based on some ideas in \cite{Ledrappier92} and \cite{YL88}.
\end{remark}

The next result is a consequence of the preceding theorem.
Recall that $ s ( \tau , \lambda ) \in \mathbb{R} $ and $ \nu _{\tau, \lambda} \in \mathcal{P} ([0,1] ) $ are defined in \eqref{eq:Pressure}.

\begin{theorem}	\label{thm:main_theorem}
If the distribution of $ \Theta ( \cdot , x ) $ under $ \nu _{\tau, \lambda} ^- $ has Hausdorff dimension $ 1 $ for $ \nu _{\tau, \lambda} $-a.a. $ x \in [0,1]  $, then $ \dim _H ( \mathrm{graph} ( W _{\tau , \lambda} ) ) = \dim _B ( \mathrm{graph} ( W _{\tau ,\lambda} ) ) = s ( \tau , \lambda ) $.
\end{theorem}

\begin{proof}
As mentioned, $ \dim _H ( \mathrm{graph} ( W _{\tau, \lambda} ) ) \leqslant \dim _B ( \mathrm{graph} ( W _{\tau ,\lambda} ) ) = s ( \tau , \lambda ) $ is proved in \cite{Bedford89}.\footnote{To be precise, in \cite{Bedford89} all functions are assumed to be $ C ^{1+} $ on $ \mathbb{T} ^1 $, instead of the piecewise regularity. However, the difference is not essential in that proof. Alternatively, it is not difficult to derive the upper bound by means of the Gibbs property of $ \nu _{\tau , \lambda} $, constructing Moran covers.}

\footnote{Note also that a direct proof for $ \dim _H ( \mathrm{graph} ( W _{\tau, \lambda} ) ) \leqslant s ( \tau , \lambda ) $ is given as proof of \cite[Proposition 2.2]{Moss12}.}

In order to prove $ \dim _H ( \mathrm{graph} ( W _{\tau, \lambda} ) ) \geqslant s ( \tau , \lambda ) $, it suffices to show $ \dim _H ( \mu _{\tau , \lambda} )  \geqslant s ( \tau , \lambda ) $ since $ \dim _H ( \mathrm{graph} ( W _{\tau, \lambda} ) ) \geqslant \dim _H ( \mu _{\tau , \lambda} ) $ by Lemma \ref{lem:mass_d_p}.
Observe that $ h _{\nu _{\tau , \lambda}} \geqslant - \int \log \lambda \, d \nu _{\tau , \lambda} $ follows from \eqref{eq:s_formula} as $ s ( \tau , \lambda ) \geqslant 1 $.
Concerning the fact that the equilibrium measure $ \nu _{\tau, \lambda} $ is a Gibbs measure, the claim follows from Theorem \ref{thm:mu_dim}.
\end{proof}

Although the last theorem is a powerful tool to determine the Hausdorff dimension of the graph, the verification of the assumption on $ \Theta ( \cdot, x ) $ is in many cases quite challenging.
Finally, we give two examples.
The details will be discussed in Section \ref{sec:Tsujii}.

\begin{theorem}	\label{thm:example}
Let $ \gamma _0 , \gamma _1 \in (0,1) $ and $ a _0, a _1 \in \mathbb{R} $ satisfy $ \gamma _0 a _0 \neq \gamma _1 a _ 1 $.
Suppose $ l = 2 $ and that $ \tau $ and $ g $ are piecewise linear.
Furthermore, suppose that $ \lambda (x) := \gamma _i \, | I _i | $ and $ g ' (x) = a _i $ for $ x \in I _i $, $ i = 0 , 1 $.
Then there is a set $ E \subset \mathbb{R} $ of Hausdorff dimension $ 0 $ such that $ \dim _H ( \mathrm{graph} ( W _{\tau, t \lambda } ) ) = \dim _B ( \mathrm{graph} ( W _{\tau, t \lambda } ) ) = s ( \tau , t \lambda ) $ for all $ t \in \left( \max\{ \gamma _0 , \gamma _1 \} , \min  \{ \frac{\gamma _0}{\sqrt{| I _0 |}} , \frac{\gamma _1}{\sqrt{| I _1 |}}  \} \right] \setminus E $.
\end{theorem}

\begin{remark}
In Theorem \ref{thm:example}, if additionally $ | I _0 | = | I _1 | $ and $ \gamma _0 = \gamma _1 $ are satisfied, and if $ g (x) = \mathrm{dist} ( x , \mathbb{Z} ) $, then $ W _{\tau , t \lambda} $ are Takagi functions which are studied in \cite{Ledrappier92}.
\end{remark}

\begin{theorem}	\label{thm:example2}
Suppose $ \ell \geqslant 2 $ and that $ \tau $ is a piecewise linear.
Furthermore, suppose $ g (x) := \cos ( 2 \pi x ) $ and $ \lambda  := ( \tau ' ) ^{-\theta} $ for a $ \theta \in ( 0 , 1 ) $.
If both
\begin{equation}
	\frac{| I _i|}{| I _j |} < | I _j | ^{\frac{- \theta}{2 - \theta}}  \quad ( \forall i \neq j ) \quad \mbox{and}	\label{eq:cond1}
\end{equation}
\begin{equation}
 G \left( ( \min _i | I _i | ) ^{1 - \theta} ,  ( \max _i | I _i | ) ^{1 - \theta}  \right)  + G  \left( ( \min _i | I _i | ) ^{2 - \theta} ,   ( \max _i | I _i | ) ^{2 - \theta}  \right) 
< \delta _0 ,	\label{eq:cond2}
\end{equation}
are satisfied, then $ \dim _H ( \mathrm{graph} ( W _{\tau, \lambda} ) ) = \dim _B ( \mathrm{graph} ( W _{\tau,\lambda} ) ) = 2 - \theta $, where
\begin{equation}
	\delta _0 := \inf _{i \neq j} \inf _x \sin ^2 \left( \pi ( \rho _i (x) - \rho _j (x) )  \right) , \quad  \mbox{and} \quad  G ( s , t ) := \left(  s ^{-1} \left( \frac{t ^2}{1 - t} + \frac{t -s}{2}\right) \right) ^2	\label{eq:G}
\end{equation}
for $ s , t \in ( 0 , 1 ) $.
\end{theorem}

\begin{remark}
Let us consider the special case of Theorem \ref{thm:example2}, where $ \tau ( x ) := \ell x \bmod 1  $.
Furthermore, let $ \lambda \in ( 1 / \ell , 1) $ be a constant function.
Then the condition \eqref{eq:cond1} is trivial, while the other condition \eqref{eq:cond2} is satisfied, if $ \ell \geqslant 3 $ and $ ( \tau ' ) ^{- \theta } = \ell ^{- \theta} \in ( \tilde{\lambda _\ell} , 1 ) $.
Here $ \tilde{\lambda} _\ell $ is the unique zero of
\[
	h _\ell ( \lambda ) := \frac{1}{(\ell \lambda  - 1) ^2} + \frac{1}{(\ell ^2 \lambda -  1) ^2} - \sin ^2 \left( \frac{\pi}{\ell} \right) .
\]
This is a part of \cite[Theorem A]{Baranski14} which we technically extended.
Notice that the cited theorem provides another condition for the case $ \ell = 2 $ which apparently relies on the specific choice of $ \tau $.
\end{remark}

\section{Preliminaries}	\label{sec:pre}

As long as the parameters are fixed, we simply write $ W $ instead of $ W _{\tau , \lambda} $.

Here is the general notation of this note:
\begin{itemize}
 \item Let $ \rho _i : [0,1] \rightarrow \overline{I _i} $ be the continuous extension of the inverse map of the branch $ \tau _{| I _i ^\circ} $ for each $ i \in \{ 0 , \ldots , \ell -1 \} $.
 \item Let $ \rho _{(\omega _1 , \ldots , \omega _n)} := \rho _{\omega _n} \circ \cdots \circ \rho _{\omega _1}$ for $ (\omega _1 , \ldots , \omega _n ) \in \{ 0 , \ldots , \ell -1 \} ^n $ and $ n \in \mathbb{N} $.
 \item Let $ k (x) := i $ for $ x \in I _i $.
 \item Let $ [ x ] _n := ( k (x ) , k (\tau (x) ) , \ldots , k (\tau ^{n-1} (x) ) ) \in \{ 0 , \ldots , \ell -1 \} ^n $ for $ x \in [0,1] $ and $ n \in \mathbb{N} $.
 \item Let $ I _N ( x ) $ be the monotonicity interval of $ \tau ^N $ containing $ x $ for $ x \in [0,1] $ and $ N \in \mathbb{N} $.
 \item Let $ \mathcal{S} _\tau := \{ I _N (x  ) : x \in [0,1] \mbox{ and } N \in \mathbb{N}  \} \cup \{ \emptyset \} $.
 \item For a function $ \phi $ on $[0,1] $ we write $ \phi ^n := \prod _{i=0} ^{n-1} \phi \circ \tau ^i $ and $ \phi _n := \sum _{i=0} ^{n-1} \phi  \circ \tau ^i  $.
 \item We use expressions like $ \tau ' $, when the non-differential points can be ignored.
\end{itemize}

For a metric space $ E $ let $ \mathcal{P} ( E ) $ be the set of Borel probability measures and $ \mathcal{B} ( E ) $ be the Borel algebra.
Further, for a family $ \mathcal{F} \subseteq \mathcal{B} ( E ) $ let $ \sigma ( \mathcal{F} ) $ denote the smallest $ \sigma $-algebra containing $ \mathcal{F} $.
Note that $ \mathcal{S} _\tau $ is a semiring, which generates $ \mathcal{B} ( [0,1] ) $.
Thus if two finite Borel measures on $ [0,1] $ have identical values on $ \mathcal{S} _\tau $, they must be the same measure.
Indeed, we have the following theorem.

\begin{lemma}[Corollary of the approximation theorem (cf. {\cite[Theorem 1.65(ii)]{Klenke}})]	\label{lem:approximation}
For any $ \nu \in \mathcal{P} ([0,1] ) $, $ A \in \mathcal{B} ( [0,1] ) $ and $ \varepsilon > 0 $ there are $ N \in \mathbb{N} $ and mutually disjoint $ I _1 , \ldots ,  I _N \in \mathcal{S} _\tau $ such that
\[
	\nu \left( A \bigtriangleup \bigcup _{i=1} ^N I _i \right) < \varepsilon .
\]
\end{lemma}

\subsection{The dynamical system and its inverse}

We introduce a dynamical system which characterises the graph of $ W _{\tau , \lambda} $ as a repeller.
Then we construct an invertible extension which enables us to work with an attractor.

If we consider the skew product dynamical system $ G : [0,1] \times \mathbb{R} \rightarrow [0,1] \times \mathbb{R} $ defined by $ G ( x , y ) := \left( \tau (x) , \frac{y - g(x)}{\lambda(x)} \right) $, the graph of $ W  $ can be characterised as its unique repeller.
Indeed, the graph is an invariant set, i.e. we have
\[
	G ( x , W  (x) ) = (\tau (x) , W  ( \tau (x) ))
\]
for all $ x \in [0,1] $, while all other points are driven to infinity because of the uniform expansion in the vertical direction.

Now, we construct an inverse system of this as follows.
First, we extend the basis dynamics $ \tau : [0,1] \rightarrow [0,1] $ to the second component of the inverse of the following non-linear Baker map.
We define the non-linear Baker map $ B : [0,1] ^2 \rightarrow [0,1] ^2 $ by
\[
	B ( \xi , x ) := ( \tau (\xi ) , \rho _{k(\xi)} (x) ) .
\]
Notice that our definition of the Baker map may be unusual, especially when $ \tau $ preserves Lebesgue measure.
We just adopt the simplest one.
Observe that $ B ^{-1} ( \xi , x ) = ( \rho _{k(x)} (\xi ) , \tau ( x )  ) $.
Then we consider the skew product system $ F : [0,1] ^2 \times \mathbb{R} \rightarrow [0,1] ^2 \times \mathbb{R} $ by
\[
	F ( \xi , x , y ) := \left( B ( \xi , x ) , \,  \lambda \left( \rho _{k(\xi)} (x) \right) \cdot y + g \left( \rho _{k(\xi)} ( x ) \right) \right) .
\]
Clearly, the second and third coordinates of $ F $ is the inverse to $ G $.
Moreover, it is straightforward to verify that
\begin{equation}
	F ( \xi , x , W  ( x ) ) = \left( B (\xi, x ) , W \left( \rho _{k(\xi)} (x) \right) \right)	\label{eq:W_inv_F}
\end{equation}
holds for all $ ( \xi , x ) \in [0,1] $, i.e. the graph of $ (\xi , x ) \mapsto W (x) $ is invariant under $ F $.
Moreover, this graph is the global attractor of the system due to the uniform contraction in the vertical direction.

Observe that any iterates are also skew-products, i.e. we can define for $ n \in \mathbb{N} $ the fibre map $  F _{(\xi , x )} ^n ( y ) : \mathbb{R} \rightarrow \mathbb{R} $ with respect to $ F ^n $ by
\[
	F ^n ( \xi , x , y ) =: \left( B ^n ( \xi , x ) , F _{(\xi, x)} ^n (y)  \right) .
\]
Indeed, one can calculate
\begin{equation}
	F ^n _{(\xi, x)} (y) = \lambda ^n \left( \rho _{[\xi] _n} (x) \right) \cdot y + W _n \left( \rho _{[\xi] _n} (x) \right) ,	\label{eq:Fn}
\end{equation}
where $ 
	W _n ( x ) := \sum _{j=0} ^{n-1} \lambda ^j (x) \, g ( \tau ^j x ) $.

We define the (KS-)entropy w.r.t. a $ \tau $-invariant probability measure $ \nu \in \mathcal{P} ( [0,1]  ) $ by
\begin{equation}
	h _\nu := - \lim _{N \rightarrow \infty} \frac{1}{N} \int \log \nu ( I _N (x) ) \, d \nu (x) .	\label{eq:entropy_def}
\end{equation}
Recall that, if $ \nu $ is ergodic, then the Shannon-McMillan-Breiman theorem guarantees the convergence
\begin{equation}
	h _\nu =  \lim _{N \rightarrow \infty} \frac{- \log \nu ( I _N (x) ) }{N}	\label{eq:SMB}
\end{equation}
for $ \nu $-a.a. $ x \in [0,1] $.

We also note the following facts.
\begin{proposition}	\label{prop:log_I}
If $ \nu \in \mathcal{P} ( [0,1] ) $ is ergodic, then we have
\[
	\lim _{N \rightarrow \infty} \frac{- \log | I _N (x) |}{N} = \int \log \tau ' \, d \nu
\]
for $ \nu $-a.a. $ x \in [0,1] $.
In particular, $ \int \log \tau ' \, d \nu = - \int \log | I _{k(\cdot)} | \, d \nu $.
\end{proposition}

\begin{proof}
By the mean-value theorem and the distortion estimate of $ \log \tau ' $ there is a $ C > 0 $ such that
\begin{equation*}
	C ^{-1} \leqslant (\tau ^N )' ( x ) \cdot | I _N (x) | \leqslant C
\end{equation*}
for all $ x \in [0,1] $ and $ N \in \mathbb{N} $.
Now, the claim follows by Birkhoff ergodic theorem.
\end{proof}

The following elementary property of the function $ W $ describes its local behaviour accurately.

\begin{proposition}[{\cite[Proposition 3.1]{Moss12}}\footnote{Note that the additional assumption of that paper that all functions are $ C ^{1+} $ on $ \mathbb{T} ^1 $ is not used in the proof of \cite[Proposition 3.1]{Moss12}.}]	\label{prop:Moss}
There is a $ C _m > 0 $ such that
\[
	\sup _{u,v \in I _N (x)} | W (u) - W(v) | \leqslant C _m \, \lambda ^N (x)
\]
holds for all $ x \in [0,1] $ and $ N \in \mathbb{N} $.
\end{proposition}

\subsection{Strong stable fibres}	\label{subsec:strong_stable}

Observe that $ F $ has the derivative matrix
\[
	D F ( \xi , x , y )	=	\left(	\begin{matrix}
		\tau ' (\xi)	&	0	&	0 \\
		0	&	(1 / \tau ') ( \rho _{k(\xi)} (x) )	&	0 \\
		0	&	((y \lambda ' + g') / \tau ') ( \rho _{k(\xi)} (x) )	&	\lambda ( \rho _{k (\xi)} (x) )
	\end{matrix} \right) ,
\]
where we consider only those $ x \in [0,1] $ for which all derivatives exist.

As this is a triangular matrix, one can see in its diagonal the characteristic contracting and expanding scales of this uniformly hyperbolic system.
Especially, the middle value contributes to the strongest contraction.
In order to determine the corresponding direction, which we call the strong stable direction, we define
\begin{eqnarray}
	X _3 ( \xi , x , y ) 	& := &	  - \sum _{n=1} ^\infty  \gamma ^n \left( \rho _{[\xi] _n} (x) \right) \left( F ^{n-1} _{(\xi , x )} ( y ) \cdot \lambda ' \left( \rho _{[\xi] _n} (x) \right) + g ' \left( \rho _{[\xi] _n} (x) \right) \right) , \mbox{ and}	\label{eq:X3_def}
 \\
	 \Theta ( \xi , x ) 	& := &	 X _3 ( \xi , x , W(x) ) \label{eq:Theta_Def} \end{eqnarray}
for $ ( \xi ,x , y ) \in [0,1] ^2 \times \mathbb{R} $, where $ \gamma (x) := 1 / ( \tau ' \lambda ) (x) $.
Since we have
\begin{equation}
	D F ( \xi , x , y )
	\left[
		\begin{matrix}
			0 \\
			1 \\
			X _3 ( \xi , x , y )
		\end{matrix}
	\right]
	= \frac{1}{\tau ' \left( \rho _{k (\xi)} (x) \right)}  
	\left[
		\begin{matrix}
			0 \\
			1 \\
			X _3 \circ F ( \xi , x , y )
		\end{matrix}
	\right]	\label{eq:ss_eigenvalue}
\end{equation}
for each $ ( \xi , x , y ) \in [0,1] ^2 \times \mathbb{R} $, the vector $ ( 0 , 1 , X _3 ( \xi , x , y ) ) ^T $ indicates the strong stable direction at $ ( \xi , x , y ) $.
More precisely, let $ \ell ^{ss} _{(\xi, x , y)} $ be the solution of the initial vale problem
\[
	\left\{
		\begin{array}{rcl}
		\left( \ell ^{ss} _{(\xi,x,y)} \right) ' (v)	& = &	X _3 \left( \xi , v , \ell ^{ss} _{(\xi,x,y)} (v) \right) \\
		 \ell ^{ss} _{(\xi,x,y)} (x) & = & y  
		 \end{array} \right.  .
\]
Clearly, the curve $ v \mapsto \left( \xi , v ,  \ell ^{ss} _{(\xi, x, y)} ( v ) \right ) $ represents the strong stable fibre through $ (\xi, x , y) $, which in particular satisfies
\begin{equation}
	F \left( \xi, v , \ell ^{ss} _{(\xi, x, y)} (v) \right) = \left( B ( \xi , v ) , \ell ^{ss} _{F ( \xi , x , y )} \left( \rho _{k(\xi)} (v) \right) \right) .	\label{eq:lss_inv}
\end{equation}
Slightly abusing notation, we call the function $  \ell ^{ss} _{(\xi, x, y)} $ also the strong stable fibre through $ ( \xi , x , y ) $.

The next proposition shows that the strong stable fibres are almost parallel to each other.

\begin{proposition}	\label{prop:l_parallel}
There is a $ C _s > 0 $ such that
\[
	C _s  ^{-1} \, | y - y' | \leqslant	\left| \ell ^{ss} _{(\xi, x, y)} (v) - \ell ^{ss} _{(\xi, x, y')} (v) \right| \leqslant C _s \, | y - y' |
\]
holds for all $ \xi, x , v \in [0,1] , \; y , y' \in \mathbb{R}$ and $ r > 0 $.
\end{proposition}

\begin{proof}
From \eqref{eq:Fn} and \eqref{eq:X3_def} follows
\begin{eqnarray*}
	\ell ^{ss} _{(\xi, x, y)} (v) - \ell ^{ss} _{(\xi, x, y')} (v)	& = &	y - y' + \int _x ^v X _3 \left( \xi , s , \ell ^{ss} _{(\xi, x, y)} (s) \right) - X _3 \left( \xi , s , \ell ^{ss} _{(\xi, x, y')} (s) \right) \, d s \\
	& = &	y - y' - \int _x ^v A ( \xi, s ) \, \left( \ell ^{ss} _{(\xi, x, y)} (s) - \ell ^{ss} _{(\xi, x, y')} (s) \right) \, d s ,
\end{eqnarray*}
where
\[
	A ( \xi, s ) :=  \sum _{n=1} ^\infty \gamma ^n \left( \rho _{[\xi] _n} (s) \right) \, \lambda ^{n-1} \left( \rho _{[\xi] _{n-1}} (s) \right) \, \lambda ' \left( \rho _{[\xi] _n} (s) \right) .
\]
Thus we have
\[
	\ell ^{ss} _{(\xi, x, y)} (v) - \ell ^{ss} _{(\xi, x, y')} (v) = ( y - y ' ) \, e ^{ - \int _x ^v A (\xi, s) \, d s } \in [ e ^{- \|A \| _\infty},  e ^{\| A \| _\infty} ] \cdot  | y - y' | .
\]
\end{proof}

Any $ \tau $-invariant $ \nu \in \mathcal{P} ([0,1] )$ has the unique $B$-invariant extension $ \nu  ^\ext \in \mathcal{P} ( [0,1] ^2 ) $ which is determined by
\begin{equation}
	\nu  ^\ext \left( I _M \left( \rho _{[\xi] _M} (x) \right) \times I _N ( x ) \right) := \nu  \left( I _{M + N} \left( \rho _{[\xi] _M} (x) \right) \right)	\label{eq:ext_def}
\end{equation}
for $ M, N \in \mathbb{N} $ and $ (\xi ,x ) \in [0,1] ^2 $.

Furthermore, we define the ($ \tau $-invariant) marginal measure $ \nu ^- := \nu ^\ext ( \, \cdot \, \times [0,1 ] ) $.
Then there is a conditional distribution with respect to the vertical partition $ \{ \{ \xi \} \times [0,1] : \xi \in [0,1] \} $ which can be written as product measures $ \delta _\xi \otimes \nu ^+ _\xi $ with $ \nu ^+ _\xi \in \mathcal{P} ( [0,1] ) $ and satisfies
\begin{equation}
	\nu ^\ext = \int \delta _\xi \otimes \nu _\xi ^+ \, d \nu ^- ( \xi ) ,	\label{eq:disintegral_vert}
\end{equation}
see \cite[Example 5.16]{Einsiedler11}.
In addition, let $ \mu ^\ext  := ( \mathrm{Id} ,  W ) ^\ast \nu ^\ext  $ and $ \mu ^+ _\xi := ( \mathrm{Id} , W ) ^\ast \nu ^+ _\xi $.

\begin{remark}
The preceding disintegration of the measure is a prototypical example of Rokhlin's work \cite{Rokhlin52} (see also \cite[Theorem 5.14]{Einsiedler11}).
We mention that some statements in Section \ref{sec:LY} can be alternatively verified by applying this theorem.
\end{remark}

\subsection{Gibbs measure}	\label{subsec:Gbbs}

We call a $ \tau $-invariant $ \nu  \in \mathcal{P} ( [0,1] ) $ Gibbs measure for the potential $ \phi $ which is Hölder continuous on $ I _i $ for each $ i \in \{ 0 , \ldots , \ell -1 \} $, if there are $ C _\phi > 0$ and $ P ( \phi ) \in \mathbb{R} $ such that
\begin{equation}
	\frac{1}{C _\phi}	\leqslant	\frac{\nu  ( I _N ( x ) )}{e ^{\phi _N ( x ) - N P ( \phi )}} \leqslant C _\phi	\label{eq:Gibbs}
\end{equation}
for all $ x \in [0,1] $ and $ N \in \mathbb{N} $.

Note that our potential is always supposed to be Hölder continuous in the above sense.

We prove several properties of Gibbs measures.
In this subsection, let $ \nu $ be a Gibbs measure.

\begin{proposition}	\label{prop:Gibbs}
We have
\[
	C _\phi ^{-3} \leqslant \frac{\nu \left( I _N (x) \cap \tau ^{-N} A  \right)}{\nu ( I _N ( x ) ) \cdot \nu ( A ) } \leqslant C _\phi ^3
\]
for all $ x \in [0,1] $, $ N  \in \mathbb{N} $ and $ A \in \mathcal{B} ( [0,1] ) $.
\end{proposition}

\begin{proof}
Let $ I _N(x) , I _M (\tilde{x} ) \in \mathcal{S} _\tau $ be arbitrary.
Observe that $ I _N (x) \cap \tau ^{-N} \left( I _M (\tilde{x}) \right) = I _{N + M} ( x' ) $ and $ I _M ( \tau ^N x' ) = I _M ( \tilde{x} ) $ for any $ x' \in I _N (x) \cap \tau ^{-N} I _M (\tilde{x}) $.
Since such a $ x' $ always exists, by \eqref{eq:Gibbs} we have
\begin{eqnarray*}
	\nu \left( I _N (x) \cap \tau ^{-N} \left( I _M (\tilde{x}) \right)  \right)	& \in &	[C _\phi ^{-1} , C _\phi  ] \cdot e ^{\phi _{N + M} (x') - (N + M) P ( \phi )} \\
	& \subseteq &	[C _\phi ^{-2}, C _\phi ^2] \cdot e ^{\phi _N (x') - N P ( \phi )} \cdot \nu ( I _M ( \tau ^N x')) \\
	& \subseteq &	[C _\phi ^{-3} , C _\phi ^3] \cdot \nu ( I _N ( x ) ) \cdot \nu ( I _M (\tilde{x}) ) .
\end{eqnarray*}
As this consequence is true for all $ I _M ( \tilde{x} ) \in \mathcal{S} _\tau $, the claim follows by Lemma \ref{lem:approximation}.
\end{proof}

\begin{proposition}	\label{prop:Gibbs_comparison}
We have
\[
	\frac{\nu _\xi ^+}{\nu} , \, \frac{\mu _\xi ^+}{\mu}  \in [C _\phi ^{-3} , C _\phi ^3] 
\]
for $ \nu ^- $-a.a. $ \xi $.

In particular, we have
\[
	\frac{\nu ^\ext}{\nu ^- \otimes \nu} \in [C _\phi ^{-3} , C _\phi ^3] ,
\]
where $ \nu ^- \otimes \nu $ denotes the product measure.
\end{proposition}

\begin{proof}
Consider the $ \sigma $-algebras $ \mathcal{I} _N := \sigma \left(  \{ I _N (\xi) \times [0,1] : \xi \in [0,1] \} \right) $ for $ N \in \mathbb{N} $.
Observe that the filtration $ ( \mathcal{I} _N ) _{N  \in \mathbb{N}}$ has the limit $ \sigma $-algebra $ \mathcal{I} _\infty := \sigma ( \bigcup _N \mathcal{I} _N ) = \mathcal{B} ( [0,1] ) \times [0,1] $.
Let $ \tilde{x} \in [0,1] $ and $ M \in \mathbb{N} $ be fixed.
Then, by \eqref{eq:ext_def} and \eqref{eq:Gibbs} we have
\begin{eqnarray*}
	\frac{\nu ^\ext ( I _N ( \xi ) \times I _M (\tilde{x}) )}{\nu ^\ext ( I _N ( \xi ) \times [0, 1] )}	& = &	 \frac{\nu (I _{M + N} ( \rho _{[\xi] _N} (\tilde{x} )))}{\nu (I _N ( \rho _{[\xi ] _N} (\tilde{x} )))} \\
	& \in &	[ C _\phi ^{-2} , C _\phi ^2 ] \cdot \frac{e ^{\phi _{M+N} (\rho _{[\xi] _N} (\tilde{x} ) )- (M+N) P (\phi)}}{e ^{\phi _N (\rho _{[\xi] _N} (\tilde{x} ) )- N P (\phi)}} \\
	& \subseteq	&	[ C _\phi ^{-3} , C _\phi ^3 ] \cdot \nu ( I _M ( \tilde{x} ) ) .
\end{eqnarray*}
Furthermore, since $ ( \delta _\xi \otimes \nu ^+ _\xi ) _{\xi \in [0,1]} $ is a conditional distribution of $ \nu ^\ext $ with respect to $ \mathcal{I} _\infty $, by the martingale convergence theorem we obtain
\begin{eqnarray*}
	\nu ^+ _\xi ( I _M ( \tilde{x} ) )	& = &	 \nu ^\ext ( [0,1] \times I _M ( \tilde{x} )  | \mathcal{I} _\infty ) ( \xi , x ) \\
	& = &	\lim _{N \rightarrow \infty} \nu ^\ext ( [0,1] \times I _M (\tilde{x}) | \mathcal{I} _N ) ( \xi , x )	\\
	& = &	\lim _{N \rightarrow \infty}	\frac{\nu ^\ext ( I _N ( \xi ) \times I _M (\tilde{x}) )}{\nu ^\ext ( I _N ( \xi ) \times [0, 1] )} \quad \in	 [ C _\phi ^{-3} , C _\phi ^3 ] \cdot \nu ( I _M ( \tilde{x} ) )
\end{eqnarray*}
for $ \nu ^\ext $-a.a. $ ( \xi, x ) $.
As this consequence is true for all $ I _M ( \tilde{x} ) \in \mathcal{S} _\tau $, the inclusion of $ \nu ^+ _\xi / \nu $ in the first claim follows by Lemma \ref{lem:approximation}, which implies also the inclusion of $ \mu ^+ _\xi / \mu $.
Finally, the second claim now follows by virtue of \eqref{eq:disintegral_vert}.
\end{proof}

\begin{proposition}	\label{prop:helper2}
Let $ t \in ( 0 , 1 ) $.
Then, for $ \nu  $-a.a. $ x \in [0, 1] $ there is a $ N _x \in \mathbb{N} $ such that
\[
	B _{| I _N (x) | t ^N}( x ) \subseteq I _N (x)
\]
holds for all $ N \geqslant N _x $.
\end{proposition}

\begin{proof}
We introduce the symbolic notation $ I _{\omega _1 , \ldots , \omega _N} := I _{\omega _1} \cap \tau ^{-1} ( I _{\omega _2} ) \cap \cdots \cap \tau ^{-(N-1)} ( I _{\omega _N} ) $ for $ \omega _1 , \ldots , \omega _N \in \{ 0 , \ldots , \ell - 1 \} $.
Furthermore, let
\[
	E _{M , N} :=	\bigcup _{\omega _1 , \ldots , \omega _N \in  \{ 0 , \ldots , \ell - 1 \} } I _{(\omega _1 , \ldots , \omega _N ) \ast 0 _M }  \cup I _{(\omega _1 , \ldots , \omega _N ) \ast 1 _M } ,
\]
where $ \ast $ is the concatenation operator, and $ 0 _M $ and $ 1 _M $ are the sequences of $ 0 $ and $ 1 $ with length $ M $, respectively.
Let $ \delta := \min \{ | I _0 | , | I _{\ell -1} |\} $.
Then, for $ \alpha := \frac{\log t}{\log \delta} $ let $ M _N := \lfloor \alpha ( N - 1) \rfloor $ and $ E _N := E _{M _N ,N} $.
Note that $ \phi (0) < P ( \phi ) $ and $ \phi (1)< P(\phi) $ are satisfied due to the Gibbs property, as $ 0 $ and $ 1 $ are fix points of $ \tau $.
Since
\[
	\nu  ( E _N ) = \nu ( [0 _{M _N} ]) + \nu  ( [1 _{M _N}] ) \leqslant C _\phi \, \left( e ^{((N-1) \alpha - 1)  (\phi (0)- P(\phi) )} + e ^{((N-1) \alpha - 1 )(\phi (1)- P(\phi) )} \right)
\]
is summable, by Borel-Cantelli Lemma for $ \nu $-a.e. $x $ there is a $ N _x $ such that $ x \not \in E _N $ holds for all $ N \geqslant N _x $, which means especially that
\[
	B _{| I _N ( x ) | t ^N }( x ) = B _{| I _N ( x ) | \delta ^{N \alpha}} ( x  ) \subseteq	B _{| I _N ( x ) |  \delta ^{M _N + \alpha}}  ( x ) \subseteq I _N ( x ) .
\]
\end{proof}

\subsection{Dimension of measures}	\label{subsec:dim}

We review some facts from the dimension theory.
For a metric space $ E $ the lower and upper pointwise dimension of $ \mu \in \mathcal{P} (E) $ at $ x \in E $ is defined by
\begin{equation*}
	\underline{d} _\mu  (x)	 =	\liminf _{r \rightarrow 0} \frac{\log  \mu \left( B _r (x) \right)}{\log r}  \quad  \mbox{and}  \quad
	\overline{d} _\mu  (x)	 =	\limsup _{r \rightarrow 0} \frac{\log  \mu \left( B _r (x) \right)}{\log r}  ,
\end{equation*}
where $ B _r (x)  $ denotes the closed ball with radius $ r > 0 $ and centre $ x \in E $.
If the limit exists, we can also define the pointwise dimension
\[
		d _\mu  (x)	 =	\lim _{r \rightarrow 0} \frac{\log  \mu \left( B _r (x) \right)}{\log r}  .
\]
Moreover, the measure $ \mu $ is called exact dimensional, if $ d _\mu ( x ) $ exists and is constant for $ \mu $-a.a. $ x \in E $.
In this case, we call $ \dim  _H ( \mu ) := d _\mu $ the Hausdorff dimension of $ \mu $.

We give slightly more flexible forms.

\begin{proposition}	\label{prop:pwd_helper}
Let $ \mu $ be a Borel measure on a metric space $ E $.
Let $ K > 0 $ and $ ( \beta _N ) _{N \in \mathbb{N}}  \subset ( 0 , \infty ) $ be such that $ \beta _N \searrow 0 $ monotonically by $ N \rightarrow \infty $ and
\[
	\lim _{N \rightarrow \infty} \frac{\log \beta _{N+1} }{\log  \beta _N } = 1 .
\]
Then we have
\[
	\underline{d} _{\mu } ( x ) = \liminf _{N \rightarrow \infty } \frac{\log \mu  \left( B  _{K \beta _N } ( x  )  \right)}{\log \beta _N}		\quad \mbox{and} \quad		\overline{d} _{\mu } ( x ) = \limsup _{N \rightarrow \infty } \frac{\log \mu \left( B  _{K \beta _N } ( x  )  \right)}{\log \beta _N }
\]
for each $ x \in E $.
\end{proposition}

\begin{proof}
Let $ x \in E $ be fixed.
The claim follows from the fact that
\[
	\frac{\log \mu  \left( B _{ K \beta _N } (x) \right)}{\log K \beta _{N + 1}  } 	 \leqslant	\frac{\log \mu  \left(  B _r ( x  )  \right)}{\log r} 	 \leqslant 	\frac{\log \mu  \left( B _{ K \beta _{N  + 1 }  } (x) \right)}{\log K \beta _N  }
\]
holds for all $ r > 0 $ and $ N \in \mathbb{N} $ such that $ K \beta _{N+1} \leqslant r < K \beta _N $.
\end{proof}

In this note, for a given ($\tau$-invariant or not) $ \nu \in \mathcal{P} ([0,1]) $ we study lift $ \mu \in \mathcal{P} ([0,1] \times \mathbb{R}) $ on the graph of $ W $, i.e. $ \mu := (\mathrm{Id} , W  ) ^\ast \nu $.
This lifted measure $ \mu $ plays a crucial roll since its lower pointwise dimension delivers the lower bound of the Hausdorff dimension of the graph of $ W  $ as the next lemma shows.

\begin{lemma}	\label{lem:mass_d_p}
Let $ \mu $ be the lift of a $ \nu \in \mathcal{P} ( [0,1] ) $ on $ W $.
If $ \underline{d} _\mu \geqslant d $ $ \mu $-a.s. for a $ d \in \mathbb{R} $, then $ \dim _H ( \mathrm{graph} ( W )) \geqslant d $.
\end{lemma}

\begin{proof}
As $ \mu \left( \mathrm{graph} ( W  ) \right) = 1 $, the claim follows from \cite[Theorem 2.1.5]{Barreira08}.
\end{proof}

In order to calculate pointwise dimensions making use of the underlying dynamics, we often need to deal with bad-shaped objects instead of balls.
For this purpose, the following general version of Besicovitch covering theorem, which is originally proved in \cite{Morse47}, is a very powerful tool.

\begin{lemma}[Besicovitch covering theorem (cf. {\cite[P.6 Remarks (4)]{Guzman75}})]	\label{lem:Besicovitch}
Let $ A $ be a bounded subset of $ \mathbb{R} ^n $.
For each $ x \in A $  a set $ H(x) $ is given satisfying the two following properties: {\rm(a)} there exists a fixed number $ M > 0 $, independent of $ x $, and two closed Euclidean balls centered at $ x $, $ B _{r(x)} (x) $ and $ B _{Mr(x)} (x) $ such that $ B _{r(x)} (x)  \subseteq H(x) \subseteq B _{Mr(x) } (x) $; {\rm(b)} for each $ z \in H(x) $, the set $ H(x) $ contains the convex hull of the set $ \{ z \} \cup  B _{r (x)} (x) $.
Then one can select from among $ (H(x)) _{x \in A} $ a sequence $ ( H _k ) _k $ satisfying
\begin{enumerate}
\item The set $ A $ is covered by the sequence, i.e. $ A \subseteq \bigcup _k H _k $.
\item No point of $ \mathbb{R} ^n $ is in more than $ \theta _{n, M} $ (a number that only depends on $ n , M $) elements of the sequence $ ( H _k ) _k $.
\item The sequence $ ( H _k ) _k $ can be divided in $ \xi _{n, M} $ (a number that depends only on $ n, M $) families of disjoint elements.
\end{enumerate}
\end{lemma}

As the first application, we derive a flexible version of Lebesgue density theorem.

\begin{lemma}[Borel density theorem]	\label{lem:Boreldensity}
Let $ \nu \in \mathcal{P} ( \mathbb{R} ^n ) $ and $ g \in L ^1 _\nu $.
Suppose that for all $ x \in \mathbb{R} ^n $ and $ \delta > 0 $ there is a measurable set $ H _\delta (x) $ such that:
\begin{itemize}
 \item $ ( H _\delta ( x ) ) _{x \in \mathbb{R} ^n} $ satisfies (a) and (b) of Lemma \ref{lem:Besicovitch} (with $ M > 0 $ independent of $ \delta $), and
 \item $ \lim _{\delta \rightarrow 0} \mathrm{diam} ( H _\delta (x) ) = 0 $ for each $ x \in \mathbb{R} ^n $.
\end{itemize}
Then we have
\[
	\lim _{\delta \rightarrow 0} \frac{1}{\nu ( H _\delta (x))}	\int _{H _\delta (x)} g \, d \nu = g ( x ) 
\]
for $ \nu $-a.a. $ x $.
\end{lemma}

\begin{proof}
As this result is surely not new, we only sketch the proof mimicking that of \cite[Lemma 4.1.2]{YL185}.
Since the claim is clearly true for continuous $ g $, it suffices to show that the set of the functions $ g $ which satisfy the above condition is norm closed in $ L ^1 _\nu $.
The closedness follows immediately if we show the maximal inequality
\[
	\mu \left\{ x \in [0,1]  : \sup _{\delta >0} \frac{1}{\nu ( H _\delta (x))}	\int _{H _\delta (x)} g \, d \nu > \lambda  \right\} \leqslant \frac{\xi _{n, M}}{\lambda} \int g \, d \nu ,
\]
for $ \lambda > 0 $, as we demonstrate at the end of the proof of Proposition \ref{prop:mu_stabil} for a slightly different situation.

Finally, the proof of the maximal inequality is the same as that of \cite[Lemma 4.1.1(a)]{YL185} if we replace the classical Besicovitch covering theorem by Lemma \ref{lem:Besicovitch}.
\end{proof}

\section{Ledrappier-Young theory}	\label{sec:LY}

We state and prove several formulas about the entropies and dimensions.
The goal is to prove Theorem \ref{thm:entropy_formula}.

In this section we use the following convention.
\begin{itemize}
 \item Let $ \nu \in \mathcal{P} ([0,1]) $ be a Gibbs measure for a potential $ \phi $.
 \item Let $ C _\phi  $ denote the constant of the Gibbs measure $ \nu $.
 \item Let $ \mu \in \mathcal{P} ([0,1] \times \mathbb{R}) $ be the lift of $ \nu  $ on the graph of $ x \mapsto W (x) $.
 \item Let $ \nu ^\ext \in \mathcal{P} ( [0,1] ^2 ) $ be the $ B $-invariant extension of $ \nu $.
 \item Let $ \mu ^\ext \in \mathcal{P} ([0,1] ^2 \times \mathbb{R}) $ be the lift of $ \nu ^\ext $ on the graph of $ (\xi , x ) \mapsto W (x) $.
 \item Let $ I _{(\omega _1 , \ldots , \omega _N )} := I _{\omega _1} \cap \tau ^{-1} I _{\omega _2} \cdots \cap \tau ^{-{N-1}} I _{\omega _N} $ for $ \omega _1 , \ldots , \omega _N \in \{ 0 , \ldots , \ell -1 \} $ and $ N \in \mathbb{N} $.
 \item Let $  R _{(\omega _1 , \ldots , \omega _N )} :=  I _{(\omega _1 , \ldots , \omega _N )} \times \mathbb{R} $ and $ R _N ( x ) := I _N (x) \times \mathbb{R} $.
 \item Let $ Z _{(\omega _1 , \ldots , \omega _N )} := \left\{ ( \eta _i ) _{\mathbb{N}} \in \{ 0, \ldots , \ell -1 \} ^{\mathbb{N}}  : \eta _1 = \omega _1 , \ldots , \eta _N = \omega _N \right\} $ be the cylinder set.
 \item Let $ \mathcal{Z} := \left\{ Z _{(\omega _1 , \ldots , \omega _N )} : \omega _1 , \ldots , \omega _N \in \{ 0 , \ldots , \ell -1 \} \mbox{ and } N \in \mathbb{N} \right\} \cup \{ \emptyset\} $.
   \item Let $ \chi : \{ 0, \ldots , \ell -1 \} ^{\mathbb{N}} \rightarrow [0,1] $ be the coding defined by $ \bigcap _{N \in \mathbb{N}} \overline{I _{(\omega _1 , \ldots , \omega _N )} } = \{ \chi ( \omega ) \}$.
 \item Let $ \pi _\xi ^{ss} ( x , y ) := \ell ^{ss} _{(\xi, x, y)} ( 0 ) $ and $ \Pi ^{ss} ( \xi , x , y ) := ( \xi , \pi _\xi ^{ss} ( x , y) ) $.
 \item Let $ q _\xi ( x ) := \pi _\xi ^{ss} \left( x , W(x) \right) $ and $ q ( \xi , x ) := ( \xi, q _\xi (x ) ) $.
 \item Let $ B ^T _{\xi , r } ( y ) := ( \pi _\xi ^{ss} ) ^{-1} \left( [y - r , y + r ] \right) $.
  \item Let $ B ^T _r ( \xi , y ) := ( \Pi ^{ss} ) ^{-1} \left( [ \xi - r , \xi + r ] \times [y - r , y + r ] \right) $.
   \item Let $ \Sigma _r ( \xi, x ):= \left\{ ( v, y )  \in [0 ,1] \times \mathbb{R} : | \ell ^{ss} _{(\xi, v, y)} (x) - W(x) | \leqslant r \right\} $.
  \item Let $ \delta _x $ denote the Dirac measure on a point $ x $.
\end{itemize}
In addition, we simply write $ B _r (x,y) $ instead of $ B _r ((x,y)) $.

As usual, we consider the probability measures $ \mathcal{P} ( \{ 0, \ldots , \ell -1 \} ^{\mathbb{N}} ) $ with respect to the the cylinder topology.
The following fact is well-known.

\begin{lemma}[Corollary of the extension theorem]	\label{lem:ext_thm}
Suppose that the set function $ \zeta : \mathcal{Z} \rightarrow [0,1] $ is additive and $ \zeta ( \emptyset ) = 0 $.
Then there is a unique $ \tilde{\zeta} \in \mathcal{P} ( \{ 0, \ldots , \ell -1 \} ^{\mathbb{N}} ) $ such that  $\tilde{\zeta} _{| \mathcal{Z} } = \zeta $.
\end{lemma}

\begin{proof}
Evidently, the family $ \mathcal{Z} $ is a semiring.
As every cylinder set is compact, from the additivity follows the $ \sigma $-subadditivity.
Therefore we can apply the extension theorem (cf. \cite[Theorem 1.53]{Klenke}).
\end{proof}

\subsection{Conditional measures for $ \mu $}

The purpose of this subsection is to define the conditional measures of $ \mu $ with respect to the induced $ \xi $-strong stable fibres.

Let $ \xi \in [0,1] $ be fixed arbitrarily.
Let us consider
\[
	\mathcal{G} _\xi := \left\{ z \in \mathbb{R} : \begin{array}{l}
	 \lim _{r \rightarrow 0} \frac{\mu \left( B ^T _{\xi, r} (z) \cap R _{(\omega _1 , \ldots , \omega _N )}  \right)}{\mu \left( B ^T _{\xi, r} (z) \right)}  \mbox{ exists } \\ \mbox{for all } \omega _1 , \ldots , \omega _N \in \{ 0 , \ldots , \ell -1 \} \mbox{ and } N \in \mathbb{N} 
	\end {array}	\right\} .
\]
We have $ \mu \circ ( \pi ^{ss} _\xi )^{-1} ( \mathcal{G} _\xi ) = 1 $ since by Lemma \ref{lem:Boreldensity} we have for all $ \omega _1 , \ldots , \omega _N \in \{ 0 , \ldots , \ell -1 \} $ and $ N \in \mathbb{N} $ that
\begin{equation}
	\lim _{r \rightarrow \infty} \frac{\mu \left( B ^T _{\xi, r} (z) \cap R _{(\omega _1 , \ldots , \omega _N )}  \right)}{\mu \left( B ^T _{\xi, r} (z) \right)}   = \rho _{\xi, (\omega _1 , \ldots , \omega _N )} ( z ) \label{eq:density}
\end{equation}
holds for $ \mu \circ ( \pi ^{ss} _\xi) ^{-1} $-a.a. $ z $, where $ \rho _{\xi, (\omega _1 , \ldots , \omega _N )} ( z ) $ denotes the Radon–Nikodym derivative
\[
	 \frac{d \left( \mu \left( \left( (\pi _\xi ^{ss}) ^{-1} \, \cdot\,  \right) \cap R _{(\omega _1 , \ldots , \omega _N )} \right) \right)}{d ( \mu \circ ( \pi _\xi ^{ss} ) ^{-1} )}  .
\]
Let $ z \in \mathcal{G} _\xi $.
We define
\begin{equation}
	\tilde{\nu} _{(\xi, z )} \left( Z _{(\omega _1 , \ldots , \omega _N )} \right) := \rho _{ \xi , (\omega _1 , \ldots , \omega _N )} ( z )	\label{eq:nu_def}
\end{equation}
for all $ Z _{(\omega _1 , \ldots , \omega _N )} \in \mathcal{Z} $.
As this set function is clearly additive on $ \mathcal{Z} $, by Lemma \ref{lem:ext_thm} we can extend it to a measure $ \tilde{\nu} _{(\xi, z)}  \in \mathcal{P} ( \{ 0, \ldots , \ell -1 \} ^{\mathbb{N}}  ) $ uniquely.
Now, we define $ \nu _{(\xi, z)} := \tilde{\nu} _{(\xi, z)} \circ \chi ^{-1} $.
For $ z \not \in \mathcal{G} _\xi $ let simply $ \nu _{(\xi, z)} := \nu $.
Now, we can define the family of conditional measures\footnote{Lemma \ref{lem:ss_conditional} implies that, given $  \xi \in [0,1] $, the family $ \{ \mu  _{(\xi, \pi ^{ss} _\xi (x,y))} : (x,y) \in [0,1] \times \mathbb{R} \} $ is actually a conditional distribution of $ \mu $ with respect to the $ \sigma $-algebra of the $\xi$-strong stable fibres, i.e. $ (\pi ^{ss} _\xi) ^{-1} \mathcal{B} (\mathbb{R})$.} $ \{ \mu  _{(\xi, z)} : (\xi, z) \in [0,1] \times \mathbb{R} \} $, where
\[
	\mu _{(\xi, z )} := ( \mathrm{Id} , \ell _{(\xi, 0, z)} ^{ss} ) ^\ast \nu _{(\xi , z)} .
\]

\begin{proposition}	\label{prop:mu_atomfree}
We have
\[
	\int  \tilde{\nu} _{(\xi, z)} \left( \{ \omega \} \right)  \, d \left( \mu \circ ( \pi ^{ss} _\xi ) ^{-1} \right) ( z ) = 0
\]
for any $ \xi \in [0,1] $ and $ \omega \in \{ 0, \ldots , \ell -1 \} ^{\mathbb{N}} $.
\end{proposition}

\begin{proof}
By the monotone convergence theorem, we have
\begin{eqnarray*}
	\int  \tilde{\nu} _{(\xi, z)} \left( \{ \omega \} \right) \, d \left( \mu \circ ( \pi ^{ss} _\xi ) ^{-1} \right) ( z ) &  =&		\lim _{N \rightarrow \infty} \int  \tilde{\nu} _{(\xi, z)} \left( Z _{\omega _1 , \ldots , \omega _N} \right)  \, d \left( \mu \circ ( \pi ^{ss} _\xi ) ^{-1} \right) ( z )	\\
	& = &	\lim _{N \rightarrow \infty}  \int  \rho _{ \xi , (\omega _1 , \ldots , \omega _N )} ( z )	  \, d \left( \mu \circ ( \pi ^{ss} _\xi ) ^{-1} \right) ( z ) \\
	& = &	\lim _{N \rightarrow \infty}  \nu \left( I _{ (\omega _1 , \ldots , \omega _N )} \right) \quad =  \quad 0  ,
\end{eqnarray*}
where the last equality is due to the Gibbs property.
\end{proof}

\begin{proposition}	\label{prop:decomposition}
Let $ \xi \in [0,1] $.
Then we have
\[
	\mu = \int \mu _{(\xi, q _\xi (x))} \, d \nu (x) .
\]

In particular, for $ \nu $-a.a. $ x $ we have
\begin{equation}
	\mu _{(\xi, q _\xi (x))} \left( \mathrm{graph} ( W ) \right) = 1 .	\label{eq:cond_on_graph}
\end{equation}
\end{proposition}

\begin{proof}
Evidently, the family
\[
	\mathcal{C} := \left\{  R _{(\omega _1 , \ldots , \omega _N)} \cap ( \pi ^{ss} _\xi ) ^{-1} A :  \omega _1 , \ldots , \omega _N  \in \{ 0 , \ldots , \ell -1 \} , \,  N \in \mathbb{N}  \mbox{ and } A \in \mathcal{B} ( \mathbb{R} ) \right\}
\]
is $ \cap $-stable.
We consider
\[
	V ^{ss} _r ( x , y ) := \left\{ ( \tilde{x} , \tilde{y}) : \tilde{x} \in I _{N _r (x)} (x) \mbox{ and } \left| \ell ^{ss} _{(\xi , \tilde{x}, \tilde{y})} ( x ) - y \right| \leqslant r   \right\} ,
\]
where $ N _r ( x ) := \min \{ N : | I _N (x) | \leqslant r \} $.
As $ V ^{ss} _r ( x , y ) \in \mathcal{C} $ for all $ (  x , y ) \in [0,1] \times \mathbb{R} $ and $ r > 0 $, we have $ \sigma ( \mathcal{C} ) = \mathcal{B} ( [0,1] \times \mathbb{R} ) $.
Observe that $ \chi ( Z _{(\omega _1 , \ldots , \omega _N)}) \triangle I _{(\omega _1 , \ldots , \omega _N)} $ is a finite set.
Thus, by \eqref{eq:density}, \eqref{eq:nu_def} and Proposition \ref{prop:mu_atomfree} we have
\begin{eqnarray*}
	\mu \left( R _{(\omega _1 , \ldots , \omega _N)} \cap ( \pi _\xi ^{ss} ) ^{-1} A \right)	&  =&	\int _A \rho _{\xi , (\omega _1 , \ldots , \omega _N)}  \, d \left( \mu \circ ( \pi ^{ss} _\xi ) ^{-1} \right)  \\
	& = &	\int _A \tilde{\nu} _{(\xi, z)} \left( Z_ {(\omega _1 , \ldots , \omega _N)} \right)  \, d \left( \mu \circ ( \pi ^{ss} _\xi ) ^{-1} \right) ( z ) \\
	& = &	\int _A \nu _{(\xi, z)} \left( I_ {(\omega _1 , \ldots , \omega _N)} \right)  \, d \left( \mu \circ ( \pi ^{ss} _\xi ) ^{-1} \right) ( z ) \\
		& = &	\int \mu _{(\xi, q _\xi (x))} \left( R _{(\omega _1 , \ldots , \omega _N)} \cap ( \pi _\xi ) ^{-1}  A \right) \, d \nu (x)
\end{eqnarray*}
for all $ R _{(\omega _1 , \ldots , \omega _N)} \cap ( \pi _\xi ) ^{-1} A  \in \mathcal{C} $ since $ \mu _{(\xi, z)} \left( ( \pi ^{ss} _\xi ) ^{-1} (z) \right) = 1 $.
This finishes the proof of the first claim.

Finally, the remaining claim is true as $ \int \mu _{(\xi, q _\xi (x))} \left( \mathrm{graph} ( W ) \right) \, d \nu ( x ) =	\mu \left( \mathrm{graph} ( W ) \right)	 = 1 $.
\end{proof}

\begin{lemma}	\label{lem:ss_conditional}
We have
\begin{equation*}
	\int f \left( x , q _\xi (x) \right) \, d \nu ( x ) = \int \left[ \int f \left( \tilde{x} , q _\xi (x) \right) \, d \nu _{\left( \xi ,  q _\xi ( x ) \right)} (\tilde{x} ) \right] d \nu (x)
\end{equation*}
for any bounded measurable $ f : [0,1] \times \mathbb{R} \rightarrow \mathbb{R} $ and $ \xi \in [0,1] $.

\end{lemma}

\begin{proof}
By Proposition \ref{prop:decomposition} we have
\begin{eqnarray*}
	\int f \left( x , q _\xi (x) \right) \, d \nu ( x )	& =&	\int f \left( x , \pi _\xi ^{ss} (x , y) \right) \, d \mu ( x , y ) \\
		& =&	\iint f \left( \tilde{x} , \pi _\xi ^{ss} \left( \tilde{x} , \ell ^{ss} _{(\xi, 0, q _\xi (x))} (\tilde{x}) \right) \right) \, d \nu _{(\xi, q _\xi (x))} ( \tilde{x} )  \, d \nu ( x ) \\
		& =&	\iint f \left( \tilde{x} , q _\xi ( x ) \right) \, d \nu _{(\xi, q _\xi (x))} ( \tilde{x} )  \, d \nu ( x ) .
\end{eqnarray*}

\end{proof}

\subsection{Conditional measures for $ \mu ^\ext $}
The purpose of this subsection is to define the conditional measures of $ \mu ^\ext $ with respect to the strong stable fibres on each $ \xi $-hyperplane.
We also define the associated conditional entropy.

Let us consider the set
\[
	\mathcal{F} := \left\{ ( \xi , z ) \in  [0,1] \times \mathbb{R} : \begin{array}{l}
	 \lim _{r \rightarrow 0} \frac{\mu ^\ext \left( B ^T _r (\xi, z) \cap [0,1] \times R _{(\omega _1 , \ldots , \omega _N )}  \right)}{\mu ^\ext \left( B ^T _r ( \xi , z) \right)}  \\ \mbox{ exists for all } \omega _1 , \ldots , \omega _N \in \{ 0 , \ldots , \ell -1 \} \mbox{ and } N \in \mathbb{N} 
	\end {array}	\right\} .
\]
We have $ \mu ^\ext \circ ( \Pi ^{ss} )^{-1} ( \mathcal{F} ) = 1 $ since by Lemma \ref{lem:Boreldensity} we have for all $ \omega _1 , \ldots , \omega _N \in \{ 0 , \ldots , \ell -1 \} $ and $  N \in \mathbb{N} $ that
\begin{equation}
	\lim _{r \rightarrow \infty} \frac{\mu ^\ext \left( B ^T _r (\xi, z) \cap [0,1] \times R _{(\omega _1 , \ldots , \omega _N )}  \right)}{\mu ^\ext \left( B ^T _r ( \xi , z) \right)}  =	\rho ^\ext _{(\omega _1 , \ldots , \omega _N )} ( \xi, z)  \label{eq:density_ext}
\end{equation}
holds for $ \mu ^\ext \circ ( \Pi ^{ss}) ^{-1} $-a.a. $ ( \xi, z ) $, where $ \rho ^\ext _{(\omega _1 , \ldots , \omega _N )} $ denotes the Radon–Nikodym derivative
\[
	 \frac{d \left( \mu ^\ext \left( \left( ( \Pi ^{ss} ) ^{-1} \, \cdot\,  \right) \cap [0,1] \times R _{(\omega _1 , \ldots , \omega _N )} \right) \right)}{d ( \mu ^\ext \circ ( \Pi ^{ss}) ^{-1} )} .
\]
Let $ ( \xi, z ) \in \mathcal{F} $.
We define
\begin{equation}
	\tilde{\nu} ^\ext _{(\xi, z )} \left( Z _{(\omega _1 , \ldots , \omega _N )} \right) := \rho ^\ext _{ (\omega _1 , \ldots , \omega _N )} ( \xi,  z )	\label{eq:nu_def_ext}
\end{equation}
for all $ Z _{(\omega _1 , \ldots , \omega _N )} \in \mathcal{Z} $.
As this set function is clearly additive on $ \mathcal{Z} $, by Lemma \ref{lem:ext_thm} we can extend it to a measure $ \tilde{\nu} ^\ext _{(\xi, z)}  \in \mathcal{P} ( \{ 0, \ldots , \ell -1 \} ^{\mathbb{N}}  ) $ uniquely.
Now, we define $ \nu ^\ext _{(\xi, z)} := \tilde{\nu} ^\ext _{(\xi, z)} \circ \chi ^{-1} $.
For $ (\xi , z ) \not \in \mathcal{F} $ let simply $ \nu ^\ext _{(\xi, z)} := \nu $.
Now, we can define the family of conditional measures\footnote{Lemma \ref{lem:ss_conditional_ext} implies that, the family $ \{ \mu ^\ext _{\Pi ^{ss} (\xi, x ,y)} : (\xi, x ,y) \in [0,1] ^2 \times \mathbb{R} \} $ is a conditional distribution of $ \mu ^\ext $ with respect to the $ \sigma $-algebra of the strong stable fibres, i.e. $ (\Pi ^{ss}) ^{-1} \mathcal{B} ([0,1] \times \mathbb{R})$.} $ \{ \mu ^\ext _{(\xi, z)} : (\xi, z) \in [0,1] \times \mathbb{R} \} $ by letting
\[
	\mu _{(\xi, z )} ^\ext := ( \mathrm{Id} , \ell _{(\xi, 0, z)} ^{ss} ) ^\ast \nu _{(\xi , z)} ^\ext .
\]

\begin{proposition}	\label{prop:mu_atomfree_ext}
We have
\[
	\int  \tilde{\nu} ^\ext _{(\xi, z)} \left( \{ \omega \} \right)  \, d \left( \mu ^\ext \circ ( \pi ^{ss} _\xi ) ^{-1} \right) ( z ) = 0
\]
for each $ \omega \in \{ 0, \ldots , \ell -1 \} ^{\mathbb{N}} $.
\end{proposition}

\begin{proof}
By the monotone convergence theorem, we have
\begin{eqnarray*}
	\int  \tilde{\nu} ^\ext _{(\xi, z)} \left( \{ \omega \} \right) \, d \left( \mu ^\ext \circ ( \Pi ^{ss}  ) ^{-1} \right) ( \xi, z ) &  =&		\lim _{N \rightarrow \infty} \int  \tilde{\nu} ^\ext _{(\xi, z)} \left( Z _{\omega _1 , \ldots , \omega _N} \right)  \, d \left( \mu ^\ext \circ ( \Pi ^{ss}  ) ^{-1} \right) ( \xi, z )  \\
	& = &	\lim _{N \rightarrow \infty}  \int  \rho ^\ext _{ (\omega _1 , \ldots , \omega _N )} ( \xi, z )	  \, d \left( \mu ^\ext \circ ( \Pi ^{ss}  ) ^{-1} \right) ( \xi, z )  \\
	&  =& 	\lim _{N \rightarrow \infty}  \nu ^\ext \left( [0,1] \times I _{ (\omega _1 , \ldots , \omega _N )} \right) \\
	& = &	\lim _{N \rightarrow \infty}  \nu \left( I _{ (\omega _1 , \ldots , \omega _N )} \right) \quad =  \quad 0  ,
\end{eqnarray*}
where the last equality is due to the Gibbs property.
\end{proof}

\begin{proposition}	\label{prop:decomposition_ext}
We have
\[
	\mu ^\ext = \int \delta _\xi \otimes  \mu ^\ext _{(\xi, q _\xi (x))} \, d \nu ^\ext ( \xi, x ).
\]
\end{proposition}

\begin{proof}
Evidently, the family
\[
	\tilde{\mathcal{C}} := \left\{ [0,1] \times R _{(\omega _1 , \ldots , \omega _N)} \cap ( \Pi ^{ss} ) ^{-1} A :  \omega _1 , \ldots , \omega _N  \in \{ 0 , \ldots , \ell -1 \} , \,  N \in \mathbb{N}  \mbox{, and } A \in \mathcal{B} ( [0,1] \times \mathbb{R} ) \right\}
\]
is $ \cap $-stable.
We consider
\[
	\tilde{V} ^{ss} _r ( \xi , x , y ) := \left\{ (\tilde{\xi} , \tilde{x} , \tilde{y}) : \tilde{x} \in I _{N _r (x)} (x) , \, | \xi - \tilde{\xi} | \leqslant r  \mbox{ and } \left| \ell ^{ss} _{(\tilde{\xi}, \tilde{x}, \tilde{y})} ( x ) - y \right| \leqslant r   \right\} ,
\]
where $ N _r ( x ) := \min \{ N : | I _N (x) | \leqslant r \} $.
As $ \tilde{V} ^{ss} _r ( \xi ,x , y ) \in \mathcal{C} $ for all $ ( \xi, x , y ) \in [0,1] ^2 \times \mathbb{R} $ and $ r > 0 $, we have $ \sigma ( \mathcal{C} ) = \mathcal{B} ( [0,1] ^2 \times \mathbb{R} ) $.
Observe that $ \chi ( Z _{(\omega _1 , \ldots , \omega _N)}) \triangle I _{(\omega _1 , \ldots , \omega _N)} $ is a finite set.
Thus, by \eqref{eq:density_ext}, \eqref{eq:nu_def_ext} and Proposition \ref{prop:mu_atomfree_ext} we have
\begin{eqnarray*}
	&&	\mu ^\ext \left( [0,1] \times R _{(\omega _1 , \ldots , \omega _N)} \cap ( \Pi ^{ss} ) ^{-1} A \right)	\\ 
		& = &	\int _A  \rho ^\ext _{ (\omega _1 , \ldots , \omega _N )} \, d \left( \mu ^\ext \circ ( \Pi ^{ss} ) ^{-1} \right)  \\
		& = &	\int _A   \tilde{\nu} ^\ext _{(\xi, z )} \left(  Z _{(\omega _1 , \ldots , \omega _N )} \right) \, d \left( \mu ^\ext \circ ( \Pi ^{ss} ) ^{-1} \right) ( \xi , z )  \\
		& = &	\int _A   \nu ^\ext _{(\xi, z )} \left(  I _{(\omega _1 , \ldots , \omega _N )} \right) \, d \left( \mu ^\ext \circ ( \Pi ^{ss} ) ^{-1} \right) ( \xi , z )  \\
	& = &	\int _A  \delta _\xi \otimes \mu ^\ext _{(\xi, z )} \left( [0,1 ] \times R _{(\omega _1 , \ldots , \omega _N )} \right) \, d \left( \mu ^\ext \circ ( \Pi ^{ss} ) ^{-1} \right) ( \xi , z )  \\
		& = &	\int  \delta _\xi \otimes \mu ^\ext _{(\xi, q _\xi (x) )} \left( [0,1] \times R _{(\omega _1 , \ldots , \omega _N)} \cap ( \Pi ^{ss} ) ^{-1} A  \right) \, d \nu ^\ext ( \xi , x )  \\
\end{eqnarray*}
for all $ [0,1] \times R _{(\omega _1 , \ldots , \omega _N)} \cap ( \Pi ^{ss} ) ^{-1} A \in \tilde{\mathcal{C}} $ since $ \delta _\xi \otimes \mu _{(\xi, z)} ^\ext \left(  ( \Pi ^{ss} ) ^{-1} \left(  (\xi ,y) \right) \right) = 1 $.
This finishes the proof.
\end{proof}

\begin{lemma}	\label{lem:ss_conditional_ext}
We have
\begin{equation*}
	\int f \left(  x , \xi , q _\xi (x) \right) \, d \nu ^\ext (  \xi , x ) = \int \left[ \int f \left( \tilde{x} , \xi , q _\xi (x) \right) \, d \nu ^\ext _{\left( \xi , q _\xi ( x ) \right)} (\tilde{x} ) \right] d \nu ^\ext (\xi , x )
\end{equation*}
for all bounded measurable $ f : [0,1] ^2 \times \mathbb{R} \rightarrow \mathbb{R} $.
\end{lemma}

\begin{proof}
By Proposition \ref{prop:decomposition_ext} we have
\begin{eqnarray*}
	\int f \left( x , \xi ,  q _\xi (x) \right) \, d \nu ^\ext ( \xi, x )	& =&	 \int f \left( x , \Pi ^{ss} (\xi , x , y) \right) \, d \mu ^\ext ( \xi,  x , y ) \\
		& =& 	\iint f \left( \tilde{x} , \Pi ^{ss} \left( \xi ,\tilde{x} , \ell ^{ss} _{(\xi, 0 , q _\xi (x))} (\tilde{x}) \right) \right) \, d \nu ^\ext _{(\xi, q _\xi (x))} ( \tilde{x} )  \, d \nu ^\ext ( \xi ,  x ) \\
		& =&	\iint f \left( \tilde{x} , \xi , q _\xi (x) \right) \, d \nu ^\ext _{(\xi, q _\xi (x))} ( \tilde{x} )  \, d \nu ^\ext( \xi,  x ) .
\end{eqnarray*}
\end{proof}

The next proposition can be interpreted as the tower rule, regarding the family $ \{ \mu _\xi ^+ : \xi \in [0,1] \} $ as a conditional distribution w.r.t. the partition into the stable hyperplanes, i.e. $ \{ \{ \xi \} \times [0,1] \times \mathbb{R} : \xi \in [0,1] \} $.

\begin{proposition}	\label{prop:mu_+_ext}
We have
\[
	\nu _\xi ^+ = \int \nu ^\ext _{(\xi, q_\xi (x) )} \, d \nu ^+ _\xi (x) 
\]
for $ \nu ^- $-a.a. $ \xi $.

In particular, we have for $ \nu ^\ext $-a.a. $ ( \xi, x ) $ that
\begin{equation}
	\lim _{r \rightarrow 0} \frac{\mu ^+ _\xi \left( B ^T _{\xi, r} ( q _\xi (x) ) \cap R _{(\omega _1 , \ldots , \omega _N )} \right)}{\mu ^+ _\xi \left( B ^T _{\xi, r} ( q _\xi (x) ) \right)} = \nu ^\ext _{(\xi, q_\xi (x) )} \left(  I _{(\omega _1 , \ldots , \omega _N )} \right) \label{eq:density_+}
\end{equation}
holds for all $ \omega _1 , \ldots , \omega _N \in \{ 0 , \ldots , \ell -1 \} $ and $ N \in \mathbb{N}  $.
\end{proposition}

\begin{proof}
Let $ \tilde{\nu} _\xi := \int \nu ^\ext _{(\xi, q_\xi (x) )} \, d \nu ^+ _\xi (x)  $ for $ \xi \in [0,1] $.
In view of the uniqueness of the conditional distribution, for the first part of the claim it suffices to show that the family $ \{ \delta _\xi \otimes \tilde{\nu} _\xi : \xi \in [0,1] \} $ is a conditional distribution of $ \nu ^\ext $ with respect to the vertical partition $ \eta := \{ \{ \xi \} \times [0,1] : \xi \in [0,1] \} $ as $ \{ \delta _\xi \otimes \nu ^+ _\xi : \xi \in [0,1] \} $ is.
This is true since by \eqref{eq:disintegral_vert} and Proposition \ref{prop:decomposition_ext} we have
\begin{eqnarray*}
	\int  \delta _\xi \otimes \tilde{\nu} _\xi ( A \times B ) \, \mathbf{1} _{C \times [0,1]} ( \xi , x ) \, d \nu ^\ext ( \xi , x )	&  = &	\iint \delta _{\xi} ( A \cap C ) \cdot \nu ^\ext _{(\xi, q _\xi (x) )} ( B ) \, d \nu ^+ _\xi ( x ) \, d \nu ^- (\xi ) \\
	&  = &	\int \delta _\xi \otimes \mu ^\ext _{(\xi, q _\xi (x))} ( ( A \cap C ) \times ( B \times \mathbb{R} ) ) \, d \nu ^\ext ( \xi ,x ) \\
	& = &	\mu ^\ext ( ( A \cap C ) \times ( B \times \mathbb{R} ) ) \\
	&  = &	\nu ^\ext ( ( A \times B ) \cap (C \times [0,1]) )
\end{eqnarray*}
for all $ A , B , C \in \mathcal{B} ( [0,1] ) $, where $ \mathbf{1} _{C \times [0,1]} $ denotes the indicator function of $ C \times [0,1] $.
Indeed, we have $ \sigma ( \{ A \times B : A , B \in \mathcal{B} ( [0,1] ) \} ) = \mathcal{B} ( [0,1] ^2 ) $ and $ \sigma ( \{ C \times [0,1] : C \in \mathcal{B} ( [0,1] ) \} ) = \sigma ( \eta ) $.

In order to prove the remaining part, let $ \omega _1 , \ldots , \omega _N \in \{ 0 , \ldots , \ell -1 \} $ and $ N \in \mathbb{N} $.
Observe that for $ \nu ^- $-a.a. $ \xi $ we have that
\begin{eqnarray*}
	\mu ^+ _\xi \left( B ^T _{\xi, r} ( z ) \cap R _{(\omega _1 , \ldots , \omega _N )} \right)	&  = &	\int \mu ^\ext _{(\xi, \tilde{z})} \left( B ^T _{\xi, r} ( z ) \cap R _{(\omega _1 , \ldots , \omega _N )} \right) \, d \left( \nu ^+ _\xi \circ q _\xi ^{-1} \right) ( \tilde{z} ) \\
		&  = &	\int _{B _r ( z )} \nu ^\ext _{(\xi, \tilde{z})} \left(  I _{(\omega _1 , \ldots , \omega _N )} \right) \, d \left( \nu ^+ _\xi \circ q _\xi ^{-1} \right) ( \tilde{z} ) 
\end{eqnarray*}
holds for all $ z \in \mathbb{R} $ and $ r > 0 $.
Thus, for any such $ \xi \in [0,1] $, the convergence of \eqref{eq:density_+} holds for $ \nu _\xi ^+ $-a.a. $ x $ by Lemma \ref{lem:Boreldensity}.
In view of \eqref{eq:disintegral_vert} the convergence holds thus for $ \nu ^\ext $-a.a. $ ( \xi, x ) $.
\end{proof}

\begin{proposition}	\label{prop:Gibbs_comparison_cond}
We have
\[
	\frac{\nu ^\ext _{(\xi, q_\xi (x) )}}{\nu _{(\xi, q _\xi (x) )}} \in [C _\phi ^{-6} , C _\phi ^6] 
\]
for $ \nu ^\ext $-a.a. $ ( \xi, x ) $.
\end{proposition}

\begin{proof}
From \eqref{eq:density}, \eqref{eq:nu_def} and \eqref{eq:density_+} together with Proposition \ref{prop:Gibbs_comparison} we have for $ \nu ^\ext $-a.a. $ ( \xi, x ) $ that
\[
	 \nu _{(\xi, q_\xi (x) )} \left(  I _{(\omega _1 , \ldots , \omega _N )} \right) \in [C _\phi ^{-6} , C _\phi ^6] \cdot \nu ^\ext _{(\xi, q_\xi (x) )} \left(  I _{(\omega _1 , \ldots , \omega _N )} \right)
\]
holds for all $  I _{(\omega _1 , \ldots , \omega _N )} \in \mathcal{S} _\tau $.
Thus the claim follows by Lemma \ref{lem:approximation}.
\end{proof}

\begin{proposition}	\label{prop:cond_entropy}
We have
\[
	\nu _{q \circ B ^{-1} (\xi, x)} ^\ext \left( I _N ( \tau (x) ) \right) = \frac{\nu _{q (\xi, x)} ^\ext \left( I _{N + 1} (x) \right)}{\nu _{q (\xi, x)} ^\ext \left( I _1 (x) \right)}
\]
for all $ N \in \mathbb{N} $ and $ \nu ^\ext $-a.a. $ ( \xi , x ) $.
\end{proposition}

\begin{proof}
Let $ \eta ( \xi , x , y ) := ( (\Pi) ^{ss} ) ^{-1} ( \Pi ^{ss} ( \xi , x , y ) ) $ and $ \eta := \left\{ \eta ( \xi , x , y) : ( \xi , x , y) \in [0,1] ^2 \times \mathbb{R} \right\} $.
Further, let
\[
	F \eta ( \xi , x , y)	 := 	 F \left( \eta ( F ^{-1} (\xi,x,y)) \right) =  \eta ( \xi , x , y ) \cap \left( [0,1] \times R _{k (x)} \right)
\]
and $ F \eta  := \left\{  F \eta ( \xi , x , y ) : ( \xi , x ) \in [0,1] ^2 \times \mathbb{R} \right\} $.
We define
\[
	\mu ^{F \eta} _{(\xi, z)} ( A ) := \frac{\mu _{(\xi, z)} ^\ext \left(   R _{k (x)} \cap A   \right) }{\mu _{(\xi, z)} ^\ext \left( R _{k (x)} \right) }
\]
for $ (\xi, z ) \in [0,1] \times \mathbb{R} $ and $ A \in \mathcal{B} ( [0,1] \times \mathbb{R} ) $.
By virtue of the $F$-invariance of $ \mu ^\ext $ it is straightforward to verify that both
\[
 \left( \delta _\xi \otimes \mu ^{F \eta} _{\Pi ^{ss} (\xi, x, y)} \right) _{(\xi, x , y) \in [0,1] ^2 \times \mathbb{R}} \quad \mbox{ and } \quad  \left( \left( \delta _{\rho _{k(x)} (\xi)} \otimes \mu _{\Pi ^{ss} \circ F ^{-1} (\xi, x, y)} ^\ext  \right) \circ F ^{-1} \right) _{(\xi, x , y) \in [0,1] ^2 \times \mathbb{R}}
 \]
are conditional probability distributions of $ \mu ^\ext $ with respect to $  F \eta $.

Thus due to their uniqueness we have for $ \nu ^\ext $-a.a. $ ( \xi ,x  ) $ that
\begin{eqnarray*}
	\mu ^{F \eta} _{q (\xi , x )} \left( R _{\omega _0 , \omega _1, \ldots , \omega _N)} \right)	& = &	\delta _\xi \otimes \mu ^{F \eta} _{q (\xi , x )} \left( [0,1] \times R _{( \omega _0 , \omega _1 , \ldots , \omega _N)} \right)	\\
	& = &	\left( \delta _{\rho _{k(x)} (\xi)} \otimes \mu _{q \circ B ^{-1} (\xi, x )} ^\ext  \right) \left( F ^{-1} \left( [0,1] \times R _{(\omega _0 , \omega _1, \ldots , \omega _N)} \right) \right) \\
	& = &	\left( \delta _{\rho _{k(x)} (\xi)} \otimes \mu _{q \circ B ^{-1} (\xi, x )} ^\ext  \right) \left(  I _{\omega _0} \times R _{(\omega _1, \ldots , \omega _N)}  \right) \\
	& = &	\delta _{\rho _{k(x)} (\xi)} ( I _{\omega _0} ) \cdot \mu _{q \circ B ^{-1} (\xi, x )} ^\ext \left( R _{( \omega _1, \ldots , \omega _N)} \right) 
\end{eqnarray*}
holds for all $ \omega _0 , \ldots , \omega _N \in \{ 0 , \ldots , \ell -1 \} $ and $ N \in \mathbb{N} $.
Inserting $ ( \omega _0 , \ldots, \omega _N ) = [x] _{N+1} $ finishes the proof.
\end{proof}

\begin{proposition}	\label{prop:h_ss_conv}
we have
\begin{equation*}
	\lim _{N \rightarrow \infty } \frac{\log \nu _{ (\xi, q _\xi (x))} \left( I _N ( x ) \right)}{N}	 = \lim _{N \rightarrow \infty } \frac{\log \nu _{ (\xi, q _\xi (x))} ^\ext \left( I _N ( x ) \right)}{N}	 =	\int \log \nu _{(\xi, q _\xi ( x ))} ^\ext \left( I _1 (x) \right) \, d \nu ^\ext ( \xi, x ) 	
\end{equation*}
for $ \nu ^\ext $-a.a. $ ( \xi, x )$.
\end{proposition}

\begin{proof}
The second equality follows from Proposition \ref{prop:cond_entropy} due to Birkhoff ergodic theorem, whereas the first one can be derived by Proposition \ref{prop:Gibbs_comparison_cond}.
\end{proof}

In view of Proposition \ref{prop:h_ss_conv} it is natural to define the conditional entropy with respect to the the strong stable fibres as
\begin{equation}
		h ^{ss} _\mu 	:=	- \int \log \nu _{(\xi, q _\xi ( x ))} ^\ext \left( I _1 (x) \right) \, d \nu ^\ext ( \xi, x ) .	\label{eq:h_ss}
\end{equation}

\subsection{Dimension and entropy formula}
We shall state and prove several formulas.

The following proposition is a version of the maximal inequality.
Note that this does not follow from the classical Besicovitch covering theorem directly since the distance from $ q _\xi (x) $ to the centre of $ H _r (\xi, x ) $ depends not only on $ \xi $ but also on $ x $.

\begin{proposition}	\label{prop:Hardy}
There is a $ C _h > 0 $ such that
\[
	\nu \left\{ x \in [0,1]  : \sup _{r > 0}  \frac{1}{ \nu \circ q _\xi ^{-1}  ( H _r ( \xi , x ))}   \int _{H _r ( \xi , x  )} \, g \, d \left( \nu \circ q _\xi ^{-1} \right)  > \lambda  \right\} \leqslant \frac{C _h}{\lambda} \int g \, d \left( \nu \circ q _\xi ^{-1} \right)
\]
holds for all $ \xi \in [0,1] $, $ g \in L ^1 _{\nu \circ q _\xi ^{-1}} $ and $ \lambda > 0 $, where $ H _r ( \xi ,x ) := \pi ^{ss} _\xi \left( \Sigma _r ( \xi , x ) \right) $.
\end{proposition}

\begin{proof}
Let $ \xi \in [0,1] $, $ g \in L ^1 _{\nu \circ q _\xi ^{-1}} $ and $ \lambda > 0 $ be fixed.
We consider
\[
	E :=  \left\{  x \in [0,1] : \sup _{r > 0}  \frac{1}{ \nu \circ q _\xi ^{-1}  ( H _r ( \xi , x ))}  \int _{H _r ( \xi , x )} \, g \, d \left(  \nu \circ q _\xi ^{-1}  \right)  > \lambda \right\} \mbox{, and}
\]
\[
	\mathcal{F} _z := \left\{ ( x , r ) \in q _\xi ^{-1} (z) \times (0, \infty ) : \frac{1}{ \nu \circ q _\xi ^{-1}  ( H _r ( \xi , x ))}  \int _{H _r ( \xi , x )} \, g \, d \left(  \nu \circ q _\xi ^{-1}  \right)  > \lambda \right\}
\]
for $ z \in \mathbb{R} $.
As $ \mathcal{F} _z \neq \emptyset $ for each $ z \in q _\xi ( E ) $, by the axiom of choice we can parametrise $ z \mapsto ( x _z, r _z ) $ so that $ (x _z, r _z ) \in \mathcal{F} _z $.
Let $ H ( z ) := H _{r _z} ( \xi, x _z ) $.
Since we have $ B _{ C _s ^{-1} \, r _z } (z) \subseteq H (z) \subseteq B _{  C _s \, r _z } ( z ) $ for each $ z \in q _\xi ( E ) $ by Proposition \ref{prop:l_parallel}, $ ( H (z) ) _{z \in q_\xi ( E ) } $ is a cover of $ q _\xi ( E ) $, which satisfies the assumptions (a) and (b) of Lemma \ref{lem:Besicovitch} with $ M := C _s ^2 $.
Thus by that lemma there is a subcover $ ( H ( z _i ) ) _{i \in J} $ with $ J \subseteq \mathbb{N} $ satisfying (1)-(3) of the lemma.
Since $ q _\xi ( E ) \subseteq \bigcup _{i \in J} H ( z _i ) $ by (1), we have by (3)
\begin{eqnarray*}
	\nu   ( E )	 \quad \leqslant \quad 	\nu  \left( \bigcup _{i \in J}  q _\xi ^{-1} H ( z _i ) \right) 
		& \leqslant &	\sum _{i \in J}  \nu \circ q _\xi ^{-1}  \left( H ( z _i ) \right) \\
		& \leqslant &	\frac{1}{\lambda} \sum _{i \in J} \int _{H (z _i)} g \, d \left(  \nu \circ q _\xi ^{-1}  \right) \quad \leqslant \quad	\frac{\xi _{1 , M}}{\lambda} \int g \, d \left(  \nu \circ q _\xi ^{-1}  \right) .
\end{eqnarray*}
\end{proof}

\begin{proposition}	\label{prop:mu_stabil}
For any $ \xi \in [0,1] $ and $ A \in \mathcal{B} \left( [0,1] \times \mathbb{R} \right) $ we have that
\[
	\lim _{r \rightarrow 0} \frac{\mu ( \Sigma _r ( \xi, x) \cap A )}{ \mu  ( \Sigma _r ( \xi , x ))} = \lim _{r \rightarrow 0} \frac{\mu  ( B ^T _{\xi ,r} \left( q _\xi (x) ) \cap A \right)}{ \mu \left( B ^T _{\xi, r} ( q _\xi (x) ) \right)}
\]
holds for $ \nu  $-a.a. $ x$.
\end{proposition}

\begin{proof}
Let $ \xi $ and $ A $ be fixed.
Then there exists the Radon–Nikodym derivative
\[
	\rho _A := \frac{d \left( \mu  \left( (  \pi ^{ss} _\xi ) ^{-1} ( \, \cdot \, )  \cap A \right) \right) }{d \left( \mu \circ ( \pi ^{ss} _\xi ) ^{-1} \right)} .
\]
Note that $ \mu \circ ( \pi ^{ss} _\xi ) ^{-1} = \nu  \circ q _\xi ^{-1} $.
Since
\begin{eqnarray*}
	\lim _{r \rightarrow 0} \frac{\mu  ( B ^T _{\xi , r} \left( q _\xi (x) ) \cap A \right)}{ \mu \left( B ^T _{\xi, r} ( q _\xi (x) ) \right)}	& = &	\lim _{r \rightarrow 0}  \frac{1}{\nu \circ q _\xi ^{-1} ( B _r ( q _\xi (x) ) )}  \int _{B _r ( q _\xi (x) )} \rho _A \, d \left( \nu \circ q _\xi ^{-1} \right) \\
	& = &	\rho _A \circ q _\xi (x)
\end{eqnarray*}
for $ \nu $-a.a. $ x $ by Lemma \ref{lem:Boreldensity}, it suffices to show that the left hand side of the claim also converges to $ \rho _A \circ q _\xi  $.
In order to prove this we need another density theorem.
Let $ g \in L ^1 _{\nu  \circ q _\xi ^{-1}} $.
We claim that
\[
	\lim _{r \rightarrow 0}  \frac{1}{\nu \circ q _\xi ^{-1} ( H _r ( \xi , x ))}  \int _{H _r ( \xi, x )} \, g \, d \left( \nu \circ q _\xi ^{-1} \right) = g \circ q _\xi
\]
holds for $ \nu $-a.a. $ x $, where $ H _r ( \xi ,x ) := \pi ^{ss} _\xi \left( \Sigma _r ( \xi, x ) \right) $.
The proof follows then by taking $ g = \rho _A $.
The following argument mimics the proof of \cite[Lemma 4.1.2]{YL185}.
Since the claim is clearly true for any continuous $ g $, it suffices to show that the set of the functions $ g $ which satisfy the above equality is norm closed in $ L ^1 _{\nu \circ q _\xi ^{-1}} $.
Recall that for an arbitrary $ g \in L ^1 _{\nu \circ q _\xi ^{-1}} $ there are continuous functions $ ( g _k ) _k $ such that $  g _k \rightarrow g   $ in $ L ^1 $.
Furthermore, we have then
\begin{eqnarray*}
	&&	\limsup _{r \rightarrow 0} \left| \frac{1}{\nu \circ q _\xi ^{-1} ( H _r ( \xi , x ))}  \int _{H _r ( \xi, x )} \, g \, d \left( \nu  \circ q _\xi ^{-1} \right) - g ( q _\xi (x) ) \right|	\\
	& \leqslant &	\sup _{r > 0} \frac{1}{\nu  \circ q _\xi ^{-1} ( H _r ( \xi , x ))}  \int _{H _r ( \xi, x )} \, | g - g _k | \, d \left( \nu \circ q _\xi ^{-1} \right) \\
	&&	\quad + \limsup _{r \rightarrow 0} \left| \frac{1}{\nu \circ q _\xi ^{-1} ( H _r ( \xi , x ))}  \int _{H _r ( \xi, x )} \,  g _k  \, d \left( \nu  \circ q _\xi ^{-1} \right) - g _k ( q _\xi (x) ) \right|	+ | g _k ( q _\xi (x) ) - g ( q _\xi ( x ) ) | \\
	&  =& \sup _{r > 0} \frac{1}{\nu \circ q _\xi ^{-1} ( H _r ( \xi , x ))}  \int _{H _r ( \xi, x )} \, | g - g _k | \, d \left( \nu \circ q _\xi ^{-1} \right) + \left| g _k ( q _\xi (x) ) - g ( q _\xi ( x ) ) \right| .
\end{eqnarray*}
Thus it follows that
\begin{eqnarray*}
	&&	\nu \left\{ x \in [0,1] : \limsup _{r \rightarrow 0} \left| \frac{1}{\nu \circ q _\xi ^{-1} ( H _r ( \xi , x ))}  \int _{H _r ( \xi, x )} \, g \, d \left( \nu \circ q _\xi ^{-1} \right) - g ( q _\xi (x) )  \right| > 2 \lambda \right\}  \\
	& \leqslant & \nu \left\{ \sup _{r > 0} \frac{1}{\nu \circ q _\xi ^{-1} ( H _r ( \xi , x ))}  \int _{H _r ( \xi, x )} \, | g - g _k | \, d \left( \nu \circ q _\xi ^{-1} \right)  > \lambda \right\}  + \nu \circ q _\xi ^{-1} \left\{ \left| g _k - g  \right| > \lambda \right\} \\
	& \leqslant &	 \frac{C _h}{\lambda} \| g - g _k \| _{L ^1 } + \frac{1}{\lambda} \| g _k - g \| _{L ^1} \quad \stackrel{k \rightarrow \infty}{\longrightarrow} 0 
\end{eqnarray*}
for each $ \lambda > 0 $ by Proposition \ref{prop:Hardy} and Markov's inequality.
Thus the claim is proved.
\end{proof}

We introduce the quantity
\[
	d _{M , r} (\xi , x ) := \frac{\mu   \left( \Sigma _r ( \xi , x ) \cap R _M (x) \right)}{\mu   \left( \Sigma _r ( \xi , x ) \right)}
\]
for $ M \in \mathbb{N} $, $ r > 0 $ and $ (\xi, x) \in [0,1] ^2 $.

\begin{proposition}		\label{prop:d_conv}
Let $ \xi \in [0,1] $ and $ M \in \mathbb{N} $.
Then we have
\[
	\lim _{r \rightarrow 0} d _{M, r} ( \xi , x ) = \nu _{\left( \xi , q _\xi (x) \right)} ( I _M (x) )
\]
for $\nu $-a.a. $ x \in [0,1] $.
\end{proposition}

\begin{proof}
From \eqref{eq:density} and \eqref{eq:nu_def} together with Proposition \ref{prop:mu_stabil} we have for $ \nu $-a.a. $ x $ that
\[
	\lim _{r \rightarrow 0} \frac{\mu ( \Sigma _r ( \xi, x) \cap R _{(\omega _1, \ldots , \omega _M )} )}{ \mu  ( \Sigma _r ( \xi , x ))} = \nu _{\left( \xi , q _\xi (x) \right)} ( I _{(\omega _1, \ldots , \omega _M )} )
\]
for all $ \omega _1 , \ldots , \omega _M \in \{ 0 , \ldots , \ell - 1 \} $.
Now, inserting $ (\omega _1, \ldots , \omega _M ) = [x] _M $ finishes the proof.
\end{proof}

\begin{proposition}	\label{prop:maximal_ineq}
We have
\[
	\int \sup _{r>0} \left| \log d _{M, r} ( \xi , x )  \right| d \nu ^\ext (\xi, x) < \infty
\]
for each $ M \in \mathbb{N} $.
\end{proposition}

\begin{proof}
It suffices to prove $ \sum _{n=1} ^\infty \nu ^\ext \left\{ (\xi, x ) \in [0,1] ^2 : \inf _{r > 0} d _{M, r} (\xi, x) < e ^{-n} \right\} < \infty $.
Observe that
\begin{eqnarray*}
	&&	\nu ^\ext \left\{ ( \xi, x) \in [0,1] ^2 : \inf _{r > 0} d _{M, r} (\xi, x) < e ^{-n} \right\}		\\
	& =	&	\sum _{\omega _1, \ldots , \omega _M \in \{ 0 , \ldots , \ell - 1 \}} \nu ^\ext \left\{ (\xi, x) \in [0,1] \times I _{(\omega _1 , \ldots , \omega _M)} : \inf _{r > 0} \frac{\mu  \left(\Sigma _r ( \xi , x )  \cap R _{(\omega _1 , \ldots , \omega _M)} \right)}{\mu \left( \Sigma _r ( \xi , x )  \right)} < e ^{-n} \right\} .
\end{eqnarray*}
Let $ (\omega _1 , \ldots , \omega _M) $ and $ n $ be fixed.
Further, let $ \xi \in [0,1] $ be any of those for which the first assertion of Proposition \ref{prop:Gibbs_comparison} is true.
Then we consider
\[
	E := \left\{ x \in I _{(\omega _1 , \ldots , \omega _M)} : \inf _{r > 0} \frac{\mu  \left(\Sigma _r ( \xi , x )  \cap R _{(\omega _1 , \ldots , \omega _M)} \right)}{\mu \left( \Sigma _r ( \xi , x )  \right)} < e ^{-n} \right\}  \mbox{, and}
\]
\[
\mathcal{F} _y := \left\{ (x,r) \in q _\xi ^{-1} (y) \cap I _{(\omega _1 , \ldots , \omega _M)} \times ( 0 , \infty ) : \frac{\mu  \left(\Sigma _r ( \xi , x )  \cap R _{(\omega _1 , \ldots , \omega _M)} \right)}{\mu \left( \Sigma _r ( \xi , x )  \right)} < e ^{-n}  \right\} .
\]
Now, exactly the same argument as in the proof of Proposition \ref{prop:Hardy}, where we also named the corresponding objects $ E $ and $ \mathcal{F} _z $, yields a cover $ ( H (z _i ) ) _{i \in J} $ of $ q _\xi ( E ) $ with $ J \subseteq \mathbb{N} $ satisfying the same properties.
By the construction we have $ H ( z _i ) = \pi ^{ss} _\xi ( \Sigma _{r _i } ( \xi , x _i ) ) $ if we define $ x _i := x _{z_i} $ and $ r _i := r _{z _i } $ by using the parametrisation there.
Recall that $ ( x _i , r _i ) \in \mathcal{F} _{z _i} $.
Since this cover satisfies in particular the multiplicity condition (3) of Lemma \ref{lem:Besicovitch} with the constant $ \xi _{1 , C _s ^2} > 0 $, we have
\begin{eqnarray*}
	\nu ^+ _\xi  ( E )	& \leqslant &	\nu ^+ _\xi  \left( q _\xi ^{-1} \left( \bigcup _{i \in J } H (z _i ) \right) \cap I _{(\omega _1 , \ldots , \omega _M)}  \right) \\
	& =&	\mu ^+ _\xi  \left( \bigcup _{i \in J} \Sigma _{ r _i } ( \xi , x _i ) \cap R _{(\omega _1 , \ldots , \omega _M)}  \right) \\
	& \leqslant &	 \sum _{i \in J } \mu ^+ _\xi   \left( \Sigma _{ r _i } ( \xi , x _i ) \cap R _{(\omega _1 , \ldots , \omega _M)} \right) \\
	& \leqslant &	C _\phi ^3 \, \sum _{i \in J } \mu    \left( \Sigma _{ r _i } ( \xi , x _i ) \cap R _{(\omega _1 , \ldots , \omega _M)} \right) \\
	& \leqslant &	C _\phi ^3 \, e ^{-n} \, \sum _{i \in J} \mu \left( \Sigma _{ r _i } ( \xi , x _i ) \right) \\
	& \leqslant &	C _\phi ^3 \, \xi _{1, C _s ^2} \, e ^{-n}  .
\end{eqnarray*}
As this is true for all $ (\omega _1 , \ldots , \omega _M) \in \{ 0, \ldots , \ell -1 \} ^M $ and $ n \in \mathbb{N} $, and for $ \nu ^- $-a.a. $ \xi $, by \eqref{eq:disintegral_vert} we obtain
\begin{eqnarray*}
	&&	\sum _{n=1} ^\infty \nu ^\ext \left\{ ( \xi, x ) \in [0,1] ^2 : \inf _{r > 0} d _r (\xi, x) < e ^{-n} \right\}	\\
	&= &	\sum _{\omega _1, \ldots , \omega _M \in \{ 0 , \ldots , \ell - 1 \}}  \sum _{n=1} ^\infty \int  \nu ^+ _\xi \left\{ x \in I _{(\omega _1 , \ldots , \omega _M)} : \inf _{r > 0} d _r (\xi, x) < e ^{-n} \right\} \, d \nu ^- ( \xi )  \leqslant \frac{C _\phi ^3 \, \xi _{1, C _s ^2} \, e}{e - 1} \, \ell ^M  .
\end{eqnarray*}
\end{proof}

\begin{proposition}	\label{prop:ledrappier_telescoping}
We have
\begin{eqnarray*}
	\mu   \left( \Sigma _{\lambda ^{M N} (x)} ( \xi , x ) \right)	 & = &	 \mu   \left( \Sigma _1  ( B ^{-M N} ( \xi , x ) )  \right) \cdot \prod _{k=0} ^{N-1} \frac{\mu \left( \Sigma _{\lambda ^{M(N-k)} ( \tau ^{M k} (x))} ( B ^{-M k} ( \xi , x ) ) \right)}{\mu   \left( \Sigma _{\lambda ^{M (N-k-1)} ( \tau ^{M (k+1)} (x))} ( B ^{-M ( k+1 )} ( \xi , x ) ) \right)} \\
	& \in &	[C  _\phi ^{-3 N} , C _\phi ^{3N}] \cdot \mu   \left( \Sigma _1  ( B ^{-M N} ( \xi , x ) )  \right) \cdot  \prod _{k=0} ^{N-1} \frac{\nu   ( I _M ( \tau ^{M k}(x)))}{d _{M, \lambda ^{M(N-k)} (\tau ^{M k} (x))} (B ^{- M k} ( \xi , x))}
\end{eqnarray*}
for all $ ( \xi , x ) \in [0,1] ^2 $ and $ M , N \in \mathbb{N} $.
\end{proposition}

\begin{proof}
The fist line is only a rewriting.
For the second line, observe that we have
\begin{eqnarray*}
&&	\left\{ v \in [0,1] : ( v , W(v) ) \in  \Sigma _{\lambda ^{M(N-k)} ( \tau ^{M k} (x))} ( B ^{-M k} ( \xi , x ) ) \cap R _M \left( \tau ^{M k} (x) \right)  \right\} \\
& = &	\left\{ v \in I _M ( \tau ^{M k} (x) ) : \left| \ell _{\left( \rho _{[x] _{M k} (\xi)} , v , W(v) \right)} ^{ss} ( \tau ^{M k} (x) ) - W ( \tau ^{M k} (x) ) \right| \leqslant \lambda ^{M ( N -k )} ( \tau ^{M k} (x)) \right\} \\
	& = &	 \left\{ v \in I _M ( \tau ^{M k} (x) ) :  \left| \ell _{\left( B ^{-M} \left( \rho _{[x] _{M k} (\xi)} ,  v \right) , W( \tau ^M (v)) \right)} ^{ss}  ( \tau ^{M (k+1)} (x) ) - W ( \tau ^{M (k+1)} (x) ) \right|  \right. \\
	&&	\qquad \qquad \qquad \qquad \qquad  \left. \leqslant \frac{\lambda ^{M ( N -k )} ( \tau ^{M k} (x))}{\lambda ^M ( \tau ^{M k} (x))} \right\} \\
	& = &	 \left\{ v \in I _M ( \tau ^{M k} (x) ) :  \left| \ell _{\left(  \rho _{[x] _{M (k+1)}} (\xi)  , \tau ^M (v) , W(\tau ^M (v)) \right)} ^{ss} ( \tau ^{M (k+1)} (x) ) - W ( \tau ^{M (k+1)} (x) ) \right| \right. \\
	&& \quad \qquad \qquad \qquad \qquad   \left. \leqslant \lambda ^{M(N-k-1)} (\tau ^{M(k+1)} (x)) \right\} \\
	& = &	I _M ( \tau ^{M k} (x) ) \cap \tau ^{-M} \left\{ v \in [0,1] : ( v , W(v) ) \in \Sigma _{\lambda ^{M(N-k-1)} ( \tau ^{M (k+1)} (x))} ( B ^{-M (k+1)} ( \xi , x ) ) \right\}
\end{eqnarray*}
for $ k = 0 , \ldots , \ell -1 $ since $ [v] _M = [\tau ^{Mk} x ] _M $.
Thus by Proposition \ref{prop:Gibbs} we have
\begin{eqnarray*}
	&&	\mu \left( \Sigma _{\lambda ^{M(N-k)} ( \tau ^{M k} (x))} ( B ^{-M k} ( \xi , x ) ) \cap R _M \left( \tau ^{M k} (x) \right) \right) \\
	& \in & [C  _\phi ^{-3} , C _\phi ^3]  \cdot \nu ( I _M ( \tau ^{M k} (x) ) ) \cdot  \mu \left( \Sigma _{\lambda ^{M(N-k-1)} ( \tau ^{M (k+1)} (x))} ( B ^{-M (k+1)} ( \xi , x ) )  \right)
\end{eqnarray*}
for $ k = 0 , \ldots , N - 1 $.
The claim follows by inserting these into the first line.
\end{proof}

\begin{proposition}	\label{prop:hM}
Let $ M \in \mathbb{N} $.
Then we have
\[
	\lim _{N \rightarrow \infty} \frac{1}{N} \sum _{k=0} ^{N-1} \log d _{M, \lambda ^{M(N-k)} (\tau ^{M k} (x))} (B ^{- M k} ( \xi , x)) = \int \log \nu _{ \left( \xi , q _\xi (x) \right)} ( I _M (x) )  \, d \nu   ^\ext (\xi, x)
\]
for $ \nu  ^\ext $-a.a. $ (\xi ,x ) $.
\end{proposition}

\begin{proof}
Let $ g _M ( \xi , x ) :=  \log \nu _{\left( \xi , q _\xi (x) \right)} ( I _M (x) )  $.
By Birkhoff ergodic theorem we have
\[
	\lim _{N \rightarrow \infty} \frac{1}{N} \sum _{k=0} ^{N-1} g _M (B ^{-M k} ( \xi , x)) = \int g _M  \, d \nu   ^\ext 
\]
for $ \nu ^\ext $-a.a. $ (\xi ,x ) $.
Furthermore, again by Birkhoff ergodic theorem we have
\begin{eqnarray*}
	&&	\limsup _{N \rightarrow \infty} \left| \frac{1}{N} \sum _{k=0} ^{N-1}  \log d _{M, \lambda ^{M(N-k)} (\tau ^{M k} (x))} (B ^{- M k} ( \xi , x))  -   \frac{1}{N} \sum _{k=0} ^{N-1} g _M   (B ^{-M k} ( \xi , x))  \right| \\
	& \leqslant &	\limsup _{N \rightarrow \infty} \frac{1}{N} \sum _{k=0} ^{N-1} \left| \log d _{M, \lambda ^{M(N-k)} (\tau ^{M k} (x))} (B ^{- M k} ( \xi , x))  - g _M   (B ^{-M k} ( \xi , x))  \right| \\
	& \leqslant &	\limsup _{N \rightarrow \infty} \frac{1}{N} \sum _{k=0} ^{N-1} D _R (B ^{-M k} ( \xi , x)) \quad =	\int D _R \, d \nu   ^\ext
\end{eqnarray*}
for $ \nu ^\ext $-a.a. $ (\xi ,x ) $, where $ D _R ( \xi , x ) := \sup _{r \in ( 0 , R )} \left| \log d _{M, r} ( \xi ,x ) - g _M ( \xi , x ) \right| $.
Letting $ R \rightarrow 0 $ finishes the proof since the last integral goes to zero by Propositions \ref{prop:d_conv} and \ref{prop:maximal_ineq} due to the dominated convergence theorem.
\end{proof}

\begin{proposition}	\label{prop:vertical_limit_estimate}
We have
\begin{eqnarray*}
	&&	\limsup _{r \rightarrow 0} \left| \frac{\log \mu   ( \Sigma _r ( \xi , x ) )}{\log r} - \frac{ \int \log \nu _{ \left( \xi , q _\xi ( x ) \right)}  ( I _M (x) )  \,  d \nu   ^\ext (\xi, x) -  \int  \log \nu   ( I _M (x) ) \, d \nu   (x) }{- M \int \log \lambda \, d \nu  } \right| \\
	& &	\quad \leqslant  \frac{3 \log C _\phi }{- M \int \log \lambda \, d \nu   }
\end{eqnarray*}
for $ \nu ^\ext $-a.a. $ (\xi ,x ) $.
\end{proposition}

\begin{proof}
Similarly to Proposition \ref{prop:pwd_helper}, one can show
\[
	\limsup _{r \rightarrow 0} \left| \frac{\log \mu   ( \Sigma _r ( \xi , x ) )}{\log r} - a \right| =	\limsup _{N \rightarrow \infty} \left| \frac{\log \mu   ( \Sigma _{\lambda ^{M N} (x)} ( \xi , x ) )}{\log \lambda ^{ M N} (x)} - a \right|
\]
for any $ (\xi, x )  \in [0,1] ^2 $ and $ a \in \mathbb{R} $.
Furthermore, we have $ \log \mu ( \Sigma _1 ) \in L ^1 _{\nu ^\ext} $ since $ \inf _{(\xi, x)} \mu ( \Sigma _1 (\xi, x) ) > 0 $ by Proposition \ref{prop:Moss} and the Gibbs property.
Thus the claim follows from Proposition \ref{prop:ledrappier_telescoping} since we have
\begin{eqnarray*}
	&& \lim _{N \rightarrow \infty} \frac{ \log \left(	\prod _{k=0} ^{N-1} \frac{\nu   ( I _M ( \tau ^{M k}(x)))}{d _{M, \lambda ^{M(N-k)} (\tau ^{M k} (x))} (B ^{- M k} ( \xi , x))} \right)}{\log \lambda ^{ M N} (x) } \\
	&& \quad= 	\frac{ \int \log \nu _{ \left( \xi , q _\xi ( x ) \right)}  ( I _M (x) )  \,  d \nu   ^\ext (\xi, x) -  \int  \log \nu   ( I _M (x) ) \, d \nu   (x) }{- M \int \log \lambda \, d \nu  } 
\end{eqnarray*}
by Proposition \ref{prop:hM} and Birkhoff ergodic theorem.
\end{proof}

Recall that $ h _\nu $ and $ h ^{ss} _\mu $ are defined in \eqref{eq:entropy_def} and \eqref{eq:h_ss}, respectively.

\begin{lemma}	\label{lem:pre-YL}
We have
\[
	\lim _{r \rightarrow 0} \frac{\log \mu   ( B ^T _{\xi, r} ( q _\xi (x) ) )}{\log r} = \lim _{r \rightarrow 0} \frac{\log \mu   ( \Sigma _r ( \xi , x ) )}{\log r}  = \frac{ h  _\nu - h _{\mu } ^{ss}  }{ - \int \log \lambda \, d \nu  }
\]for $ \nu ^\ext $-a.a. $ ( \xi, x ) $.

In particular, $ \nu \circ q _\xi ^{-1} $ is exact dimensional for $ \nu ^- $-a.a. $ \xi $ with
\begin{equation*}
	\dim _H ( \nu \circ q _\xi ^{-1} ) = \frac{ h  _\nu - h _\mu  ^{ss}  }{ - \int \log \lambda \, d \nu  } .
\end{equation*}
\end{lemma}

\begin{proof}
The second equality in the first claim follows from Proposition \ref{prop:vertical_limit_estimate} by letting $ M \rightarrow \infty $ since we have
\begin{eqnarray*}
	- M \cdot h ^{ss} _\mu	&  = &	M \int \log \nu ^\ext _{ \left( \xi , q _\xi ( x ) \right)}  ( I _1 (x) )  \,  d \nu   ^\ext (\xi, x) \\
		&  =&	\int \log \nu ^\ext  _{ \left( \xi , q _\xi ( x ) \right)}  ( I _M (x) )  \,  d \nu   ^\ext (\xi, x) \\
		& \in &	 [\log C _\phi ^{-6} , \log C _\phi ^6 ] + \int \log \nu  _{ \left( \xi , q _\xi ( x ) \right)}  ( I _M (x) )  \,  d \nu   ^\ext (\xi, x)
\end{eqnarray*}
by Propositions \ref{prop:Gibbs_comparison_cond} and \ref{prop:cond_entropy}.
On the other hand, the first equality is true in view of $\Sigma _{C _s ^{-1} r} ( \xi ,x ) \subseteq B ^T _{\xi, r} ( q _\xi (x) ) \subseteq  \Sigma _{C _s r} ( \xi ,x ) $, where $ C _s > 0 $ denotes the constant of Proposition \ref{prop:l_parallel}.

Now, we prove the remaining claim on the dimension.
Let $ \gamma _2 := -( h  _\nu - h _\mu  ^{ss}  )/ \int \log \lambda \, d \nu  $.
As $ \mu   ( B ^T _{\xi, r} ( y ) ) = \nu \circ q _\xi ^{-1} ( B _r (y) ) $, we have by \eqref{eq:disintegral_vert}
\begin{eqnarray*}
	&&	\int \nu ^+ _\xi \circ q _\xi ^{-1}  \left\{ y \in \mathbb{R} : \lim _{r \rightarrow 0} \frac{\log \nu \circ q _\xi ^{-1} ( B _r (y) )}{\log r} = \gamma _2 \right\} \, d \nu ^- (\xi) \\
	&  =&	\nu ^\ext \left\{ (\xi, x) \in [0,1] ^2 : \lim _{r \rightarrow 0} \frac{\log \nu \circ q _\xi ^{-1} ( B _r (q _\xi (x)) )}{\log r} = \gamma _2 \right\}	\quad = 1 .
\end{eqnarray*}
Thus by Proposition \ref{prop:Gibbs_comparison} we can conclude
\[
	\nu \circ q _\xi ^{-1}  \left\{ y \in \mathbb{R} : \lim _{r \rightarrow 0} \frac{\log \nu \circ q _\xi ^{-1} ( B _r (y) )}{\log r} = \gamma _2 \right\} = 1
\]
for $ \nu ^- $-a.a. $ \xi $.
\end{proof}

\begin{lemma}	\label{lem:dim_cond}
We have
\[
	\lim _{r \rightarrow 0} \frac{\log \nu _{(x, q_\xi (x) )} ( B _r (x) )}{\log r} = \frac{h ^{ss} _\mu}{\int \log \tau ' \, d \nu} 
\]
for $ \nu ^\ext $-a.a. $(\xi, x) $.

In particular, $ \nu  _{(x, q_\xi (x) )} $ and $ \mu _{(x, q_\xi (x) )} $ are exact dimensional for $ \nu ^\ext $-a.a. $ ( \xi, x ) $ with
\[
	\dim _H (  \nu _{(x, q_\xi (x) )} ) = \dim _H (  \mu _{(x, q_\xi (x) )} ) =  \frac{h ^{ss} _\mu}{\int \log \tau ' \, d \nu} .
\]
\end{lemma}

\begin{proof}
By Propositions \ref{prop:log_I} and \ref{prop:h_ss_conv} we have
\[
 \lim _{N \rightarrow \infty} \frac{\log \nu _{(x, q_\xi (x) )} ( I _N (x) ) }{\log | I _N (x) |} = \frac{h ^{ss} _\mu}{\int \log \tau ' \, d \nu}
\]
for $ \nu ^\ext $-a.a. $ ( \xi , x ) $.
Therefore, for the first claim it suffices to show that
\begin{equation*}
		d _{ \nu _{(x, q_\xi (x) )} } ( x ) = \frac{\log \nu _{(x, q_\xi (x) )} ( I _N (x) ) }{\log | I _N (x) |} 
\end{equation*}
holds for $ \nu ^\ext $-a.a. $ ( \xi , x ) $.
Let $ t \in (0,1) $ be arbitrary.
By Proposition \ref{prop:helper2} for $ \nu $-a.a. $ x $ there is a $ N _x $ such that
\[
	B _{| I _N (x) | t ^N } ( x ) \subseteq I _N ( x ) \subseteq B _{|I _N (x) |} (x) 
\]
holds for all $ N \geqslant N _x $.
From this follows for such  $ x $ and any $ \xi \in [0,1] $
\begin{eqnarray*}
	\overline{d} _{ \nu _{(x, q_\xi (x) )} } ( x ) 	&  =& 	\limsup _{N \rightarrow \infty}	\frac{\log \nu _{(x, q_\xi (x) )} ( B _{| I _N (x) |} (x) )}{\log | I _N (x) |}	\\
	& \leqslant &	\lim _{N \rightarrow \infty} \frac{\log \nu _{(x, q_\xi (x) )} ( I _N (x) )}{\log | I _N (x) |} \\
	& \leqslant & \liminf _{N \rightarrow \infty} \frac{\log \nu _{(x, q_\xi (x) )} ( B _{| I _N (x) | t ^N} (x) )}{\log ( | I _N (x) | t ^N)} \left( 1 + \frac{\log t}{\log \min _i | I _i |} \right) \\
		& = & \underline{d} _{ \nu _{(x, q_\xi (x) )} } ( x )  \left( 1 + \frac{\log t}{\log \min _i | I _i |} \right)
\end{eqnarray*}
by virtue of Proposition \ref{prop:pwd_helper}, as $ \lim _{N \rightarrow \infty} \frac{\log | I _{N+1} (x) |}{\log | I _N (x) |} =  \lim _{N \rightarrow \infty} \frac{\log ( | I _{N+1} (x) | t ^{N+1} )}{\log ( | I _N (x) | t ^N )} = 1 $.
Letting $ t \nearrow 1 $ finishes this proof.

Next, we consider the dimension of $ \nu _{(x, q_\xi (x) )} $.
Let $ \gamma _1 := h ^{ss} _\mu / \int \log \tau ' \, d \nu $. 
By Lemma \ref{lem:ss_conditional_ext} we have
\begin{eqnarray*}
	&&	\int \nu ^\ext_{(\xi, q_\xi(x) )}  \left\{ \tilde{x} \in [0,1] : \lim _{r \rightarrow 0} \frac{\log \nu _{(\xi, q_\xi(x) )} ( B _r (\tilde{x}) )}{\log r} = \gamma _1 \right\} \, d \nu ^\ext (\xi, x) \\
	&  =&	\nu ^\ext \left\{ ( \xi , x ) \in [0,1] ^2 : \lim _{r \rightarrow 0} \frac{\log \nu _{(\xi, q_\xi(x) )} ( B _r (x) )}{\log r} = \gamma _1 \right\}	\quad = 1 .
\end{eqnarray*}
Thus by Proposition \ref{prop:Gibbs_comparison_cond} we can conclude
\[
	\nu _{(\xi, q_\xi(x) )}  \left\{ \tilde{x} \in [0,1] : \lim _{r \rightarrow 0} \frac{\log \nu _{(\xi, q_\xi(x) )} ( B _r (\tilde{x}) )}{\log r} = \gamma _1 \right\}  = 1
\]
for $ \nu ^\ext $-a.a. $ ( \xi, x ) $.

Finally, the pointwise dimension of $ \mu _{(x, q_\xi (x) )} $ at $ ( v , \ell ^{ss} _{(\xi, x, W(x))} (v) ) $ coincides with that of $ \nu _{(x, q_\xi (x) )} $ at $ v $ for all $ v \in [0,1] $ and $ ( \xi , x ) \in [0,1] ^2 $ since the strong stable fibres are bi-Lipschitz continuous.
\end{proof}

The next lemma describes the product structure of $ \mu $ from the dimension theoretic point of view.
Its proof mimics that of \cite[Proposition 3.7]{Barani14} partially.

\begin{lemma}	\label{lem:mu_product}
$ \mu $ is exact dimensional with
\[
	\dim _H ( \mu )	 = 	\dim _H ( \mu _{(\xi, q_\xi (x) )} ) + \dim _H ( \nu \circ q _\xi ^{-1} ) 
\]
for $ \nu ^\ext $-a.a. $ ( \xi ,x ) $.
\end{lemma}

\begin{proof}
Let $ \beta := \exp ( \int \log \lambda \, d \nu ) $, $ \gamma _1 :=  h ^{ss} _\mu / \int \log \tau ' \, d \nu $ and $ \gamma _2 := - ( h ^+ _\mu - h ^{ss} _\mu )  / \int \log \lambda \, d \nu  $.
Further, let $ C _0 := \| X _3 \| _\infty + 1 $ so that
\begin{equation}
	B _r ( x , W(x) ) \subseteq \{ ( v ,y ) \in \Sigma _{C _0 r} ( \xi ,x ) : | x - v | \leqslant r  \}	\label{eq:proof_prod_5}
\end{equation}
holds for all $ x \in [0,1] $ and $ r > 0 $.
We shall prove that the lower and upper pointwise dimension are almost surely bounded by $ \gamma _1 + \gamma _2 $ from below and above, respectively.

For the lower bound, let $ \varepsilon > 0 $ be arbitrary.
For $ N _0 \in \mathbb{N} $ let $ \mathcal{G} _{N _0} $ denote the set of those $ ( \xi, x ) \in [0,1] ^2 $ which satisfy
\begin{eqnarray}
 \nu _{(x, q_\xi (x) )} ( B _{2 \beta ^N} ( x ) ) 	& \in &	[\beta ^{ N ( \gamma _1 + \varepsilon )} , \beta ^{ N ( \gamma _1 - \varepsilon )}] \mbox{, and} \label{eq:proof_prod_1} \\
 \mu  ( \Sigma _{C _0 \beta ^N} ( \xi, x) ) 	& \in &	[\beta ^{N (\gamma _2 + \varepsilon)} , \beta ^{N (\gamma _2 - \varepsilon)}  ]  \label{eq:proof_prod_2} 
\end{eqnarray}
for all $ N \geqslant N _0 $.
Then we have $ \lim _{N \rightarrow \infty} \int \nu ^+ _\xi ( ( \mathcal{G} _N ) _\xi ) \, d \nu ^- ( \xi ) = \lim _{N \rightarrow \infty} \nu ^\ext ( \mathcal{G} _N ) = 1 $ by \eqref{eq:disintegral_vert}, Lemmas \ref{lem:pre-YL} and \ref{lem:dim_cond}, where $ (\mathcal{G} _N ) _\xi := \{ x \in [0,1 ] : (\xi, x ) \in \mathcal{G} _N \} $.
By the monotone convergence theorem and Proposition \ref{prop:Gibbs_comparison}, we have thus $ \lim _{N \rightarrow \infty} \nu  ( (\mathcal{G} _N ) _\xi ) = 1 $ for $ \nu ^- $-a.a. $ \xi $.

Let such a $ \xi \in [0,1] $ be fixed.
Then we choose a $ N _0 \in \mathbb{N} $ so large that $ \nu  ( (\mathcal{G} _{N _0} ) _\xi ) \geqslant 1 - \varepsilon $.
By Lemma \ref{lem:Boreldensity} and Egorov's theorem there are $ N _1 \geqslant N _0 $ and $ \mathcal{G} ' \subseteq (\mathcal{G} _{N _0} ) _\xi  $ with $ \nu ( \mathcal{G} ' ) > 1 - 2 \varepsilon $ such that
\begin{equation}
	\mu \left( B _{\beta ^N} ( x , W(x) ) \cap (\mathcal{G} _{N _0} ) _\xi \times \mathbb{R} \right) \geqslant \frac{1}{2} \mu  \left( B _{\beta ^N}  (x, W(x) ) \right)	\label{eq:proof_prod_3}
\end{equation}
holds for all $ x \in \mathcal{G} ' $ and $ N \geqslant N _1 $.

We claim that
\begin{equation}
	\mu _{(\xi, z )} \left( B _{\beta ^N} ( x , W(x) ) \cap (\mathcal{G} _{N _0} ) _\xi \times \mathbb{R} \right) \leqslant \beta ^{N (\gamma _2 - \varepsilon)} 	\label{eq:proof_prod_4}
\end{equation}
holds for all $ z \in \mathbb{R} $, $ x \in \mathcal{G} ' $ and $ N \geqslant N _1 $.
If $  \mu _{(\xi, z )} \left( B _{\beta ^N} ( x , W(x) ) \cap (\mathcal{G} _{N _0} ) _\xi \times \mathbb{R} \right)  = 0 $, the claim is trivial.
Thus we assume that there is a $ ( \tilde{x} , \ell ^{ss} _{(\xi, 0, z)} ( \tilde{x} ) ) \in B _{\beta ^N} ( x , W(x) ) \cap (\mathcal{G} _{N _0} ) _\xi \times \mathbb{R}  $.
In this case, $	\mu _{(\xi, z )} \left( B _{\beta ^N} ( x , W(x) ) \cap (\mathcal{G} _{N _0} ) _\xi \times \mathbb{R} \right)  \leqslant \nu _{(\xi, z )} ( B _{2 \beta ^N} ( \tilde{x} ) )   = \nu _{(\xi, q _\xi ( \tilde{x} ) )} ( B _{2 \beta ^N} ( \tilde{x} ) )) \leqslant \beta ^{N (\gamma _2 - \varepsilon)}  $ by \eqref{eq:proof_prod_1}.

Now, from \eqref{eq:proof_prod_5}, \eqref{eq:proof_prod_2}, \eqref{eq:proof_prod_3} and \eqref{eq:proof_prod_4} together with Lemma \ref{lem:ss_conditional} follows
\begin{eqnarray*}
	\mu  \left( B _{\beta ^N} ( x , W(x) ) \right)	& \leqslant & 2 \mu \left( B _{\beta ^N} ( x , W(x) ) \cap (\mathcal{G} _{N _0} ) _\xi \times \mathbb{R} \right)	\\
		& \leqslant &	2 \int _{\pi ^{ss} _\xi (\Sigma _{C _0 \beta ^N} ( \xi ,x ))}  \mu _{(\xi, z )} \left( B _{\beta ^N} ( x , W(x) ) \cap (\mathcal{G} _{N _0} ) _\xi \times \mathbb{R} \right) \, d (\nu \circ q _\xi ^{-1}) ( z ) \\
		& \leqslant &	2 \beta ^{N (\gamma _1 + \gamma _2 - 2 \varepsilon)}
\end{eqnarray*}
for all $ x \in \mathcal{G}  ' $ and $ N \geqslant N _1 $.
Therefore by Proposition \ref{prop:pwd_helper} we have
\[
	\liminf _{r \rightarrow 0} \frac{\log \mu \left( B _r ( x ,W(x) ) \right) }{\log r} \geqslant \gamma _1 + \gamma _2 - 2 \varepsilon
\]
for $ x \in \mathcal{G}  ' $.
As $ \nu ( \mathcal{G} ' ) > 1 - 2 \varepsilon $ and $ \varepsilon > 0 $ was arbitrary, the lower estimate is proved $ \nu $-a.s..

Now, we prove the upper bound.
Again, let $ \varepsilon > 0 $ be arbitrary.
For $ ( \xi ,x ) \in [0,1] ^2 $ and $ N _2 \in \mathbb{N} $ let $ \mathcal{F} _{(\xi, x ) , N _2} $ denote the set of those $ v \in [0,1] $ which satisfy
\begin{eqnarray}
	W ( v ) 	& = &	\ell ^{ss} _{(\xi, x , q _\xi (x) )} ( v ) , \label{eq:proof_prod_01}\\
 \lambda ^N (v) & \in &	[\beta ^{ N ( 1 + \varepsilon)} ,\beta ^{ N ( 1 - \varepsilon)} ] , \label{eq:proof_prod_02} \\
 \nu _{(x, q_\xi (x) )} ( B _{\beta ^N} ( v ) ) 	& \in &	[\beta ^{N (\gamma _1 + \varepsilon)} , \beta ^{N(\gamma _1 - \varepsilon)} ] , \label{eq:proof_prod_03} \\
   \nu ( I _N (v) )	& \in &	 [e ^{- N (  h _\nu  + \varepsilon )} , e ^{- N (  h _\nu - \varepsilon )}], \label{eq:proof_prod_04}  \\
  \nu _{ ( \xi , q _\xi (x) )} ( I _N ( v ) ) 	& \in &	 [e ^{- N (  h ^{ss} _\mu + \varepsilon )} , e ^{- N (  h ^{ss} _\mu - \varepsilon )}] \label{eq:proof_prod_05} \mbox{, and} \\
   | I _N (v) |	& \in &	 [ 0, \beta ^N ] \label{eq:proof_prod_055}  
\end{eqnarray}
for all $ N \geqslant N _2 $.
Then we have $ \lim _{N \rightarrow \infty} \int \nu _{(\xi, q _\xi (x))} ( \mathcal{F} _{ ( \xi , x ) , N} ) \,d \nu  ( \xi, x ) = 1 $ by Lemma \ref{lem:ss_conditional} in view of \eqref{eq:cond_on_graph}, Birkhoff ergodic theorem, Lemma \ref{lem:dim_cond}, \eqref{eq:SMB}, Proposition \ref{prop:h_ss_conv}, Proposition \ref{prop:log_I} and the fact that $ \exp (- \int \log \tau ' \, d \nu ) < \beta $.
By the monotone convergence theorem and Proposition \ref{prop:Gibbs_comparison_cond}, we have therefore $ \lim _{N \rightarrow \infty} \nu _{(\xi, q _\xi (x))} ( \mathcal{F} _{ ( \xi , x ) , N} ) = 1 $ for $ \nu ^\ext $-a.a. $ ( \xi, x) $.

Let such a $ ( \xi ,x ) $ be fixed and let $ N _2 $ be so large that $ \nu _{(\xi, q _\xi (x))} ( \mathcal{F} _{ ( \xi , x ) , N _2} )  > 1 - \varepsilon $.
By Lemma \ref{lem:Boreldensity} and Egorov's theorem, there are $ N _3 \geqslant N _2 $ and $  \mathcal{F} ' \subseteq \mathcal{F} _{ ( \xi , x ) , N _2}  $ with $  \nu _{(\xi, q _\xi (x))} ( \mathcal{F} ' ) > 1 - 2 \varepsilon $ such that
\begin{equation}
	\nu _{(\xi, q _\xi (x))} \left( B _{\beta ^N} (v) \cap \mathcal{F} _{ ( \xi , x ) , N _2} \right)  \geqslant	\frac{1}{2} \nu _{(\xi, q _\xi (x))} \left( B _{\beta ^N} (v) \right)	\label{eq:proof_prod_06}
\end{equation}
holds for all $ v \in \mathcal{F} ' $ and $ N \geqslant N _3 $.
Let a $ N \geqslant N _3 $ be fixed.
Then we choose $ ( u _i ) _{i=0} ^{\ell ^N - 1} $ so that $ I _N ( u _i ) \neq I _N ( u _j ) $ for $ i \neq j $ and that $ u _i \in \mathcal{F} _{ ( \xi , x ) , N _2} $ whenever $ I _N ( u _i ) \cap \mathcal{F} _{ ( \xi , x ) , N _2} \neq \emptyset $.
By \eqref{eq:proof_prod_03}, \eqref{eq:proof_prod_06} and \eqref{eq:proof_prod_05} we obtain
\begin{eqnarray}
	\beta ^{ N ( \gamma _1 + \varepsilon)} 	\quad  \leqslant \quad	\nu _{(\xi, q _\xi (x))} \left( B _{\beta ^N} (v) \right) 
		& \leqslant &	2 \, \nu _{(\xi, q _\xi (x))} \left( B _{\beta ^N} (v) \cap \mathcal{F} _{ ( \xi , x ) , N _2} \right) \nonumber \\
		& \leqslant &	 2 \sum _{i \in A _v} \nu  _{(\xi, q _\xi (x))} ( I _N ( u _i ) ) \quad \leqslant \quad \sharp A _v \cdot  e ^{- N (  h ^{ss} _\mu - \varepsilon )} 		\label{eq:proof_prod_07}
\end{eqnarray}
for all $ v \in \mathcal{F} ' $, where $ A _v :=  \{ i  : I _N (u_i) \cap B _{\beta ^N} (v) \cap \mathcal{F} _{ ( \xi , x ) , N _2} \neq \emptyset \} $ and $ \sharp A _v $ denotes the cardinality of $ A _v $.
On the other hand, by Proposition \ref{prop:Moss} and \eqref{eq:proof_prod_02} we have
\[
	\{ (w , W(w) ) : w \in I _N (u _i ) \}	\subseteq I _N (u _i ) \times B _{C _m \lambda ^N ( u _i)} ( W (u _i ) ) \subseteq I _N (u _i ) \times B _{C _m \beta ^{N ( 1 - \varepsilon)}} ( W (u _i ) ) 
\]
Furthermore, $ | W ( v ) - W ( u _i ) | = \left| \ell ^{ss} _{(\xi, x, W(x))} (v) - \ell ^{ss} _{(\xi, x, W(x))} (u _i) \right| \leqslant C _0 \beta ^N $
holds for all $ v \in \mathcal{F} ' $ and $ i \in A _v $ by \eqref{eq:proof_prod_01} since $ | v - u _i | \leqslant \beta ^N $ by the definition of $ A _v $.
Thus we have
\begin{equation}
	\{ (w , W(w) ) : w \in I _N (u _i ) \}	\subseteq B _{\tilde{C} \, \beta ^{N ( 1 - \varepsilon )}} ( v , W(v) )	\label{eq:proof_prod_08}
\end{equation}
for all $ v \in \mathcal{F} ' $ and $ i \in A _v $, where $ \tilde{C} := 1 + C _m + C _0  $.
Now, from \eqref{eq:proof_prod_04}, \eqref{eq:proof_prod_07} and \eqref{eq:proof_prod_08} follows
\begin{equation*}
	\mu \left( B _{\tilde{C} \, \beta ^{N ( 1 - \varepsilon )}} (v , W(v) ) \right) \geqslant \sum _{i \in A _v} \nu  ( I _N ( u _i ) ) \geqslant \sharp A _v \cdot e ^{- N (  h  _\nu + \varepsilon )} \geqslant \beta ^{ N ( \gamma _1 + \varepsilon)} \cdot e ^{- N (  h  _\nu - h ^{ss} _\mu + 2 \varepsilon )}
\end{equation*}
for all $ v \in \mathcal{F} ' $.

As the above consequence is true for all $ v \in \mathcal{F} ' $ and $ N \geqslant N _3 $, by Proposition \ref{prop:pwd_helper} we have
\begin{eqnarray*}
	\limsup _{r \rightarrow 0 } \frac{\log \mu ( B _r (v, W(v)) )}{\log r} 	&=&	 \limsup _{N \rightarrow \infty} \frac{\log \mu \left( B _{\tilde{C} \, \beta ^{N ( 1 - \varepsilon )}} ( v , W(v) ) \right)  }{\log \beta ^{N ( 1 - \varepsilon )}} \\
	& \leqslant &	\left( \gamma _1 + \varepsilon  + \frac{ h ^{ss} _\mu - h _\nu - 2 \varepsilon}{- \log \beta} \right) \frac{1}{1 - \varepsilon}
\end{eqnarray*}
for all $ v \in \mathcal{F} ' $.
As $ \nu _{(\xi, q _\xi (x))} ( \mathcal{F} ' ) \geqslant 1 - 2 \varepsilon $ and $ \varepsilon > 0 $ was arbitrary, letting $ \varepsilon \rightarrow 0 $ yields in view of Lemma \ref{lem:pre-YL} that
\[
	\limsup _{r \rightarrow 0 } \frac{\log \mu \left( B _r ( v , W(v) ) \right) }{\log r} \leqslant \gamma _1 + \gamma _2
\]
holds for $ \nu  _{(\xi, q _\xi (x) )} $-a.a. $ v $.
Finally, since this is true for $ \nu ^\ext $-a.a. $ ( \xi , x ) $, by Lemma \ref{lem:ss_conditional} we can conclude
\begin{eqnarray*}
	&&	\nu \left\{ x \in [0,1] : \limsup _{r \rightarrow 0 } \frac{\log \mu \left( B _r ( x , W(x) ) \right) }{\log r} \leqslant \gamma _1 + \gamma _2 \right\} \\
	&  =&	\int \nu _{(\xi, q _\xi (x) )} \left\{ v \in [0,1] : \limsup _{r \rightarrow 0 } \frac{\log \mu \left( B _r ( v , W(v) ) \right) }{\log r} \leqslant \gamma _1 + \gamma _2 \right\} \, d \nu (x) \quad = 1 .
\end{eqnarray*}
\end{proof}

\begin{proof}[Proof of Theorem \ref{thm:entropy_formula}]
This follows from Lemmas \ref{lem:pre-YL}, \ref{lem:dim_cond} and \ref{lem:mu_product}.
\end{proof}

\section{Ledrappier's lemma}	\label{sec:L}

In this section, let $ \nu \in \mathcal{P} ( [0,1] ) $ be an invariant measure, that is not necessarily a Gibbs measure.
Instead we suppose the dimension theoretic product structure of $ \mu $ which is always satisfied in case of Gibbs measures as Lemma \ref{lem:mu_product} together with Proposition \ref{prop:Gibbs_comparison} shows.
More precisely, suppose that $ \nu _{(\xi, q _\xi (x))} $ and $ \nu \circ q _\xi ^{-1} $ are for $ \nu ^- \otimes \nu $-a.a. $ ( \xi, x ) $ exact dimensional with constant dimensions, say,
\begin{equation}
	\gamma _1 := \dim _H ( \nu _{(\xi, q _\xi (x))} ) \quad \mbox{and} \quad  \gamma _2 := \dim _H ( \nu \circ q _\xi ^{-1} ) ,	\label{eq:gammas}
\end{equation}
and that $ \mu $ is also exact dimensional with
\begin{equation}
	\dim _H ( \mu ) =  \gamma _1 + \gamma _2 . \label{eq:sigma}
\end{equation}

%
%Similarly to \eqref{eq:disintegral_vert}, we consider the horizontal decomposition of the extended measure.
%Note that $ \nu = \nu ^\ext ([0,1] \times \, \cdot \, ) $.
%Thus there is a conditional distribution of $ \nu ^\ext $ with respect to %the horizontal partition $ \{ [0,1] \times \{ x \}  : x \in [0,1] \} $ %which can be written as product measures $  \nu ^- _x \otimes \delta _x $ %with $ \nu ^- _x \in \mathcal{P} ( [0,1] ) $ and satisfies
%\begin{equation}
%	\nu ^\ext = \int \nu _x ^- \otimes \delta _x \, d \nu ( x ).	%\label{eq:disintegral_horiz}
%\end{equation}

Recall that $ \Theta : [0,1] ^2 \rightarrow \mathbb{R} $ is defined in \ref{eq:Theta_Def}.
Now the lemma says the following.

\begin{lemma}[Ledrappier's Lemma]	\label{lem:L}
Under the above assumption we have
\[
\begin{cases}
	\dim _H ( \nu \circ q _\xi ^{-1} ) = 1 & \mbox{if } \dim _H ( \mu ) \geqslant 1 \\
	 \dim _H ( \nu _{(\xi, q _\xi (x))} )  = 0	 & \mbox{if } \dim _H ( \mu ) < 1 
\end{cases} ,
\]
whenever the distribution of $ \Theta (\, \cdot \, , x ) $ under $ \nu ^- $ has Hausdorff dimension $ 1 $ for $ \nu $-a.e. $ x \in [0,1 ] $.
\end{lemma}

\subsection{Outline of the proof of Lemma \ref{lem:L}}
As the procedure of the proof is a bit tricky, we sketch our approach fist.
We use the following convention.
\begin{itemize}
 \item Let $ B _r ( x,y ) := B _r ((x,y)) $.
 \item Let $ \| \cdot \|  $ denote the euclidean norm of $ \mathbb{R} ^2 $.
 \item Let $ A _\xi := \{ x \in [0,1] : ( \xi, x ) \in A \} $ for $ \xi \in [0,1] $ and $ A \subseteq [0,1] ^2 $.
 \item Let $ \mathbb{P} _x \in \mathcal{P} ( \mathbb{R} ) $ denote the distribution of $ \Theta (\, \cdot \, , x ) $ under $ \nu ^- $.
 \item Let $ \alpha \in ( 0,1 ) $ be the minimum of the Hölder exponents of $ g' $, $ \lambda ' $ and $ \log \gamma $.
 \item Let $ a := \min \{ 1 , \gamma _1 + \gamma _2 \} $.
 \item Let $ | \cdot | _\alpha $ and $ | \cdot | _{\mathrm{pw} , \alpha} $ denote the $ \alpha $-Hölder and the peacewise $ \alpha $-Hölder seminorm, respectively, i.e. for $ \varphi : [0,1] \rightarrow \mathbb{R} $ let
\[
	| \varphi | _\alpha :=  \sup _{x,\tilde{x} \in (0,1)} \frac{| \varphi (x) - \varphi ( \tilde{x} ) |}{|x - \tilde{x}| ^\alpha} \quad \mbox{and} \quad | \varphi | _{\mathrm{pw}, \alpha} := \max _{i \in \{ 0 , \ldots , \ell -1\}} \sup _{x,\tilde{x} \in I _i ^\circ} \frac{| \varphi (x) - \varphi ( \tilde{x} ) |}{|x - \tilde{x}| ^\alpha} .
\]
\end{itemize}

Suppose that we have $ \dim _H ( \mathbb{P} _x  ) = 1 $ for $ \nu  $-a.e. $ x \in [0,1 ]$.

Let $ \delta \in ( 0 , \gamma _1 + \gamma _2 ) $ and $ t \in ( \frac{1}{1 + \alpha} ,  1 ) $ be fixed.
In view of \eqref{eq:sigma} and the above assumption on $ \mathbb{P} _x $ together with Fubini's theorem, we can choose an $ E > 0 $ so large that the set
\begin{equation*}
	G := \left\{ ( \xi , x ) \in [0, 1] ^2 :  \begin{array}{l}
	\mathbb{P} _x \left( B _\eta \left( \Theta ( \xi, x ) \right)  \right) \leqslant E \,  \eta ^{1 - 2 \delta }	\mbox{ and}	\\
	\mu \left( B _\eta  (x,W(x))  \right) \leqslant E \, \eta ^{\gamma _1 + \gamma _2 - a \delta} 
	\end{array}	
	\mbox{ for } \forall \eta > 0 \right\}	\label{eq:Sigma}
\end{equation*}
has a positive $ \nu ^- \otimes \nu $ measure.
Then we consider
\begin{eqnarray*}
	b ( \xi , x , r , t )	& := &	\mu \left( G _\xi \times \mathbb{R} \cap \Sigma _r ( \xi , x ) \cap B _{r ^t} ( x , W(x) ) \right) \\
	& = &	\nu \left\{ v \in G _\xi :  \left| \ell ^{ss} _{\left( \xi ,v , W(v) \right)} (x) - W(x)  \right| \leqslant r \mbox{ and } \| (v, W(v)) - (x, W(x))  \| \leqslant r ^t  \right\} .
\end{eqnarray*}
In the following subsections, as \eqref{eq:b_lawer} and \eqref{eq:b_upper}, we shall prove that
\begin{equation}
	a ( 1  - 2 \delta ) + t ( \gamma _1 + \gamma _2 - a ) \leqslant	\limsup _{r \rightarrow 0} \frac{\log b ( \xi , x , r , t )}{\log r}	\leqslant t \gamma _1 + \gamma _2 + 3 \delta	\label{eq:b_estimate}
\end{equation}
holds on some set of $ ( \xi , x ) $ with positive $ \nu ^- \otimes \nu $ measure.

Once this has been proved, we are able to conclude the proof as follows.

\begin{proof}[Proof of Lemma \ref{lem:L}]
Assume first that $ \gamma _1 + \gamma _2 \geqslant 1 $.
Then the inequality \eqref{eq:b_estimate} implies
\begin{equation*}
	1 + t ( \gamma _1 + \gamma _2 - 1) \leqslant t \gamma _1 + \gamma _2 
\end{equation*}
for a $ t \in  ( \frac{1}{1 + \alpha} ,  1 ) $ since $  \delta \in ( 0 , \gamma _1 + \gamma _2 ) $ was arbitrary.
This can be rewritten as $ 1 - \gamma _2 \leqslant t (1 - \gamma _2 ) $.
The last expression is only possible in case $ \gamma _2 = 1 $ since $ \gamma _2 = \dim _H (\mu \circ q _\xi ^{-1}) \in [0,1] $.

Now, assume that $ \gamma _1 + \gamma _2 < 1 $.
In this case the inequality \eqref{eq:b_estimate} implies
\begin{equation*}
	 \gamma _1 + \gamma _2  \leqslant t \gamma _1 + \gamma _2 ,
\end{equation*}
for a $ t \in  ( \frac{1}{1 + \alpha} ,  1 ) $, or equivalently, $ \gamma _1 \leqslant t \gamma _1 $.
Thus $ \gamma _1 = 0 $.
\end{proof}

\subsection{Lower estimate}

The goal of this subsection is to prove the lower estimate of \eqref{eq:b_estimate}.

Let $ x \in [0, 1] $.
By Fubini's theorem we have
\begin{eqnarray*}
	&  &	\int _{ \{ \xi \in [0,1] : ( \xi, x ) \in G \} }  b ( \xi , x , r , t ) \, d \nu ^- ( \xi ) \\
		& = &	\int _{ \{ ( v , W(v) ) \in B _{r ^t} (x, W(x)) \} } \nu ^- \left\{ \xi \in [0,1] : v , x \in G _\xi \mbox{ and } \left| \ell ^{ss} _{\left( \xi ,v , W(v) \right)} (x) - W(x)  \right| \leqslant r  \right\} \, d \nu  (v).
\end{eqnarray*}

Let $ L _{(x, y)} ^\theta :[0,1] \rightarrow \mathbb{R} $ be the linear function through $ ( x , y ) $ with slope $ \theta $, i.e.
\[
	L _{(x, y)} ^\theta ( v ) := \theta ( v - x ) + y .
\]
We show now several properties of this projection function.

\begin{proposition}	\label{prop:L_and_l}
Let $ Y > 0 $.
Then there is a $ C _Y > 0 $ such that
\[
	\left|  L _{\left( x , y \right)} ^{X _3 ( \xi , x , y) } (v) - \ell ^{ss} _{(\xi, x, y)} ( v ) \right| \leqslant C _Y | x - v | ^{1 + \alpha}
\]
for all $ ( \xi ,x ) \in [0,1] ^2 $, $ v \in [0,1] $ and $ y \in [-Y , Y ] $.
\end{proposition}

\begin{proof}
Since $ \left(  L _{\left( x , y \right)} ^{X _3 ( \xi , x , y ) } \right) ' (v) = X _3 ( \xi , x , y ) $ and $ \left( \ell ^{ss} _{(\xi, x, y)} \right) ' ( v ) = X _3 \left( \xi , v , \ell ^{ss} _{(\xi, x, y)} (v ) \right) $, it suffices to show
\[
	\sup _{(\xi, x , y) \in[0,1] ^2 \times \mathbb{R}}\left| v \mapsto  X _3 \left( \xi , v , \ell ^{ss} _{(\xi, x, y)} (v ) \right) \right| _\alpha < \infty .
\]
Recall that by \eqref{eq:X3_def} we have
\begin{eqnarray}
	&&	 X _3 \left( \xi , v , \ell ^{ss} _{(\xi, x, y)} (v ) \right) 	\label{eq:X3_separate} \\
	 & = &	 - \sum _{n=1} ^\infty \gamma ^n \left( \rho _{[\xi] _n} (v) \right) \cdot F ^{n-1} _{(\xi, x)} \left( \ell ^{ss} _{(\xi, x, y)} (v ) \right) \cdot \lambda ' \left( \rho _{[\xi] _n} (v) \right) - \sum _{n=1} ^\infty \gamma ^n \left( \rho _{[\xi] _n} (v) \right) \cdot g' \left( \rho _{[\xi] _n} (v) \right) . \nonumber
\end{eqnarray}

First, we consider the latter sum in the above expression.
Since there is a $ C _1 > 0 $ such that
\begin{eqnarray*}
		\log \gamma ^n \left( \rho _{[\xi] _n} (v) \right) - \log \gamma ^n \left( \rho _{[\xi] _n} (\tilde{v}) \right)	&= 	&	\sum _{j=1} ^n \log \gamma  \left( \rho _{[\xi] _j} (v) \right) - \log \gamma  \left( \rho _{[\xi] _j} (\tilde{v}) \right) \\
	& \leqslant & | \log \gamma | _{\mathrm{pw}, \alpha} \sum _{j=1} ^n | \rho _{[\xi] _j} (v) - \rho _{[\xi] _j} (\tilde{v}) | ^\alpha \\
	& \leqslant &	 C _1 \, | v - \tilde{v} | ^\alpha
\end{eqnarray*}
holds for all $ \xi, v , \tilde{v} \in [0,1]$ and $ n \in \mathbb{N}$, we have
\begin{eqnarray}
	 \gamma ^n \left( \rho _{[\xi] _n} (v) \right) - \gamma ^n \left( \rho _{[\xi] _n} (\tilde{v}) \right)	& = & \gamma ^n \left( \rho _{[\xi] _n} (v) \right) \left( 1 - \frac{\gamma ^n \left( \rho _{[\xi] _n  } (\tilde{v}) \right)}{\gamma ^n \left( \rho _{[\xi] _n} (v) \right)} \right) \nonumber \\
	 & \leqslant & \gamma ^n \left( \rho _{[\xi] _n} (v) \right) \left( e^{\left| \log \gamma ^n \left( \rho _{[\xi] _n} (v) \right) - \log \gamma ^n \left( \rho _{[\xi] _n} (\tilde{v}) \right) \right|} - 1 \right) \nonumber \\
	 & \leqslant & C _2 \, \| \gamma \| _\infty ^n \, | v - \tilde{v} | ^\alpha	\label{eq:gamma_distorsion}
\end{eqnarray}
for $ C _2 := (e ^{C _1} -1 ) C _1 ^{-1} $, where we used the inequality $ e ^x - 1 \leqslant ( e ^M - 1 ) M ^{-1} x  $ for $ 0 \leqslant x \leqslant M $.
Further, from
\begin{eqnarray*}
	&& \gamma ^n \left( \rho _{[\xi] _n} (v) \right) \, g' \left( \rho _{[\xi] _n} (v) \right) - \gamma ^n \left( \rho _{[\xi] _n} (\tilde{v}) \right) \, g' \left( \rho _{[\xi] _n} (\tilde{v}) \right)	\\
	 & \leqslant &	 \gamma ^n \left( \rho _{[\xi] _n} (v) \right) \left( g' \left( \rho _{[\xi] _n} (v) \right) -   g' \left( \rho _{[\xi] _n} (\tilde{v}) \right)	\right) +  \left( \gamma ^n \left( \rho _{[\xi] _n} (v) \right) - \gamma ^n \left( \rho _{[\xi] _n} (\tilde{v}) \right) \right) g' \left( \rho _{[\xi] _n} (\tilde{v}) \right)	\\
	 & \leqslant &	( | g' | _{\mathrm{pw}, \alpha} + C _2 ) \, \| \gamma \| _\infty ^n \, | v - \tilde{v} | ^\alpha
\end{eqnarray*}
follows
\[
	\left| v \mapsto  \sum _{n=1} ^\infty \gamma ^n \left( \rho _{[\xi] _n} (v) \right) \cdot g' \left( \rho _{[\xi] _n} (v) \right) \right| _\alpha  \leqslant \frac{ | g' | _{\mathrm{pw}, \alpha} + C _2 }{1 - \| \gamma \| _\infty} . 
\]

Now, we estimate the first sum of \eqref{eq:X3_separate}.
Observe that by \eqref{eq:Fn} we can reformulate it as
\begin{eqnarray*}
	&&	\sum _{n=1} ^\infty \gamma ^n \left( \rho _{[\xi] _n} (v) \right) \cdot F ^{n-1} _{(\xi, x)} \left( \ell ^{ss} _{(\xi, x, y)} (v ) \right) \cdot \lambda ' \left( \rho _{[\xi] _n} (v) \right) \\
	& = &	  \ell ^{ss} _{(\xi, x, y)} (v ) \sum _{n=1} ^\infty \gamma ^n \left( \rho _{[\xi] _n} (v) \right) \, \lambda ^{n-1} \left( \rho _{[\xi] _{n-1}} (x) \right) \, \lambda ' \left( \rho _{[\xi] _n} (v) \right)  \\
	&&\qquad + \sum _{n=1} ^\infty \gamma ^n \left( \rho _{[\xi] _n} (v) \right) \, W _{n-1} \left( \rho _{[\xi] _{n-1}} (x) \right) \, \lambda ' \left( \rho _{[\xi] _n} (v) \right) .	
\end{eqnarray*}
Clearly, $  \left| \left( \ell ^{ss} _{(\xi, x, y)} \right) ' (v) \right| \leqslant \| X _3 \| _\infty $ holds for all $ v \in [0,1] $ and $ ( \xi, x, y) \in [0,1]^2\times \mathbb{R} $.
Thus, by \eqref{eq:gamma_distorsion} and since $ \| W _{n-1} \| _\infty $, $ \| \lambda ^{n-1} \| _\infty $ and $ | \lambda ' | _{\mathrm{pw}, \alpha} $ are uniformly bounded with respect to $ n $, we can obtain a bounded of the $\alpha$-Hölder seminorm  of the above expression by applying the triangle inequality several times.
\end{proof}

The next simple geometrical result originates in a work of Marstrand \cite{Marstrand54}.

\begin{proposition}	\label{prop:Marstrand}
For $ M > 0 $ there is a $ D _M > 0 $ such that
\[
	\mathrm{diam} \left\{ \theta \in [ -M , M ] : \left| L _{\left( x , y \right)} ^\theta (x') - y' \right| \leqslant r \right\} \leqslant \frac{D _M \, r}{\| (x, y) - (x', y') \|}
\]
for all $ ( x , y ) $, $ (x', y') \in [0,1] \times \mathbb{R} $ and $ r > 0 $.
\end{proposition}

\begin{proof}
The proof is elementary.
\end{proof}

We continue the proof of the lower estimate.
Observe that we have
\begin{eqnarray*}
	\left| L _{\left(v , W(v) \right)} ^{\Theta ( \xi , v ) } (x) - W (x) \right|	&  \leqslant &	\left|  L _{\left( v , W(v) \right)} ^{\Theta ( \xi , v ) } (x) - \ell ^{ss} _{\left( \xi, v, W(v) \right)} ( x ) \right| + \left|  \ell ^{ss} _{\left( \xi, v, W(v) \right)} ( x ) - W ( x ) \right| \\
		& \leqslant &	C _{\| W \| _\infty} |v - x| ^{1 + \alpha} + r \\
	& \leqslant &	C _{\| W \| _\infty} \left \| \left( v , W(v) \right) - \left( x , W (x) \right) \right \| ^{1 + \alpha} + r \\
	& \leqslant &	C _{\| W \| _\infty} (r ^t) ^{1 + \alpha} + r	\\
	& \leqslant	&	(C _{\| W \| _\infty} + 1) \, r
\end{eqnarray*}
by Proposition \ref{prop:L_and_l}, whenever both $ (v, W(x)) \in B _{r ^t} \left( x , W(x) \right) $ and $ \left|  \ell ^{ss} _{\left( \xi, v, W(v) \right)} ( x ) - W ( x ) \right| \leqslant r $ are satisfied.
Thus we have
\begin{eqnarray*}
	&& \left\{ \xi \in [0,1] : v , x \in G _\xi \mbox{ and } \left| \ell ^{ss} _{\left( \xi ,v , W(v) \right)} (x) - W(x)  \right| \leqslant r \right\} \\
	& \subseteq & \left\{ \xi \in [0,1] : v , x \in G _\xi \mbox{ and } \left| L _{\left( v , W(v) \right)} ^{\Theta ( \xi , v )}(x) - W(x)  \right| \leqslant ( C _{\| W \| _\infty} + 1 ) \, r \right\} \\
		& \subseteq & \left\{ \xi \in [0, 1] :
		\begin{array}{l}
				\mathbb{P} _v \left( B _\eta \left( \Theta ( \xi , v ) \right) \right) \leqslant E \, \eta ^{1 - 2 \delta } \mbox{ for } \forall \eta > 0 \mbox{, and} \\
				 \left| L _{\left(v , W(v) \right)} ^{\Theta ( \xi , v ) } ( x ) - W (x) \right| \leqslant (C _{\| W \| _\infty} + 1 ) \, r
		\end{array}	\right\}
\end{eqnarray*}
for $ r  > 0 $ and $ v , x  \in [0,1] $ such that $ ( v , W(v) ) \in B _{r ^t} ( x , W(x) ) $.
Recall that $ a := \{ 1 , \gamma _1 + \gamma _2 \} \in [0,1] $.
For those $ r,  v , x $ we have therefore
\begin{eqnarray*}
	&&	\nu ^- \left\{ \xi \in [0,1] : v , x \in G _\xi \mbox{ and } \left| \ell ^{ss} _{\left( \xi ,v , W(v) \right)} (x) - W(x)  \right| \leqslant r \right\} \\
	& \leqslant &	\nu ^- \left\{ \xi \in [0, 1] :
		\begin{array}{l}
				\mathbb{P} _v \left( B _\eta \left( \Theta ( \xi , v ) \right) \right) \leqslant E \, \eta ^{1 - 2 \delta } \mbox{ for } \forall \eta > 0 \mbox{ and} \\
				 \left| L _{\left(v , W(v) \right)} ^{\Theta ( \xi , v ) } (x) - W (x) \right| \leqslant (C _{\| W \| _\infty} + 1 ) \, r
		\end{array}	\right\} \\
	& = &	\mathbb{P} _v \left\{ \theta \in \left[ - \| \Theta \| _\infty , \| \Theta \| _\infty \right] :
		\begin{array}{l}
				\mathbb{P} _v \left( B _\eta ( \theta )  \right) \leqslant  E \, \eta ^{1 - 2 \delta } \mbox{ for } \forall \eta > 0 \mbox{ and} \\
				 \left| L _{\left(v , W(v) \right)} ^\theta (x) - W (x) \right| \leqslant (C _{\| W \| _\infty} + 1 ) \, r
		\end{array}	\right\} \\
		& \leqslant &		\mathbb{P} _v \left\{ \cdots \right\} ^a \\
	& \leqslant &	E ^a \, \left( \frac{ D _{\| \Theta \| _\infty} (C _{\| W \| _\infty} + 1 ) \, r}{ \left\| \left( v , W(v) \right) - \left( x, W(x) \right) \right\|} \right) ^{a  (1 - 2 \delta)} \quad  =: 	 \frac{ Q \, r ^{a  ( 1 - 2 \delta)}}{ \left\| \left( v , W(v) \right) - \left( x, W(x) \right) \right\| ^{a ( 1 - 2 \delta)}} 
\end{eqnarray*}
by Proposition \ref{prop:Marstrand}.
Now, there is a $ Q' >0 $ such that for all $ r > 0 $ and $ ( \xi , x ) \in G $ we have
\begin{eqnarray*}
	&&	\int _{ \{ \xi \in [0,1] : ( \xi, x ) \in G \} }  b ( \xi , x , r , t ) \, d \nu ^- ( \xi ) 	\\
	& \leqslant &	Q \, r ^{a ( 1 - 2 \delta)} \int _{B _{r ^t} (x, W(x)) } \frac{ d \mu (v, y)}{ \left\| \left( v , y \right) - \left( x, W(x) \right) \right\| ^{a ( 1 - 2 \delta)}}  \\
		& \leqslant &	Q \, r ^{a ( 1 - 2 \delta)} \sum _{n \geqslant \lfloor - t \log r \rfloor} \int _{B _{e ^{-n}} \left( x , W (x) \right) \setminus B  _{ e ^{-n-1}  } \left( x , W(x) \right) } \frac{ d \mu ( v, y )}{ \left\| \left( v , y \right) - \left( x, W(x) \right) \right\| ^{a (1 - 2 \delta )}}  \\
		& \leqslant &	Q  \, r ^{a ( 1 - 2 \delta)} \sum _{n \geqslant \lfloor - t \log r \rfloor}  \frac{\mu \left( B _{e ^{-n}} \left( x , W(x) \right)\right)}{ e ^{(-n-1) a (1 - 2 \delta)}} \\
		& \leqslant &	Q  \,  r ^{a ( 1 - 2 \delta)} \sum _{n \geqslant \lfloor - t \log r \rfloor}  \frac{ E \,  (e ^{-n}) ^{\gamma _1 + \gamma _2 - a \delta} }{ e ^{(-n-1) a (1 - 2 \delta)}} \\
		& \leqslant &	Q' \, r ^{a ( 1 - 2 \delta ) + t ( \gamma _1 + \gamma _2 - a + a \delta )} \\
		& \leqslant &	Q' \, r ^{a ( 1 - 2 \delta ) + t ( \gamma _1 + \gamma _2 - a )}.
\end{eqnarray*}
As the integration by $ \nu $ yields
\[
	\int _G  b ( \xi , x , r , t ) \, d ( \nu ^- \otimes \nu ) ( \xi , x ) \leqslant Q' \,  r ^{a ( 1 - 2 \delta ) + t ( \gamma _1 + \gamma _2 - a )} ,
\]
we can derive by Fatou lemma
\[
	\int _G \liminf _{r \rightarrow 0} \frac{ b ( \xi , x , r , t )}{ r ^{a ( 1 - 2 \delta ) + t ( \gamma _1 + \gamma _2 - a )} } \, d ( \nu ^- \otimes \nu ) ( \xi , x ) \leqslant Q' ,
\]
which in turn implies
\begin{eqnarray}
	a ( 1 - 2 \delta ) + t ( \gamma _1 + \gamma _2 - a )	\leqslant	\limsup _{r \rightarrow 0} \frac{\log b ( \xi ,x  , r , t )}{\log r} 	\label{eq:b_lawer}
\end{eqnarray}
for $ \nu ^- \otimes \nu $-a.a. $ ( \xi ,x ) \in G $.

\subsection{Upper estimate}
The goal of this subsection is to prove the upper estimate of \eqref{eq:b_estimate}.

We can chose $ \varepsilon _1 > 0 $ so that both
\[
	G ' := \left\{ ( \xi , x ) \in G :  
	\begin{array}{l}
	\mu \left( B _r \left( x, W(x) \right)  \cap G _\xi \times \mathbb{R} \right)	\geqslant  r ^{\gamma _1 + \gamma _2 + \delta} \mbox{ and} \\
	 \nu _{\left( \xi, q _\xi (x) \right)} \left( B _r (x) \right) \leqslant r ^{ \gamma _1 - \delta } 
	\end{array}
	\mbox{ for } \forall r \in ( 0 , \varepsilon _1 ) \right\}
\]
and
\[
	G '' := \left\{ ( \xi ,x ) \in G ' : 
	\begin{array}{r}
	\nu _{\left( \xi, q _\xi (x) \right)} \left( B _r (x)  \cap G ' _\xi  \cap ( W - \ell ^{ss} _{(\xi, x, W(x))} ) ^{-1} (0)  \right) \geqslant r ^{  \gamma _1 + \delta } \\
	\mbox{ for } \forall r  \in ( 0 , \varepsilon _1 ^{\frac{1}{1 + \alpha}} )
	\end{array}
	 \right\}
\]
have positive $ \nu ^- \otimes \nu $ measures in view of \eqref{eq:gammas}, \eqref{eq:sigma} and Lemma \ref{lem:Boreldensity} together with the fact that by \eqref{eq:cond_on_graph} we have
\[
	\nu ^- \otimes \nu \left\{ (\xi, x) \in [0,1] ^2 : \nu _{\left( \xi, q _\xi (x) \right)} \left( ( W - \ell ^{ss} _{(\xi, x, W(x))} ) ^{-1} (0)  \right) = 1 \right\} = 1 .
\]

Let $ ( \xi , x ) \in G '' $ and $ r \in ( 0 , \varepsilon _1 ) $ be fixed.

By Vitali covering theorem there are $ N _r \in \mathbb{N}  $ and $ x _1 , \ldots , x _{N _r} \in G ' _\xi \cap B _{(r/2) ^t} (x) \cap ( W - \ell ^{ss} _{(\xi, x, W(x))} ) ^{-1} (0)  $ such that
\[
	G ' _\xi \cap B _{(r/2) ^t} (x) \cap ( W - \ell ^{ss} _{(\xi, x, W(x))} ) ^{-1} (0)  \subseteq	\bigcup _{i=1} ^{N _r} B _r ( x  _i ) 
\]
with $ B _{r/3} (x_1) , \ldots , B _{r/3} (x _{N _r} ) $ being disjoint.
As we have $ \nu _{\left( x , q _\xi (x ) \right)} = \nu _{\left( x _i , q _\xi (x _i ) \right)} $ for $ i = 1 , \ldots , N _r $, taking the measure $ \nu _{\left( x , q _\xi (x ) \right)} $ of both sides of the above inclusion yields
\[
	\left( \frac{r}{2} \right) ^{t ( \gamma _1 + \delta)} \leqslant \sum _{i=1} ^{N _r} r ^{\gamma _1 - \delta} ,
\]
i.e.
\begin{equation}
	N _r \geqslant 2 ^{-t( \gamma _1 + \delta)}  r ^{(t-1)\gamma _1 + 2 \delta} .	\label{eq:onehand}
\end{equation}

On the other hand, by Proposition \ref{prop:l_parallel} there are $ c \in ( 0, 1/3 ) $ and $ \varepsilon _2 \in ( 0, 1 ) $ depending on $ t $ and $ \| X _3 \| _\infty $, such that $
	B _{c r} \left( u , \ell ^{ss} _{(\xi, x, W(x))} (u) \right) \subseteq \Sigma _r ( \xi , x ) \cap B _{r ^t} ( x , W(x) )  $ holds for all $ r \in ( 0, \varepsilon _2 ) $ and $ u \in B ( x , (r/2) ^t  ) $.
In particular, we have
\begin{equation}
	B _{c \, r } ( x _i , W(x _i ) ) \subseteq \Sigma _r ( \xi, x ) \cap B _{r ^t} ( x , W(x) )	\label{eq:B_small}
\end{equation}
for $ r \in ( 0, \varepsilon _2 ) $ and $ i = 1 , \ldots , N _r $.
As the balls $ B _{c r} \left( x _1 , W ( x _1 ) \right) , \ldots , B _{c r} \left( x _{N _r} , W ( x_{N _r} ) \right) $ are disjoint, from \eqref{eq:onehand} and \eqref{eq:B_small} follows
\begin{eqnarray*}
	b ( \xi ,x , r , t )	&  =& \mu \left( G _\xi \times \mathbb{R} \cap \Sigma _r ( \xi , x ) \cap B _{r ^t} ( x , W(x) ) \right)\\
	& \geqslant	 &	\mu \left( \bigcup _{i=1} ^{N _r} B _{c r} \left( x _i , W ( x_i ) \right) \cap G _\xi \times \mathbb{R} \right) \\
	& \geqslant &  N _r  (c r) ^{\gamma _1 + \gamma _2 + \delta} \\
	& \geqslant &	c ^{\gamma _1 + \gamma _2 + \delta} 2 ^{-t( \gamma _1 + \delta)}  r ^{(t-1) \gamma _1 + \gamma _1 + \gamma _2 + 3 \delta } \\
		& := &	Q''' \,  r ^{t \gamma _1 + \gamma _2 + 3 \delta } .
\end{eqnarray*}
Consequently, we have
\begin{equation}
	\limsup _{r \rightarrow 0} \frac{\log b (\xi , x , r , t )}{\log r} \leqslant t \gamma _1 + \gamma _2 + 3 \delta .	\label{eq:b_upper}
\end{equation}
for all $ (\xi ,x ) \in G '' $.

\section{Hausdorff dimension of $ \Theta (\cdot , x )$}	\label{sec:Tsujii}

The assumption of Theorems \ref{thm:mu_dim} and \ref{thm:main_theorem} that the distribution of $ \Theta ( \cdot , x ) $ under $ \nu ^- $ has Hausdorff dimension $ 1 $ for $ \nu $-a.a. $ x \in [0,1] $ is generally of course not true, see also Remark \ref{rem:degenerate}.
We study the condition for a simple case that is relevant to Theorem \ref{thm:example} and \ref{thm:example2}.
%
%For example, in the degenerate case the results are not true, i.e. in case the Weierstrass-type function $ W $ is differentiable, e.g. if $ g $ is constantly zero (see Remark \ref{rem:degenerate}, and also \cite{Bedford89} for some details).
For a given probability vector $ \bp = ( p_i ) _{i=0} ^{\ell -1} $ we define the Bernoulli measure $ \nu _\bp \in \mathcal{P} ( [0,1] ) $ as the $ \tau $-invariant measure determined by
\[
	\nu _\bp ( I _N (x) ) := \prod _{i=0} ^N p _{k (\tau ^i(x))} 
\]
for all $ x \in [0,1] $ and $ N \in \mathbb{N} $.
Clearly, Bernoulli measures are Gibbs measures.
In addition, they satisfy $ \nu _\bp ^- = \nu _\bp $ and $ \nu _\bp ^\ext = \nu _\bp \otimes \nu _\bp $.

Henceforth we consider only the Bernoulli measure $ \nu _\bp $, i.e. we consider the distribution of $ \Theta ( x , \cdot ) $ under $ \nu _\bp $ for $ \nu _\bp $-a.a. $ x \in [0,1] $.
In fact, under the assumption of Theorem \ref{thm:example} or \ref{thm:example2} the equilibrium measure $ \nu _{\tau, \lambda} $ associated with the Bowen equation \eqref{eq:Pressure} is a Bernoulli measure as the following lemma says.

\begin{lemma}	\label{lem:bernoulli}
Under the assumption of Theorem \ref{thm:example} or \ref{thm:example2}, $ \nu _{\tau, \lambda} $ is the Bernoulli measure with $ \bp = ( | I _i | ^{s( \tau , \lambda)}  \gamma _i ^{-1} ) _{i=0} ^{\ell -1}$, i.e. the unique $ \tau $-invariant measure satisfying $
	\nu _{\tau, \lambda} ( I _i ) =  | I _i | ^{s( \tau , \lambda)}  \gamma _i ^{-1} $ for $ i \in \{ 0 , \ldots , \ell -1 \} $, where $ \gamma _i := | I _i | \lambda _i $ or $ \gamma _i := | I _i | ^{1 - \theta} $ in case of Theorem \ref{thm:example} or \ref{thm:example2}, respectively.
	
Furthermore, in the latter case we have $ s ( \tau ,\lambda ) = 2 - \theta $, and $ \nu _\bp $ is the Lebesgue measure.
\end{lemma}

\begin{proof}
By the canonical coding of $ [0,1] $ w.r.t. $ ( I _i ) _{i=0} ^{\ell -1} $ we can apply the variational principle for the one-sided shift space with $ \ell $ symbols.
Since the pull-back of the potential $ (1-s) \log \tau ' + \log \lambda  = s \log | I _{k(\cdot)} | - \log \gamma _{k(\cdot)} $ depends only on the first symbol, the equilibrium measure on the shift space as well as the corresponding one on $ [0,1] $ are both Bernoulli measures.
We can also calculate the topological pressure
\[
	P ( ( 1 - s  ) \log \tau ' + \log \lambda ) = \log \left( \sum _{i=0} ^{\ell - 1} | I _i | ^{s( \tau , \lambda)}  \gamma _i ^{-1}  \right) ,
\]
which gives the parameter of the Bernoulli measure.
We refer to e.g. \cite[Chapter 3]{Barreira08} for the terminology we used and some related observations.

Finally, under the assumption of Theorem \ref{thm:example2} the Bowen equation is $ P ( ( 1 - s - \theta ) \log \tau ' )  = 0 $.
Thus $ 1 - s - \theta  = -1 $ and the equilibrium state $ \nu _{\tau, \lambda} $ is the Lebesgue measure.
\end{proof}

\subsection{Case of self-similar measure}

Here we prove Theorem \ref{thm:example}.
Suppose $ \ell = 2 $ and $ \tau $ is piecewise linear.
Suppose also that $ \lambda $ and $ g $ satisfy $ \lambda (x) := \gamma _{k(x)} \, | I _{k(x)} | $ and $  g ' (x) =: a _{k(x)} $ for given constants $ \gamma _0, \gamma _1 \in (0,1) $ and $ a _0, a_1 \in \mathbb{R} $.

Observe that under these conditions $ \Theta $ does not depend on $ x $, i.e. we can write $ \tilde{\Theta} := \Theta ( \cdot , x ) $ for all $ x $.
We consider the parametrisation $ t \mapsto t \lambda $ and thus $
	W _{\tau , t \lambda} ( x ) := \sum _{n=0} ^\infty t ^n \lambda ^n (x) g ( \tau ^n (x) ) $  for those $ t \in ( 0 , \infty) $ which satisfy $ t \lambda _i \in ( | I _i | , 1 ) $ for $ i = 0 , 1 $.
Correspondingly, let $ \tilde{\Theta} _t $ denote the '$ \tilde{\Theta}  $ with respect to the parameter $ t $, i.e. $
	\tilde{\Theta} _t ( \xi ) = - \sum _{n = 1} ^\infty t ^{-n} \gamma ^n ( \xi ) a _{k(\xi)} $, where $ \gamma ( \xi ) := \gamma _{k(\xi)} $.

\begin{proposition}	\label{prop:Hochman}
Under the above assumption, if $ \gamma _0 a _0 \neq \gamma _1 a _1 $, there is a set $ E \subset \mathbb{R}$ of Hausdorff dimension $ 0 $ such that the distribution of $ \tilde{\Theta} _t $ under $ \nu _\bp $ has Hausdorff dimension $ 1 $ for any probability vector $ \bp $ and $ t \in  ( \max \{ \gamma _0 , \gamma _1 \} , \infty) \setminus E $ whenever
\begin{equation}
	h _{\nu _\bp} \geqslant - \int \log \left(  t ^{-1} \gamma _{k(\xi)} \right) \, d \nu _\bp (\xi) . \label{eq:Hochman_cond}
\end{equation}
\end{proposition}

\begin{proof}
We consider the parametrisation $ s =  \max \{ \gamma _0 , \gamma _1 \} \cdot t ^{-1} $.
Observe that, for $ s \in (- 1 , 1) \setminus \{ 0 \} $, the distribution of $ \tilde{\Theta} _{\max \{ \gamma _0 , \gamma _1 \} \cdot s ^{-1}} $ under $ \nu _\bp $ is a self-similar measure with respect to the probability vector $ \bp $ and the IFS $ \Phi _s := \{ \gamma _i (s) \cdot (x - a _i) : i = 0,1 \} $ in the sense of \cite{Hochman14}, where
\[
	\gamma _i (s) := \frac{\gamma _i \cdot s}{\max \{ \gamma _0 , \gamma _1 \}} .
\]
As the separation condition on $ \Phi _s $ of \cite[Theorem 1.8]{Hochman14} is satisfied due to the assumption $ \gamma _0 a _0 \neq \gamma _1 a _1 $, by \cite[Theorem 1.7]{Hochman14} there is a $ \tilde{E} \subset ( -1 , 1 ) \setminus \{ 0 \} $ with Hausdorff dimension $ 0 $ such that the distribution of $ \tilde{\Theta} _{\max \{ \gamma _0 , \gamma _1 \} \cdot s ^{-1}} $ under $ \nu _\bp $ has Hausdorff dimension $ \min \{ 1 , \frac{h _{\nu _\bp} }{- \int \log \left( t ^{-1} \gamma _{k(\xi)} \right) \, d \nu _\bp (\xi)  }\} $ for any probability vector $ \bp $ and $ s \in (-1,1) \setminus \tilde{E} $.
As a locally bi-Lipshitz continuous transformation preserves the Hausdorff dimension, the claim is satisfied by letting $ E := \{ t \in ( \max \{ \gamma _0 , \gamma _1 \}  , \infty ) : \max \{ \gamma _0 , \gamma _1 \} \cdot s ^{-1} \in \tilde{E}  \} $.
\end{proof}

\begin{proof}[Proof of Theorem \ref{thm:example}]
Let $ E \subset \mathbb{R} $ be the set of Proposition \ref{prop:Hochman}.
Observe that the function $ W _{\tau  , t \lambda } $ satisfies $ ( \tau ' ) ^{-1} < t \lambda  < 1 $, as $  t \in \left( \max \{ \gamma _0 , \gamma _1 \} , \infty \right) $.
Thus by Theorem \ref{thm:main_theorem} and Lemma \ref{lem:bernoulli} we have $ \dim _H ( \mathrm{graph} ( W _{\tau, t \lambda } ) ) = \dim _B ( \mathrm{graph} ( W _{\tau, t \lambda } ) ) = s ( \tau , t \lambda ) $ for all $ t \in \left( \max \{ \gamma _0 , \gamma _1 \} , \infty \right)  \setminus E $ whenever the condition \eqref{eq:Hochman_cond} is satisfied.
Now we shall check it for each $ t \in \left( \max\{ \gamma _0 , \gamma _1 \} , \min  \{ \frac{\gamma _0}{\sqrt{| I _0 |}} , \frac{\gamma _1}{\sqrt{| I _1 |}}  \} \right] $.
Recall that $ P (( 1 - s ( \tau , t \lambda ) ) \log \tau ' + \log ( t \lambda ) ) = 0 $ by the definition.
In view of $ ( \lambda \tau ' ) ^{-1} = \gamma  < t $ we have $ P ( - s ( \tau , t \lambda ) \log \tau ' ) \leqslant P (  - s ( \tau , t \lambda ) \log \tau ' + \log ( \tau '\, t \lambda ) )  = 0 $, which implies $ s ( \tau , t \lambda ) \geqslant 1 $.
Hence by Proposition \ref{prop:log_I}, from the equilibrium expression (or from \eqref{eq:s_formula}) follows
\begin{eqnarray*}
	h _{\nu _{\nu _{\tau, t \lambda}}} &  =&	 ( s ( \tau , t \lambda) - 1 ) \int \log \tau '  d \nu _{\tau, t \lambda} - \int \log \left( t \lambda \right)\, d \nu _{\tau, t \lambda} \\
	& \geqslant &	 - \int \log \left( t \lambda \right)\, d \nu _{\tau, t \lambda} \\
	& = &	 - \int \log \left( t ^{-1} \gamma  \right)  \, d \nu _{\tau, t \lambda} + \int \log \left( \gamma / ( t ^2 \, \lambda ) \right) d \nu _{\tau, t \lambda} \\
	&  =& - \int \log \left( t ^{-1} \gamma  \right)  \, d \nu _{\tau, t \lambda} + \int \log \left(  \gamma _{k(\cdot)} ^2 / (t ^2 |I _{k(\cdot)} | ) \right) d \nu _{\tau, t \lambda} \quad \geqslant \quad - \int \log \left( t ^{-1} \gamma  \right)  \, d \nu _{\tau, t \lambda} ,
\end{eqnarray*}
since $ \gamma / \lambda = \gamma ^2 \tau ' $ and $ t \leqslant \min  \{ \frac{\gamma _0}{\sqrt{| I _0 |}} , \frac{\gamma _1}{\sqrt{| I _1 |}}  \} $.
\end{proof}

\subsection{Sufficient condition through transversality}
In \cite{Tsujii01} Tsujii introduced $( \varepsilon, \delta )$-transversality to study the $L ^2 $-absolute continuity of a sort of SRB-measures that corresponds the distribution of $ \Theta $ of this note.
This relation is pointed out in \cite{Baranski14}.
In order to check the condition of Theorem \ref{thm:example2} we develop his method.
Thus we assume the setting of that theorem.
In particular, let $ \ell \geqslant 2 $, let $ \tau $ be piecewise linear and $ \lambda := ( \tau ' ) ^{-\theta} $ for a $ \theta \in ( 0 , 1 ) $.
Note, however, that $ g $ does not need to be the specific functions as in that theorem until we require it explicitly.

As already proved in Lemma \ref{lem:bernoulli} the measure $ \nu _{\tau , \lambda} $ is nothing but the Lebesgue measure $ m $ on $ [0,1] $.
Thus $ \nu _{\tau , \lambda } = \nu _{\bp _c} $ for the critical probability vector $ \bp _c := ( | I _0 | , \ldots , | I _{\ell -1} | ) $.

Observe that $ \lambda (x) = \lambda _{k(x)} $ and $ \gamma (x) = ( \tau ' \lambda ) ^{-1} (x) = \gamma _{k(x)} $, where $ \lambda _i := | I _i | ^\theta $ and $ \gamma _i := | I _i | ^{1 - \theta } $ for $ i = 0, \ldots , \ell - 1 $.
In particular, $ \gamma ( \rho _{k(\xi)} (x) ) = \gamma ( \xi ) $.
Thus $ \Theta ( \xi ,x ) = \sum _{n=1} ^\infty \gamma ^n ( \xi ) \, g' ( \rho _{[\xi] _n} (x) ) $ is differentiable so that the following consideration makes sense.

Let $ \varepsilon , \delta > 0 $ and $ \xi , \eta \in [0,1] $.
We say that $ \Theta ( \xi , \cdot ) $ and $ \Theta ( \eta , \cdot ) $ are $( \varepsilon, \delta )$-transversal, if for each $ x \in [0,1] $ holds either
\[
	\left| \Theta ( \xi , x ) - \Theta ( \eta , x ) \right| > \varepsilon	\quad \mbox{ or } \quad \left| \frac{\partial \Theta }{\partial x} ( \xi , x ) - \frac{\partial \Theta }{\partial x}( \eta , x ) \right| > \delta .
\]

Observe that the distribution of $ (\xi , x  ) \mapsto  ( x , \Theta ( \xi , x ) ) $ under $ \nu _\bp \otimes \nu _\bp $ is an invariant ergodic measure of the dynamical system $ f : [0,1] \times \mathbb{R} \rightarrow [0,1] \times \mathbb{R} $ defined by
\[
	f ( x , y ) := \left( \tau (x) , \, \gamma (x) \cdot \left( y -  g' (x) \right) \right) .
\]
Indeed, we have the invariance
\[
	f ( x , \Theta ( \xi , x ) ) = \left( \tau (x) , \Theta \circ B ^{-1} (\xi , x )  \right) .
\]

Let $ \zeta _\bp :=  \nu _\bp \otimes \nu _\bp  \circ ( \mathrm{Id} , \Theta )^{-1} $ and $ \zeta _{\bp , x} := \nu _\bp \circ \Theta ( \cdot , x ) ^{-1} $ so that $ \zeta _\bp = \int \delta _{\{ x \}} \times \zeta _{\bp , x} \, d \nu _\bp ( x ) $.
Slightly abusing the notation, we define
\[
	f \zeta _{\bp, x} (A) :=   \zeta _{\bp, x} \circ f ( x , \cdot ) ^{-1} .
\]
Furthermore, let $ m $ denote (also) the Lebesgue measure on $ \mathbb{R} $.
For $ i , j \in \{ 0, \ldots , \ell -1 \} $ and $ r > 0 $ we define
\begin{eqnarray*}
	I _\bp ( r )	& := &	\frac{1}{r ^2} \int _{[0,1]} \| \zeta _{\bp ,x} \| _r ^2 \, d \nu _\bp (x) , \mbox{ and} \\
	I _\bp ( r ; i , j )	& := &	 \frac{1}{r ^2} \int _{[0,1]}  \left( f  \zeta _{\bp , \rho _i (x)} , f \zeta _{\bp , \rho _j (x)}  \right) _r \,d  \nu _\bp (x) ,
\end{eqnarray*}
where
\[
	( \nu , \tilde{\nu} ) _r := \int _{\mathbb{R}} \nu ( B _r ( z ) ) \, \tilde{\nu} ( B _r ( z ) ) \, d m ( z ) 
\]
and $ \| \nu \| _r ^2 := ( \nu , \nu ) _r $ for $ \nu , \tilde{\nu} \in \mathcal{P} ( \mathbb{R} ) $ and $ r > 0 $.

\begin{proposition}	\label{prop:selfsimilar_measure}
We have
\[
	\zeta _{\bp, x} ( A ) = \sum _{i = 0} ^{\ell -1} p _i \, f \zeta _{\bp, \rho _i (x)}  ( A )
\]
for all $ A \in \mathcal{B} ( \mathbb{R} ) $.
\end{proposition}

\begin{proof}
For $ A \in \mathcal{B} ( \mathbb{R} ) $ we have
\begin{eqnarray*}
	\zeta _{\bp, x} ( A )	& = &	\nu _\bp ( \{ \xi \in [0,1] : \Theta ( \xi , x ) \in A \} ) \\
		& = &	\nu _\bp \left( \left\{ \xi \in [0,1]: \Theta \circ B ^{-1} \left( \tau ( \xi ) , \rho _{k (\xi)} (x) \right) \in A \right\} \right) \\
		& = &	\sum _{i = 0} ^{\ell -1} \nu _\bp \left( \left\{ \xi \in I _i : \Theta \circ B ^{-1} \left( \tau ( \xi ) , \rho _i ( x ) \right) \in A \right\} \right) \\
		& = &	\sum _{i = 0} ^{\ell -1} p _i  \, \nu _\bp ( \{ \xi \in [0,1] : \Theta \circ B ^{-1} ( \xi , \rho _i (x) ) \in A \} ) \\
		& = &	 \sum _{i = 0} ^{\ell -1}  p _i \, \nu _\bp \left( \{ \xi \in [0,1] :  f ( \rho _i (x) , \Theta  ( \xi , \rho _i (x) ) ) \in \{ x \} \times A \} \right) \\
		& = &	\sum _{i = 0} ^{\ell -1} p _i \, \zeta _{\bp, \rho _i (x)} ( \{ y \in \mathbb{R} :  f ( \rho _i (x) ,y ) \in \{ x \} \times A \} ) \\
		& = &	\sum _{i = 0} ^{\ell -1}  p _i \, f  \zeta _{\bp, \rho _i (x)} ( A ) .
\end{eqnarray*}
\end{proof}

\begin{proposition}	\label{prop:transversal_ineq}
Let $ 0 \leqslant i < j \leqslant \ell -1 $.
If $ \Theta ( \xi , \cdot ) $ and $ \Theta ( \eta , \cdot ) $ are $ ( \varepsilon , \delta ) $-transversal for all $ ( \xi , \eta ) \in I_i \times I _j $, then we have
\[
	I _{\bp _c} ( r ; i ,j ) \leqslant 8 \delta ^{-1}  \max \{ 4 \alpha / \varepsilon , 1 \} 
\]
for all $ r \in ( 0 , \varepsilon / 4 ) $, where $ \alpha :=  \| \frac{\partial \Theta}{\partial x} \| _\infty$.
\end{proposition}

\begin{proof}
Observe $  I _{\bp _c} ( r ; i ,j )  = \frac{1}{r ^2} \int  \left( f  \zeta _{\bp , \rho _i (x)} , f \zeta _{\bp , \rho _j (x)}  \right) _r \,d  m (x) $.
As the integral part is bounded by $ 8 \delta ^{-1} r ^2  \max \{ 4 \alpha / \varepsilon , 1 \} $ analogously to \cite[Proposition 6]{Tsujii01}, the claim follows.
\end{proof}

\begin{proposition}	\label{prop:Tsujii_ident}
We have
\[
	\| f \zeta _{\bp , \rho _i (x)} \| _r ^2 =  \gamma ( \rho _i (x) 
) \cdot \| \zeta _{\bp, \rho _i (x)} \| _{r / \gamma ( \rho _i (x) )} ^2
\]
for all $ r > 0 $, $ x \in [0,1] $ and $i \in \{ 0, \ldots , \ell - 1\} $.
\end{proposition}

\begin{proof}
As
\[
	f \zeta _{\bp, \rho _i (x)} \left( B _r (z ) \right) = \zeta _{\bp, \rho _i ( x )} \left( B _{r / \gamma ( \rho _i (x) 
)} \left( \frac{z}{\gamma ( \rho _i (x) 
)} + g ' \left( \rho _i (x) \right) \right) \right) ,
\]
the claim follows by the substitution formula of the integral.
\end{proof}

\begin{proposition}	\label{prop:key_ineq}
Suppose that $ \Theta ( \xi , \cdot ) $ and $ \Theta ( \eta , \cdot ) $ are $ ( \varepsilon , \delta ) $-transversal for all $ ( \xi, \eta ) \in I _i \times I _j $ and $ 0 \leqslant i < j \leqslant \ell -1 $.
Then we have 
\[
	I _{\bp _c} ( r  ) \leqslant \beta \, I _{\bp _c} \left( \frac{r }{ \min _{[0,1]}  \gamma  } \right) +  8 \delta ^{-1}    \max \{ 4 \alpha / \varepsilon , 1 \} 
\]
for all $ k \in \mathbb{N} _0 $ and $ r \in \left( 0 , \varepsilon  / 4 \right) $, where $ \alpha :=  \| \frac{\partial \Theta}{\partial x} \| _\infty$ and $ \beta := \frac{\max _i \frac{| I _i | ^2 }{\lambda _i}}{(\min _{[0,1]}  \gamma ) ^2 } $.
\end{proposition}

\begin{proof}
Let $ \bp = \bp _c $.
By Propositions \ref{prop:selfsimilar_measure} and \ref{prop:transversal_ineq} we have
\begin{eqnarray*}
	I _\bp ( r )	& = &	\sum _{i,j =0} ^{\ell -1} p _i p _j \, I _\bp ( r ; i , j ) \\
	& \leqslant &	\sum _{i=0} ^{\ell -1} p _i p _i \, I _\bp ( r ; i , i ) + \sum _{i \neq j} p _i p _j \, I _\bp ( r ; i , j ) \\
		& \leqslant &	\sum _{i=0} ^{\ell -1} p _i p _i \, I _\bp ( r ; i , i ) +  8 \delta ^{-1}   \max \{ 4 \alpha / \varepsilon , 1 \}  .
\end{eqnarray*}
Furthermore, we have
\begin{eqnarray*}
	\sum _{i=0} ^{\ell -1} p _i p _i \, I _\bp ( r ; i , i ) 	&  = &	\sum _{i=0} ^{\ell -1} \frac{p _i ^2 }{r ^2}  \int \gamma ( \rho _i (x) 
) \cdot \| \zeta _{\bp, \rho _i (x)} \| _{r / \gamma ( \rho _i (x) 
)} ^2 \, d m (x) \\
	&  = &	\sum _{i=0} ^{\ell -1} \frac{p _i ^2}{r ^2}  \int _{I _i} \gamma (x) \cdot \| \zeta _{\bp, x} \| _{r / \gamma (x)} ^2 \cdot \tau ' (x) \, d m (x) \\
		&  \leqslant &	\sum _{i=0} ^{\ell -1} \frac{\max _i \frac{| I _i | ^2 }{\lambda _i}}{(\min _{[0,1]}  \gamma ) ^2 }  \left( \frac{\min _{[0,1]} \gamma}{r} \right) ^2  \int _{I _i} \| \zeta _{\bp, x} \| _{r /( \min _{[0,1]} \gamma )} ^2 d m (x) \\
	& \leqslant &	\beta \, I _\bp \left( \frac{r }{ \min _{[0,1]} \gamma } \right)
\end{eqnarray*}
by Proposition \ref{prop:Tsujii_ident}.
Thus the claim is proved.
\end{proof}

\begin{lemma}	\label{lem:abs_one_tsujii}
Under the assumption of Proposition \ref{prop:key_ineq}, if $ \beta < 1 $, then $ \liminf _{r \rightarrow 0} I _{\bp _c} (r) < \infty $.
In this case, the distribution of $ \Theta ( \cdot , x ) $ under $ \nu _{\bp _c} $ has a conditional $ L ^2 $-densities $ h _x $ w.r.t. $ m $ for $ \nu _{\bp _c} $-a.a. $ x $ such that $ \int \| h _x \| _2 ^2 \, d \nu _{\bp _c} (x) < \infty $.

In particular, the distribution of $ \Theta ( \cdot , x ) $ under $ \nu _{\bp _c} $ has Hausdorff dimension $ 1 $ for $ \nu _{\bp _c} $-a.a. $ x  $.
\end{lemma}

\begin{proof}
For example, we can take a sequence $ r _k := \varepsilon \, ( \min _{[0,1]} \gamma ) ^k / 8 $ so that
\[
	I _\bp ( r _k ) \leqslant \beta ^k I ( \varepsilon / 8 ) + \frac{8 \delta ^{-1}    \max \{ 4 \alpha / \varepsilon , 1 \}}{1 - \beta} .
\]
Since we have thus $ \liminf _{r \rightarrow 0} I _\bp (r) < \infty $, the remaining part of the fist claim can be concluded analogously to the proof of \cite[Corollary 5]{Tsujii01}.

Finally, the absolute continuity to Lebesgue measure implies the full Hausdorff dimension due to Lebesgue differentiation theorem.
\end{proof}

Now, we consider the transversality in case $ g ( x ) = \cos ( 2 \pi x ) $.
We extend an idea of \cite{Baranski14} to find explicit parameters.
Observe we have
\begin{eqnarray*}
	\Theta ( \xi ,x )	 & = &	 - 2 \pi \sum _{n=1} ^\infty \gamma ^n ( \xi ) \, \sin \left( 2 \pi \rho _{[\xi] _n} (x) \right)	, \mbox{ and}\\
	 \frac{\partial \Theta}{\partial x} ( \xi ,x )	& = &	( 2 \pi ) ^2 \sum _{n=1} ^\infty \left( \frac{\gamma}{\tau '} \right) ^n ( \xi ) \, \cos \left( 2 \pi \rho _{[\xi] _n} (x) \right) .
\end{eqnarray*}

\begin{lemma}	\label{lem:transversal_cos}
Suppose that $ g (x) = \cos(2 \pi x ) $.
If
\[
	G ( \min  \gamma , \max \gamma ) + G ( \min \frac{\gamma}{\tau '} , \max \frac{\gamma}{\tau '} ) < \delta _0 ,
\]
then there is a $ \delta > 0 $ such that $ \Theta ( \xi , \cdot ) $ and $ \Theta ( \eta , \cdot ) $ are $ ( \delta , \delta ) $-transversal for all $ ( \xi, \eta ) \in I _i \times I _j $ and $ 0 \leqslant i < j \leqslant \ell -1 $, where $ \delta _0 $ and $ G $ are defined in \eqref{eq:G}.
\end{lemma}
\begin{proof}
Let $ \xi \in I _i $ and $ \eta \in I _ j $ for some $ 0 \leqslant i < j \leqslant \ell -1 $.
From
\begin{eqnarray*}
		\gamma _i \sin ( 2 \pi \rho _i (x) ) - \gamma _j  \sin ( 2 \pi \rho _j (x) ) 
	& = &	( \gamma _i   - \gamma _j ) \sin \left( 2 \pi \frac{\rho _i (x) + \rho _j (x)}{2}  \right) \cos \left( 2 \pi \frac{\rho _i (x) - \rho _j (x)}{2}  \right) \\
	&& \qquad	+ ( \gamma _i + \gamma _j ) \cos \left( 2 \pi \frac{\rho _i (x) + \rho _j (x)}{2}  \right) \sin \left( 2 \pi \frac{\rho _i (x) - \rho _j (x)}{2}  \right)
\end{eqnarray*}
follows
\begin{eqnarray*}
	( 2 \pi ) ^ {-1} | \Theta ( \xi , x ) - \Theta ( \eta , x ) | 
	& \geqslant &	| \gamma _i \sin ( 2 \pi \rho _i (x) ) - \gamma _j  \sin ( 2 \pi \rho _j (x) ) | - 2 \sum _{n=2} ^\infty ( \max  \gamma ) ^n \\
	& \geqslant &	 ( \gamma _i  + \gamma _j  ) \left| \sin \left( 2 \pi \frac{\rho _i (x) - \rho _j (x)}{2}  \right) \right|  \left| \cos \left( 2 \pi \frac{\rho _i (x) + \rho _j (x)}{2}  \right) \right|  \\
	&&	\quad 	- \frac{2 ( \max  \gamma  ) ^2}{1 - ( \max  \gamma  ) } - | \gamma _i  - \gamma _j  | \\
		& \geqslant &	 2 ( \min \gamma ) \left| \sin \left( 2 \pi \frac{\rho _i (x) - \rho _j (x)}{2}  \right) \right| \left| \cos \left( 2 \pi \frac{\rho _i (x) + \rho _j (x)}{2}  \right) \right| \\
		&&	\qquad	- \frac{2 ( \max \gamma  ) ^2}{1 -( \max \gamma  ) } - ( \max  \gamma  - \min \gamma  ) .
\end{eqnarray*}
Furthermore, from
\begin{eqnarray*}
	&&	\gamma _i  \, \rho _i ' (x) \, \cos ( 2 \pi \rho _i  (x) ) - \gamma _j \, \rho _j ' (x) \,  \cos ( 2 \pi \rho _j (x) ) \\
	& = &	\left( \gamma _i  \, \rho _i ' (x)  - \gamma _j  \, \rho _j ' (x) \right) \cos \left( 2 \pi \frac{\rho _i (x) + \rho _j (x)}{2}  \right) \cos \left( 2 \pi \frac{\rho _i (x) - \rho _j (x)}{2}  \right) \\
	&& \qquad	+ \left( \gamma _i  \, \rho _i ' (x) + \gamma _j  \, \rho _j ' (x)  \right) \sin \left( 2 \pi \frac{\rho _i (x) + \rho _j (x)}{2}  \right) \sin \left( 2 \pi \frac{\rho _i (x) - \rho _j (x)}{2}  \right)
\end{eqnarray*}
follows
\begin{eqnarray*}
	&  &	( 2 \pi ) ^{-2} | \Theta ' ( \xi , x ) - \Theta ' ( \eta , x ) | \\
	& \geqslant &	\left| \gamma _i \, \rho _i ' (x) \, \cos ( 2 \pi \rho _i (x) ) - \gamma _j  \, \rho _j ' (x) \, \cos ( 2 \pi \rho _j (x) ) \right|  - 2 \sum _{n=2} ^\infty \left( \max \frac{\gamma}{\tau '} \right) ^n     \\
	& \geqslant &	\left( \gamma _i \, \rho _i ' (x)  + \gamma _j  \, \rho _j ' (x)  \right) \left| \sin \left( 2 \pi \frac{\rho _i (x) - \rho _j (x)}{2}  \right) \right|  \left| \sin \left( 2 \pi \frac{\rho _i (x) + \rho _j (x)}{2}  \right) \right| \\
	&&	\qquad - \frac{2 \left( \max  \frac{\gamma}{\tau '}  \right) ^2}{1 - \left(  \max  \frac{\gamma}{\tau '}  \right)} - \left| \gamma _i (x) \, \rho _i ' (x)  -   \gamma _j (x) \, \rho _j ' (x)  \right|      \\
			& \geqslant &	2 \left( \min \frac{\gamma}{\tau '} \right) \left| \sin \left( 2 \pi \frac{\rho _i (x) - \rho _j (x)}{2}  \right) \right|  \left| \sin \left( 2 \pi \frac{\rho _i (x) + \rho _j (x)}{2}  \right) \right| \\
	&&	\qquad - \frac{2 \left( \max  \frac{\gamma}{\tau '}  \right) ^2}{1 - \left(  \max  \frac{\gamma}{\tau '}  \right)} - \left( \max \frac{\gamma}{\tau '}  -  \min \frac{\gamma}{\tau '} \right)   .
\end{eqnarray*}
By squaring and summing up both equalities above we obtain
\begin{eqnarray*}
	& &	\left( \frac{( 2 \pi ) ^ {-1} | \Theta ( \xi , x ) - \Theta ( \eta , x ) | + \frac{2 ( \max \gamma  ) ^2}{1 -( \max \gamma  ) } + ( \max  \gamma  - \min \gamma  )}{2 (\min  \gamma ) } \right) ^2 \\
	& &  + \left( \frac{ ( 2 \pi ) ^{-2} | \Theta ' ( \xi , x ) - \Theta ' ( \eta , x ) | + \frac{2 \left( \max  \frac{\gamma}{\tau '}  \right) ^2}{1 - \left(  \max  \frac{\gamma}{\tau '}  \right)} + \left( \max \frac{\gamma}{\tau '}  -  \min \frac{\gamma}{\tau '} \right)  }{2 \left(  \min  \frac{\gamma}{\tau '} \right) } \right) ^2 \\
	& \geqslant &	  \sin ^2 \left( 2 \pi \frac{\rho _i (x) - \rho _j (x)}{2}  \right) 	\quad \geqslant \delta _0 .
\end{eqnarray*}

Suppose that there is a $ ( \omega , \eta ) \in I _i \times I _j $ with $ i \neq j $ such that $ \Theta ( \xi , \cdot ) $ and $ \Theta ( \eta , \cdot ) $ are not $ (\delta, \delta ) $-transversal for any $ \delta > 0 $.
Then the above inequality implies
\[
	\left( \frac{\frac{2 ( \max \gamma  ) ^2}{1 -( \max \gamma  ) } + ( \max  \gamma  - \min \gamma  )}{2 (\min  \gamma ) } \right) ^2   + \left( \frac{  \frac{2 \left( \max  \frac{\gamma}{\tau '}  \right) ^2}{1 - \left(  \max  \frac{\gamma}{\tau '}  \right)} + \left( \max \frac{\gamma}{\tau '}  -  \min \frac{\gamma}{\tau '} \right)  }{2 \left(  \min  \frac{\gamma}{\tau '} \right) } \right) ^2 \geqslant \delta _0
\]
As this contradicts the assumption of the lemma, this finishes the proof.
\end{proof}

\begin{proof}[Proof of Theorem \ref{thm:example2}]
In Lemma \ref{lem:bernoulli} we proved $ s ( \tau , \lambda ) = 2 - \theta $.
Thus the claim follows from Theorem \ref{thm:main_theorem} together with Lemmas \ref{lem:bernoulli}, \ref{lem:abs_one_tsujii} and \ref{lem:transversal_cos} by the insertion of $ \gamma = ( \tau ' ) ^\theta $ and $ \gamma / \tau ' = ( \tau ' ) ^{\theta - 1} $.
\end{proof}

\bibliography{dynamical}
\bibliographystyle{abbrv}

\end{document}